\batchmode
\makeatletter
\def\input@path{{/home/andrew/Documents/}}
\makeatother
\documentclass[oneside,english]{amsart}
\usepackage[T1]{fontenc}
\usepackage[latin9]{inputenc}
\usepackage{geometry}
\usepackage{verbatim}
\usepackage{amstext}
\usepackage{amsthm}
\usepackage{amssymb}
\usepackage{stackrel}

\makeatletter
\numberwithin{equation}{section}
\numberwithin{figure}{section}
\theoremstyle{plain}
\newtheorem{thm}{\protect\theoremname}
\theoremstyle{plain}
\newtheorem{assumption}[thm]{\protect\assumptionname}
\theoremstyle{remark}
\newtheorem*{rem*}{\protect\remarkname}
\theoremstyle{definition}
\newtheorem{defn}[thm]{\protect\definitionname}
\theoremstyle{plain}
\newtheorem{fact}[thm]{\protect\factname}
\theoremstyle{remark}
\newtheorem{rem}[thm]{\protect\remarkname}
\theoremstyle{plain}
\newtheorem{prop}[thm]{\protect\propositionname}
\theoremstyle{definition}
\newtheorem{example}[thm]{\protect\examplename}
\theoremstyle{plain}
\newtheorem{lem}[thm]{\protect\lemmaname}
\theoremstyle{plain}
\newtheorem{cor}[thm]{\protect\corollaryname}
\theoremstyle{plain}
\newtheorem*{thm*}{\protect\theoremname}

\usepackage{hyperref}

\makeatother

\usepackage{babel}
\providecommand{\assumptionname}{Assumption}
\providecommand{\corollaryname}{Corollary}
\providecommand{\definitionname}{Definition}
\providecommand{\examplename}{Example}
\providecommand{\factname}{Fact}
\providecommand{\lemmaname}{Lemma}
\providecommand{\propositionname}{Proposition}
\providecommand{\remarkname}{Remark}
\providecommand{\theoremname}{Theorem}

\begin{document}
\title{Gradient flow structure for some nonlocal diffusion equations}
\author{Andrew Warren}
\address{IHÉS and the University of British Columbia}
\email{warren@ihes.fr; awarren@math.ubc.ca}
\begin{abstract}
We study ``nonlocal diffusion equations'' of the form 
\[
\partial_{t}\frac{d\rho_{t}}{d\pi}(x)+\int_{X}\left(\frac{d\rho_{t}}{d\pi}(x)-\frac{d\rho_{t}}{d\pi}(y)\right)\eta(x,y)d\pi(y)=0\qquad(\dagger)
\]
where $X$ is either $\mathbb{R}^{d}$ or $\mathbb{T}^{d}$, $\pi$
is a probability distribution on $X$, and $\eta(x,y)$ is a ``transition
kernel'' which may be singular as $x\rightarrow y$. For a suitable
notion of weak solutions which we discuss below, we show that solutions
to these nonlocal diffusion equations can be interpreted as gradient
flows of the relative entropy with respect to a certain nonlocal Wasserstein-type
metric defined in terms of $\eta$ and $\pi$. These ``nonlocal Wasserstein
metrics'' endow the space of probability measures on $X$ with a
formal Riemannian structure, thereby providing for us a nonlocal analogue
of the \emph{Otto calculus} originally developed  in the context
of the 2-Wasserstein metric. The class of equations $(\dagger)$ includes
a family of ``nonlocal Fokker-Planck equations'', which are thus
identified as nonlocal Wasserstein gradient flows of the relative
entropy, analogously with the usual Fokker-Planck equation and the
$W_{2}$ metric.

The gradient flow structure we provide allows us to deduce: existence
and uniqueness of solutions to ($\dagger$) in a suitable class of
weak solutions; stability of solutions in the sense of evolutionary
$\Gamma$-convergence, with respect to perturbations of initial condition,
reference measure $\pi$, and transition kernel $\eta$; sufficient
conditions for exponential convergence to equilibrium, in terms of
a nonlocal analogue of the log-Sobolev inequality; as well as the
consistency of a finite-volume-type spatial discretization scheme
in the $\mathbb{T}^{d}$ case.
\end{abstract}

\maketitle
\tableofcontents{}

Nonlocal equations of the form 
\[
\partial_{t}u_{t}(x)+\int_{X}\left(u_{t}(x)-u_{t}(y)\right)\eta(x,y)d\pi(y)=0
\]
 arise in physical models exhibiting anomalous diffusion, and can
be formally understood as the Kolmogorov equation for a pure-jump
process with exponentially distributed waiting times, and jump transition
kernel $\eta(x,y)$ and stationary measure $\pi$. In this article,
by contrast, we take a variational and mass-transport-theoretic perspective:
we give conditions under which this equation can be viewed as the
gradient flow of the relative entropy $\mathcal{H}(\cdot\mid\pi)$,
with respect to a suitably chosen ``nonlocal Wasserstein'' metric.
Aside from being of conceptual interest, this gradient flow structure
turns out to furnish us with variational proofs of a range of structural
features of the evolution equation, including existence, uniqueness,
stability with respect to choice of $\eta$ and $\pi$, and consistency
of a finite volume approximation scheme.

The paper is structured as follows. In Section \ref{sec:Introduction}
immediately below, we first give a formal definition of the nonlocal
Wasserstein metric and formally compute the gradient flow of the relative
entropy, after which we summarize related work and situate our own
contributions within the existing literature. Section \ref{sec:Preliminaries}
gives the technical setup for the family of nonlocal Wasserstein metrics
we employ, and introduces the measure-valued notion of weak solutions
we use for equation ($\dagger$). In Section \ref{sec:Compactness}
we provide a compactness result (Proposition \ref{prop:compactness-nce})
which is the key ingredient for our subsequent existence and stability
results. Section \ref{sec:Identifying} gives conditions under which
weak solutions to equation ($\dagger$) can be identified with metric
gradient flows (in the sense of satisfying an entropy dissipation
inequality) in the given nonlocal Wasserstein metric (Proposition
\ref{prop:weak-soln-vs-gradient-flow}), and establish uniqueness
(Proposition \ref{prop:uniqueness}) and stability with respect to
$\eta$ and $\pi$ (Proposition \ref{prop:stability}) of said gradient
flows. Existence of gradient flows (Theorem \ref{thm:existence})
is established in Section \ref{sec:Existence-and-Discrete-to-Continuum},
by way of the consistency of a finite volume scheme (Proposition \ref{prop:eta-properties-sufficient}).
Lastly, in Section \ref{sec:contractivity} we briefly discuss conditions
for exponential contraction of solutions towards the stationary distribution
$\pi$ \emph{à la }logarithmic Sobolev inequalities; while in Section
\ref{sec:nonlocal-Dirichlet} we connect the contents of the present
article to another interpretation of equation ($\dagger$), namely
as the semigroup generated by a certain purely nonlocal Dirichlet
form. 

\section{Introduction\protect\label{sec:Introduction}}

\subsection{Formal identification of the gradient flow structure}

As motivation, we give a formal argument (along the lines of the Otto
calculus introduced in \cite{otto2000generalization,otto2001geometry})
for the claimed gradient flow structure for $(\dagger)$; a nearly
identical computation for the case where $\pi$ is the Lebesgue measure
can be found in \cite{erbar2014gradient}, but we include this argument
below in case it aids the reader in navigating the present article.
The nonlocal Wasserstein metric we consider may be formally written
as follows. Denoting $\hat{\psi}(x,y):=\frac{\psi(x)-\psi(y)}{\log\psi(x)-\log\psi(y)}$,
we define
\[
\mathcal{W}^{2}(\bar{\rho}_{0},\bar{\rho}_{1}):=\inf_{(\rho_{t},v_{t})_{t\in[0,1]}\in\mathcal{CE}(\bar{\rho}_{0},\bar{\rho}_{1})}\int_{0}^{1}\frac{1}{2}\int_{X}\int_{X}(v_{t}(x,y))^{2}\widehat{\frac{d\rho}{d\pi}}(x,y)\eta(x,y)d\pi(x)d\pi(y)
\]
where for each $t$, $v_{t}:X\times X\rightarrow\mathbb{R}$ and
\[
\mathcal{CE}(\bar{\rho}_{0},\bar{\rho}_{1}):=\left\{ (\rho_{t},v_{t})_{t\in[0,1]}:\partial_{t}\frac{d\rho_{t}}{d\pi}(x)+\int_{X}v_{t}(x,y)\widehat{\frac{d\rho}{d\pi}}(x,y)\eta(x,y)d\pi(y)=0;\bar{\rho}_{0}=\rho_{0},\bar{\rho}_{1}=\rho_{1}\right\} 
\]
denotes the set of solutions to the ``nonlocal continuity equation''
with initial and terminal conditions $\bar{\rho}_{0}$ and $\bar{\rho}_{1}$.
By analogy with the Benamou-Brenier formula for the $W_{2}$ distance
\cite{benamou2000computational}, we think of ``vector fields''
$v(x,y)$ as tangent vectors for the density $\frac{d\rho}{d\pi}$;
equipping these tangent vectors with the inner product 
\[
\langle v,w\rangle_{\rho}:=\frac{1}{2}\int_{X}\int_{X}(v(x,y)w(x,y)\widehat{\frac{d\rho}{d\pi}}(x,y)\eta(x,y)d\pi(x)d\pi(y)
\]
which we think of as a formal Riemannian metric on $\mathcal{P}(X)$,
with $\mathcal{CE}(\bar{\rho}_{0},\bar{\rho}_{1})$ telling us which
paths in $\mathcal{P}(X)$ are candidates for being ``absolutely
continuous''. The $\mathcal{W}$ distance is then identified as the
formal geodesic distance corresponding to $\langle v,w\rangle_{\rho}$:
\[
\mathcal{W}^{2}(\bar{\rho}_{0},\bar{\rho}_{1})=\inf_{(\rho_{t},v_{t})_{t\in[0,1]}\in\mathcal{CE}(\bar{\rho}_{0},\bar{\rho}_{1})}\int_{0}^{1}\langle v_{t},v_{t}\rangle_{\rho_{t}}dt.
\]

Then, to identify the gradient flow of the relative entropy $\mathcal{H}(\rho\mid\pi)=\int\frac{d\rho}{d\pi}\log\frac{d\rho}{d\pi}d\pi$
with respect to the ``Riemannian metric'' $\mathcal{W}$, we perform
the following formal computation. Given any absolutely continuous
$\rho_{t}:[0,1]\rightarrow\mathcal{P}(X)$, 
\begin{align*}
\frac{d}{dt}\mathcal{H}(\rho_{t}\mid\pi) & =\int_{X}\left(1+\log\frac{d\rho_{t}}{d\pi}(x)\right)\partial_{t}\frac{d\rho_{t}}{d\pi}(x)d\pi(x)\\
 & =-\int_{X}\left(1+\log\frac{d\rho_{t}}{d\pi}(x)\right)\left(\int_{X}v_{t}(x,y)\widehat{\frac{d\rho}{d\pi}}(x,y)\eta(x,y)d\pi(y)\right)d\pi(x).
\end{align*}
Assume that $v_{t}(x,y)=-v_{t}(y,x)$. Then we have, since $\eta(x,y)=\eta(y,x)$,
and $\widehat{\frac{d\rho}{d\pi}}(x,y)=\widehat{\frac{d\rho}{d\pi}}(y,x)$,
that
\begin{multline*}
-\int_{X}\int_{X}\left(1+\log\frac{d\rho_{t}}{d\pi}(x)\right)v_{t}(x,y)\widehat{\frac{d\rho}{d\pi}}(x,y)\eta(x,y)d\pi(y)d\pi(x)\\
\begin{aligned}\text{(change of variables)} & =-\int_{X}\int_{X}\left(1+\log\frac{d\rho_{t}}{d\pi}(y)\right)v_{t}(y,x)\widehat{\frac{d\rho}{d\pi}}(x,y)\eta(x,y)d\pi(x)d\pi(y)\\
 & =\int_{X}\int_{X}\left(1+\log\frac{d\rho_{t}}{d\pi}(y)\right)v_{t}(x,y)\widehat{\frac{d\rho}{d\pi}}(x,y)\eta(x,y)d\pi(x)d\pi(y)
\end{aligned}
\end{multline*}
and so 
\begin{align*}
\frac{d}{dt}\mathcal{H}(\rho_{t}\mid\pi) & =\frac{1}{2}\int_{X}\int_{X}\left(\log\frac{d\rho_{t}}{d\pi}(y)-\log\frac{d\rho_{t}}{d\pi}(x)\right)v_{t}(x,y)\widehat{\frac{d\rho}{d\pi}}(x,y)\eta(x,y)d\pi(x)d\pi(y)\\
 & =\left\langle \log\frac{d\rho_{t}}{d\pi}(y)-\log\frac{d\rho_{t}}{d\pi}(x),v_{t}(x,y)\right\rangle _{\rho_{t}}.
\end{align*}
On a Riemannian manifold, we have that $\frac{d}{dt}F(x_{t})=\langle\text{Grad}F(x_{t}),\dot{x}_{t}\rangle_{x_{t}}$,
and so this computation suggests that the ``gradient vector'' $\text{Grad}\mathcal{H}(\rho_{t}\mid\pi)$
with respect to the Riemannian structure $\langle v,w\rangle_{\rho}$
is none other than $\log\frac{d\rho_{t}}{d\pi}(y)-\log\frac{d\rho_{t}}{d\pi}(x)$.
At the same time, the gradient flow $\rho_{t}$ of $\mathcal{H}(\cdot\mid\pi)$
should have the property that the tangent vector equals the vector
$-\text{Grad}\mathcal{H}(\rho_{t}\mid\pi)$. In our case, the tangent
vector $v_{t}$ solves the ``continuity equation'' together with
$\rho_{t}$, and so the gradient flow of $\mathcal{H}(\cdot\mid\pi)$
should satisfy 
\[
\partial_{t}\frac{d\rho_{t}}{d\pi}(x)+\int_{X}\left(-\text{Grad}\mathcal{H}(\rho_{t}\mid\pi)(x,y)\right)\widehat{\frac{d\rho}{d\pi}}(x,y)\eta(x,y)d\pi(y)=0.
\]
Together with the preceding computation of $\text{Grad}\mathcal{H}(\rho_{t}\mid\pi)$,
and the fact that $\widehat{\frac{d\rho}{d\pi}}(x,y)=\frac{\frac{d\rho}{d\pi}(x)-\frac{d\rho}{d\pi}(y)}{\log\frac{d\rho}{d\pi}(x)-\log\frac{d\rho}{d\pi}(y)}$,
this yields
\[
\partial_{t}\frac{d\rho_{t}}{d\pi}(x)+\int\left(\frac{d\rho}{d\pi}(x)-\frac{d\rho}{d\pi}(y)\right)\eta(x,y)d\pi(y)=0
\]
as desired. In other words, equation $(\dagger)$ is indeed the gradient
flow of the relative entropy $\mathcal{H}(\cdot\mid\pi)$ with respect
to $\mathcal{W}$. Moreover, we formally compute that the gradient
flow of $\mathcal{H}(\cdot\mid\pi)$ has the dissipation rate 
\begin{align*}
\frac{d}{dt}\mathcal{H}(\rho_{t}\mid\pi) & =-\left\langle \text{Grad}\mathcal{H}(\rho_{t}\mid\pi),\text{Grad}\mathcal{H}(\rho_{t}\mid\pi)\right\rangle _{\rho_{t}}\\
 & =-\frac{1}{2}\int_{x}\int_{X}\left(\log\frac{d\rho_{t}}{d\pi}(x)-\log\frac{d\rho_{t}}{d\pi}(y)\right)^{2}\widehat{\frac{d\rho}{d\pi}}(x,y)\eta(x,y)d\pi(x)d\pi(y)
\end{align*}
which suggests that the ``entropy-entropy production inequality''
\[
\forall\rho\in\mathcal{P}(X),\;\mathcal{H}(\rho\mid\pi)\leq\frac{1}{2C}\cdot\int_{x}\int_{X}\left(\log\frac{d\rho}{d\pi}(x)-\log\frac{d\rho}{d\pi}(y)\right)^{2}\widehat{\frac{d\rho}{d\pi}}(x,y)\eta(x,y)d\pi(x)d\pi(y)
\]
should ensure exponential convergence of the gradient flow to equilibrium,
similarly to the log-Sobolev inequality for the Fokker-Planck equation. 

The preceding computation is merely formal for two reasons. Firstly,
similar to the situation with the Otto calculus for $W_{2}$, our
inner product $\langle v,w\rangle_{\rho}$ does not literally equip
the space of probability measures with a manifold structure. Accordingly,
we must define the metric $\mathcal{W}$ more carefully, and we will
compute the $\mathcal{W}$ gradient flow of $\mathcal{H}(\cdot\mid\pi)$
in the sense of \emph{curves of maximal slope}, which are a synthetic
notion of gradient flows suitable for general metrics, not just geodesic
metrics on Riemannian manifolds. Secondly, the integral $\int v(x,y)\eta(x,y)d\pi(y)$
may not be defined in a literal sense in the case where $\eta(x,y)$
becomes singular as $x\rightarrow y$, so we must employ an appropriate
notion of weak solutions for the equation $(\dagger)$. To this end,
we will show that curves of maximal slope for $\mathcal{H}(\cdot\mid\pi)$
with respect to $\mathcal{W}$ can be identified with certain solutions
to the nonlocal continuity equation, in a variant of the weak sense
already considered in \cite{erbar2014gradient}.

\subsection{Motivation and related work}

\subsubsection*{The fractional heat equation as a gradient flow of the entropy}

The present article has one particularly key prior work, namely that
of Erbar \cite{erbar2014gradient}. To briefly convey the motivation
for \cite{erbar2014gradient}, it has been known since the introduction
of the JKO scheme \cite{jordan1998variational} and the Otto calculus
for Wasserstein gradient flows \cite{otto2000generalization,otto2001geometry}
that the heat equation can be viewed as the 2-Wasserstein gradient
flow for the entropy $\mathcal{H}(\cdot\mid\text{Leb})$, while the
Fokker-Planck equation with potential $V$ can be viewed as the 2-Wasserstein
gradient flow of the \emph{relative entropy} $\mathcal{H}(\cdot\mid\pi)$
where $\pi$ is the probability measure with density proportional
to $e^{-V}$. But what about for fractional diffusion? While some
nonlocal equations can be viewed as $W_{2}$ gradient flows (in particular
certain diffusion equations with nonlocal interaction \cite{carrillo2003kinetic,carrillo2006contractions}),
the fractional heat equation does not have any apparent similar $W_{2}$
gradient flow interpretation. What Erbar showed in \cite{erbar2014gradient}
was that one could extend the \emph{ersatz} of the Wasserstein distance
on discrete Markov chains introduced by \cite{maas2011gradient,mielke2011gradient,chow2012fokker}
to the case where the state space is continuous and the Markov chain
is now a measurable transition kernel; and in the case where the state
space is chosen to be $\mathbb{R}^{d}$ and the transition kernel
is chosen to have density $c_{d,s}|x-y|^{-d-s}$ with $s\in(0,2)$,
one formally obtains a \emph{nonlocal Wasserstein metric} (NLW) on
the space of probability measures on $\mathbb{R}^{d}$, for which
there is an associated (formal) ``nonlocal Otto calculus'' according
to which the fractional heat equation is the gradient flow of the
entropy $\mathcal{H}(\cdot\mid\text{Leb})$. 

Erbar was then able to show rigorously that the fractional heat equation
is an NLW gradient flow of $\mathcal{H}(\cdot\mid\text{Leb})$ in
the synthetic sense of gradient flows on metric spaces elaborated
in \cite{ambrosio2008gradient}; specifically \cite{erbar2014gradient}
shows that the synthetic ``evolutional variational inequality''
(EVI) is satisfied, by employing the ``Eulerian'' strategy developed
in \cite{otto2005eulerian,daneri2008eulerian}. However the proof
ultimately relies on a number of ``extrinsic'' facts about the fractional
heat equation, and in particular requires heat kernel estimates for
the fractional Laplacian such as those given in \cite{chen2003heat}.
Additionally certain estimates which play a crucial role in \cite{erbar2014gradient}
seem to require the jump kernel to be translation invariant. As a
consequence, the analysis provided in \cite{erbar2014gradient} is
ultimately specific to the fractional heat equation. 

Nonetheless Erbar observes that the class of nonlocal Wasserstein
metrics described in \cite{erbar2014gradient} does allow for a number
of other nonlocal equations to be cast as NLW gradient flows, with
gradient flows of the relative entropy (viz. the topic of the present
article) singled out in particular as an obvious possibility. However
in the decade since \cite{erbar2014gradient} was published, no work
in this direction has appeared with the exception of \cite{peletier2020jump}
(which we discuss momentarily). Indeed, below we shall see that a
rigorous analysis of the relative entropy case appears to be technically
subtle, and seems to require an analytic approach substantially different
than the one given in \cite{erbar2014gradient}. 

Let us rephrase this open problem/motivation coming from \cite{erbar2014gradient}
as follows. From the transportation metric gradient flow point of
view, \emph{what is the right notion of ``nonlocal Fokker-Planck
equation''}? Indeed, since the heat equation is the $W_{2}$ gradient
flow of $\mathcal{H}(\cdot\mid\text{Leb})$, the fractional heat equation
is an NLW gradient flow of $\mathcal{H}(\cdot|\text{Leb})$, and the
usual Fokker-Planck equation is the $W_{2}$ gradient flow of $\mathcal{H}(\cdot\mid\pi)$,
then a ``nonlocal Fokker-Planck equation'' should be some equation
of the form 
\begin{equation}
\partial_{t}\rho_{t}=-(-\Delta)^{s/2}\rho_{t}+"\overline{\text{div}}_{\pi}^{s/2}"\rho_{t}\label{eq:intro-nonlocal-fokker-planck}
\end{equation}
where $"\overline{\text{div}}_{\pi}^{s/2}"$ ought to be some sort
of nonlocal divergence operator depending on the fractional exponent
$s$ and reference measure $\pi$, which is also an NLW gradient flow
of the relative entropy $\mathcal{H}(\cdot\mid\pi)$. We emphasize
that this is \emph{not} the same as the usual ``fractional Fokker-Planck
equation'' 
\[
\partial_{t}\rho_{t}=-(-\Delta)^{s/2}\rho_{t}+\text{div}(\rho_{t}\nabla V)
\]
where one keeps the local divergence term from the usual Fokker-Planck
equation and replaces the usual local diffusion term with a fractional
diffusion term. (Indeed, this latter equation can instead be appropriately
understood variationally in terms of a nonlocal-local splitting scheme
\cite{bowles2015weak}.) Our class of nonlocal diffusion equations
($\dagger$) (and its associated gradient flow structure) will be
shown to be flexible enough to include equations of the form (\ref{eq:intro-nonlocal-fokker-planck})
for an appropriate choice of NLW metric and $"\overline{\text{div}}_{\pi}^{s/2}"$.
Moreover we obtain estimates on the rate of convergence to the equilibrium
measure $\pi$ for solutions to (\ref{eq:intro-nonlocal-fokker-planck})
when $\pi$ satisfies a certain ``nonlocal log-Sobolev inequality'';
in the case where the underlying space is the flat torus $\mathbb{T}^{d}$,
we note that crude sufficient conditions for such an inequality to
hold have already been observed in the nonlocal Dirichlet forms literature
\cite{wang2014simple}, and the same sufficient conditions apply in
our case. In particular, this means that our $"\overline{\text{div}}_{\pi}^{s/2}"$
is chosen appropriately in the sense that $\pi$ actually is the equilibrium
measure for (\ref{eq:intro-nonlocal-fokker-planck}), as one would
expect by analogy with the usual Fokker-Planck equation (and which
is not true in general for the fractional Fokker-Planck equation \cite{gentil2008levy}).

\subsubsection*{Other evolution equations which are nonlocal Wasserstein gradient
flows}

Nevertheless, a number of nonlocal equations other than the fractional
heat equation have been shown to be nonlocal Wasserstein-type gradient
flows in the years since \cite{erbar2014gradient}, indeed the present
article draws on technical strategies from several such works. Erbar
showed in \cite{erbar2016boltzmann} that the spatially homogeneous
Boltzmann equation is the gradient flow of a certain NLW metric on
the space of distributions of velocities, where the metric incorporates
the collision kernel from the Boltzmann equation. This approach has
subsequently been extended to the linear Boltzmann equation \cite{basile2020gradient}
as well as the homogeneous Landau equation \cite{carrillo2022landau},
and indeed one can recover the Landau equation as the evolutionary
$\Gamma$-limit of the homogeneous Boltzmann equation in the ``grazing
collision limit'' with respect to these gradient flow structures
\cite{carrillo2022boltzmann}. We also mention \cite{esposito2021novel},
which employs a similar strategy to show that the one-dimensional
aggregation equation can be viewed as an NLW gradient flow, as well
as the article \cite{an2021gradient} which provides a formal argument
along similar lines for the isotropic Landau equation. In a different
direction, it has recently been shown that the strategy from \cite{erbar2014gradient}
can be lifted to random measures to study nonlocal evolution equations
on the space of configurations \cite{dello2024wasserstein,erbar2023optimal}.

At the discrete level, the identification of the heat flow of a Markov
chain with the Wasserstein-type gradient flow of the relative entropy
with respect to the equilibrium measure from \cite{maas2011gradient}
has been extended to the study of ``porous medium equations on Markov
chains'' in \cite{erbarmaas2014gradient}. One might guess that a
combination of the techniques from \cite{erbar2014gradient} and \cite{erbarmaas2014gradient}
could show that the \emph{fractional porous medium equation} is an
NLW gradient flow, and indeed this result is claimed in the literature
in \cite{ferreira2019minimizing} in the case where the domain is
$\mathbb{T}^{d}$. At the technical level, \cite{ferreira2019minimizing}
show that the basic setup of weak solutions to the nonlocal continuity
equation and definition of the NLW metric make sense on $\mathbb{T}^{d}$
as well (in \cite{erbar2014gradient} the domain is always $\mathbb{R}^{d}$).
This is convenient for our own work, as we work simultaneously in
$\mathbb{R}^{d}$ and $\mathbb{T}^{d}$, indeed our existence theorem
on $\mathbb{R}^{d}$ is established by bootstrapping from existence
on $\mathbb{T}^{d}$. The authors of \cite{ferreira2019minimizing}
also observe that the fractional porous medium equation can be formally
identified as an NLW gradient flow of a suitable Renyi entropy, and
prove that a curve of maximal slope for this Renyi entropy exists.
However we caution the reader that there is a technical omission in
\cite{ferreira2019minimizing}, namely it is not actually proved that
this curve of maximal slope is related at all to their notion of weak
solutions for the fractional porous medium equation. (We fully expect
that such an identification can be proved, just that it would be technically
laborious, as our own proof of the identification between these two
solution concepts in our setting attests.)

There have also been works which extend the setup of \cite{erbar2014gradient}
in such a way that one still has an ``action'' or ``kinetic energy''
defined on solutions to the nonlocal continuity equation, but the
least-action variational problem no longer induces a metric, but rather
a ``nonlocal Wasserstein cost'' which still somehow retains enough
structure to support a reasonable theory of gradient flows. In \cite{esposito2019nonlocal},
the authors provide a gradient flow structure for the \emph{nonlocal-nonlocal
interaction equation} by modifying the formal Riemannian structure
from \cite{erbar2014gradient} into an \emph{asymmetric} norm on the
tangent space. This endows the space of probability measures with
the formal structure of an infinite-dimensional Finsler manifold and
rigorously puts a quasimetric on an appropriate subspace of probability
measures. Accordingly the authors show that the machinery of gradient
flows on metric spaces from \cite{ambrosio2008gradient} can be extended
to their quasimetric setting, and prove that certain weak solutions
to the nonlocal-nonlocal interaction equation can be identified as
curves of maximal slope for a nonlocal interaction energy with respect
to their NLW quasimetric. We note that \cite{esposito2019nonlocal}
prove existence not by minimizing movements, as in \cite{ambrosio2008gradient},
but rather by using the gradient flow structure itself to establish
discrete-to-continuum evolutionary $\Gamma$-convergence result, then
noting existence at the discrete level follows from classical ODE
theory. This is the kernel of the existence proof that we must employ
below, since our gradient flow structure turns out to be slightly
too weak, in general, to employ the existence theory from \cite{ambrosio2008gradient}.
Additionally we mention the articles \cite{esposito2023graph,heinze2023nonlocal,esposito2023graphmulti}
which employ or extend the NLW quasimetric gradient flow structure
from \cite{esposito2019nonlocal}.

Likewise in \cite{peletier2020jump}, the authors provide a gradient
flow structure for a remarkably general class of laws of jump processes
on standard Borel spaces, including ones where the natural gradient
flow structure does not arise from a metric (or quasimetric) but rather
is realized as a \emph{generalized gradient system} in the sense of
\cite{mielke2016evolutionary}; this abstraction is motivated by applications
to reaction-diffusion systems and large deviations theory, see for
instance \cite{mielke2011gradient,liero2013gradient,mielke2014relation,liero2017microscopic}.
Their results include existence, uniqueness, stability with respect
to initial conditions, and identification between weak solutions to
our $(\dagger)$ and curves of maximal slope of the relative entropy,
in the case where $\sup_{x}\int_{X}\eta(x,y)d\pi(y)<\infty$. Unfortunately
this excludes the case we are particularly interested in, namely where
$\eta(x,y)$ is non-integrable in the neighborhood of $x=y$ but satisfies
$\sup_{x}\int(|x-y|^{2}\wedge1)\eta(x,y)d\pi(y)<\infty$ (the fractional
transition kernel $c_{d,s}|x-y|^{-d-s}$ is one such, so the setup
of \cite{peletier2020jump} even excludes the original results from
\cite{erbar2014gradient}). For these non-integrable transition kernels
we must proceed by a substantively different technical route than
\cite{peletier2020jump}. Nevertheless, even in the case where $\sup_{x}\int_{X}\eta(x,y)d\pi(y)<\infty$
certain of our results are new, in particular stability with respect
to $\eta$ and $\pi$ and our discrete-to-continuum results. On a
separate point, we note (as was observed in \cite{slepcev2023nonlocal}
in the Euclidean case) that assuming $\sup_{x}\int_{X}\eta(x,y)d\pi(y)<\infty$
puts a rather severe restriction on the topology of the resulting
NLW metric then puts on the space of probability measures: for \cite{peletier2020jump}
this topology is always at least as strong as the topology of setwise
convergence, so that particle approximations are not directly compatible
with the metric structure, for example. 

\subsubsection*{Connections with nonlocal Dirichlet forms}

Lastly we comment that there is an obvious ``parallel setting''
to our own, in the theory of nonlocal Dirichlet forms. Much like how
the heat equation can be simultaneously viewed as the $W_{2}$ gradient
flow of the relative entropy and also the $L^{2}$ gradient flow of
the Cheeger energy (in fact this identity holds on rather general
metric measure spaces \cite{ambrosio2014calculus}), equation ($\dagger$)
can also be formally viewed as being induced by semigroup generated
the nonlocal Dirichlet form 
\[
\mathcal{E}_{\eta,\pi}(f,g)=\frac{1}{2}\iint_{X\times X}(f(x)-f(y))(g(x)-g(y))\eta(x,y)d\pi(x)d\pi(y).
\]
however we note that our notion of weak solutions for ($\dagger$)
is not precisely the same as the one preferred in the nonlocal Dirichlet
forms literature. In Section \ref{sec:nonlocal-Dirichlet} we briefly
consider this direction, and show that for sufficiently nice choices
of initial condition and jump kernel $\eta$, the two weak solution
concepts for ($\dagger$) give rise to the same solutions (Proposition
\ref{prop:dirichlet-caloric-equivalent}). This suggests the possibility
that we can jointly exploit tools from optimal transport and Dirichlet
form theory in the study of equations like ($\dagger$), a direction
we shall consider in future work: indeed this strategy has been employed
in the local setting in \cite{ambrosio2014calculus} to study Wasserstein
gradient flows on low-regularity underlying spaces. In the meantime,
we simply note that our stability results have certain parallels in
the nonlocal Dirichlet forms literature, in particular \cite{barlow2009non}
gives a notion of convergence for a sequence of nonlocal Dirichlet
forms that ensures the convergence of the induced semigroups; and
that \cite{chen2013discrete} provides an analogous finite-volume
discrete-to-continuum result to our own for nonlocal Dirichlet forms,
albeit modulo a number of technical differences.

\section{Preliminaries\protect\label{sec:Preliminaries}}

In what follows, we take $X$ to be either $\mathbb{R}^{d}$ or the
flat torus $\mathbb{T}^{d}$. %
\begin{comment}
We focus on these two spaces largely for concreteness --- in particular,
the vast majority of our arguments will also generalize to more general
Euclidean periodic domains.
\end{comment}
{} We let $\mathcal{P}(X)$ denote the space of probability measures
on $X$, equipped with the narrow topology. 

\subsection{Graph structure, weight kernel, interpolation}

We equip $X$ with a graph structure, as follows.
\begin{assumption}[Jump kernel and reference measure]
\label{assu:eta=000020properties} Given a family of reference measures
$\{\pi_{i}\}_{i\in I}$ belonging to $\mathcal{M}_{loc}^{+}(X)$,
we consider a family of functions $\eta_{i}:\{(x,y)\in X\times X\backslash\{x=y\}\}\rightarrow[0,\infty)$,
again where $i\in I$, satisfying the following properties: 
\begin{enumerate}
\item For all $i\in I$, $\eta_{i}$ is symmetric: $\eta_{i}(x,y)=\eta_{i}(y,x)$
for all $x\neq y$ in $\text{supp}(\pi_{i})$;
\item For every bounded and continuous function $f:X\rightarrow\mathbb{R}$
and $i\in I$, the function 
\[
\text{supp}(\pi_{i})\ni x\mapsto\int f(y)(1\wedge|x-y|^{2})\eta_{i}(x,y)d\pi_{i}(y)
\]
is bounded and continuous, and 
\[
\sup_{i\in I}\sup_{x\in\text{supp}(\pi_{i})}\int f(y)(1\wedge|x-y|^{2})\eta_{i}(x,y)d\pi_{i}(y)<\infty.
\]
\item The following uniform integrability condition holds: 
\[
\lim_{R\rightarrow\infty}\sup_{i\in I}\sup_{x\in\text{supp}(\pi_{i})}\iint_{A_{R}(x)}(1\wedge|x-y|^{2})\eta_{i}(x,y)d\pi_{i}(y)=0
\]
where $A_{R}(x)=\{y\in X:|x-y|<1/R\text{ or }|x-y|>R\}$.

\hspace{-4.5em}Additionally, for some results we require that 
\item \textbf{$\eta_{i}(x,y)$ }is a continuous function on $G$;
\item $\eta_{i}(x,y)>0$ for all $x\neq y$ in $X$; and,
\item For sufficiently small $\delta_{0}>0$ there exists a constant $S_{i}\geq1$
(not necessarily uniform in $i$) such that for all $|z|<\delta_{0}$
and all $i\in I$, 
\[
\frac{\eta_{i}(x-z,y-z)}{\eta_{i}(x,y)}\leq S_{i},
\]
and moreover $S_{i}\rightarrow1$ as $\delta_{0}\rightarrow0$. 
\end{enumerate}
\end{assumption}

\begin{rem*}
In many of our results below, we fix a \emph{single} reference measure
$\pi\in\mathcal{P}(X)$ and kernel $\eta$. However, we will need
to consider a countable family of reference measures together with
kernels, when investigating the stability of gradient flows, in Proposition
\ref{prop:stability}.

In the case where $X=\mathbb{R}^{d}$, our Assumption \ref{assu:eta=000020properties}
contains Assumption 1.1 from \cite{erbar2014gradient}; the relevance
for us is that we make use of several results from \cite{erbar2014gradient}
which require this assumption. (Each of these results has been adapted
to the setting where $X=\mathbb{T}^{d}$ in \cite{ferreira2019minimizing}.)
In particular, if we write $J(x,dy):=\eta(x,y)d\pi(y)$ and considers
only a single reference measure $\pi$, one observes that our Assumption
\ref{assu:eta=000020properties} (1-3) is the same as \cite[Assu. 1.1]{erbar2014gradient}
together with the requirement that $J(x,\cdot)\ll\pi$ for $\pi$-almost
all $x\in\mathbb{R}^{d}$.
\end{rem*}
For certain results below (principally: our proof of existence, which
will go by way of our stability theory and a reduction to the discrete
setting), we would like to be able to perform the following construction:
given a kernel $\eta$ and reference measure $\pi$ which satisfy
Assumption \ref{assu:eta=000020properties}, can we produce a sequence
of \emph{discrete} measures $\pi_{n}$ converging to $\pi$, and a
sequence of kernels $\eta_{n}$ converging to $\eta$, such that the
family $\{\eta_{n},;\pi_{n}\}_{n\in\mathbb{N}}\cup\{\eta,\pi\}$ all
satisfies Assumption \ref{assu:eta=000020properties}? We give an
example of such a construction in the case where $X=\mathbb{T}^{d}$
in Proposition \ref{prop:eta-properties-sufficient} below.
\begin{defn}[Underlying graph]
  We denote $G:=\{(x,y)\in X\times X\backslash\{x=y\}\}$. The intended
interpretation is that $G$ is the set of edges we have placed on
$X$, with $\eta(x,y)$ being the edge weight between $x$ and $y$
(but note that $\eta(x,y)=0$ is allowed). We equip $G$ with the
topology it inherits as a subspace of $X\times X$; in particular,
compact sets $K\subset G$ are separated from the diagonal $\{x=y\}$.
\end{defn}

\begin{defn}[Locally finite signed measures on $G$]
 We write $\mathcal{M}_{loc}(G)$ to denote the space of locally
finite signed measures on $G$; in other words, $\mathbf{v}\in\mathcal{M}_{loc}(G)$
iff for every compact $K\subset G$, $\Vert\mathbf{v}\llcorner K\Vert_{TV}<\infty$.
We equip $\mathcal{M}_{loc}(G)$ with the \emph{local weak{*} topology};
a sequence $\mathbf{v}_{n}\in\mathcal{M}_{loc}(G)$ converges locally
weak{*}ly to $\mathbf{v}\in\mathcal{M}_{loc}(G)$ iff, for every $\varphi$
in the space $C_{c}(G)$ of continuous functions with compact support
on $G$, we have $\iint_{G}\varphi d\mathbf{v}_{n}\rightarrow\iint_{G}\varphi d\mathbf{v}$.
We abuse notation and write $\rightharpoonup^{*}$ for local weak{*}
convergence. The space $\mathcal{M}_{loc}(G)$ with the local weak{*}
topology can be viewed as a projective limit of a sequence of spaces
of signed measures along a suitable exhaustion of $G$, see \cite{ambrosio2000functions}
for details.
\end{defn}

The nonlocal Wasserstein distance we define below requires the choice
of an \emph{interpolation function} $\theta$ having the following
form:
\begin{assumption}[Interpolation function]
\label{assu:theta=000020properties}$\theta:[0,\infty)\times[0,\infty)\rightarrow[0,\infty)$
satisfies the following properties:
\begin{enumerate}
\item Regularity: $\theta$ is continuous on $[0,\infty)\times[0,\infty)$
and $C^{1}$ on $(0,\infty)\times(0,\infty)$;
\item Symmetry: $\theta(s,t)=\theta(t,s)$ for $s,t\geq0$; 
\item Positivity, normalization: $\theta(s,t)>0$ for $s,t>0$ and $\theta(1,1)=1$;
\item Monotonicity: $\theta(r,t)\leq\theta(s,t)$ for all $0\leq r\leq s$
and $t\geq0$;
\item Positive 1-homogeneity: $\theta(\lambda s,\lambda t)=\lambda\theta(s,t)$
for $\lambda>0$ and $s\geq t\geq0$;
\item Concavity: the function $\theta:[0,\infty)\times[0,\infty)\rightarrow[0,\infty)$
is concave;
\item Connectedness: $C_{\theta}:=\int_{0}^{1}\frac{dr}{\theta(1-r,1+r)}\in[0,\infty).$
\end{enumerate}
\end{assumption}

However, for our main results concerning equation $(\dagger)$, we
shall always take $\theta$ to be the \emph{logarithmic mean} $\theta(r,s)=\frac{r-s}{\log r-\log s}$. 
\begin{fact}
\label{fact:arithmetic=000020mean=000020upper=000020bound}Any $\theta$
satisfying Assumption \ref{assu:theta=000020properties} also satisfies
$\theta(r,s)\leq\frac{r+s}{2}$.%
\begin{comment}
cite the fact that $\theta\leq(r+s)/2$
\end{comment}
\end{fact}

\begin{defn}[Nonlocal gradient and divergence]
 For any function $\phi:X\rightarrow\mathbb{R}$ we define its \emph{nonlocal
gradient} $\overline{\nabla}\phi:G\rightarrow\mathbb{R}$ by 
\[
\overline{\nabla}\phi(x,y)=\phi(y)-\phi(x)\text{ for all }(x,y)\in G.
\]
For any $\mathbf{v}\in\mathcal{M}_{loc}(G)$, its \emph{nonlocal divergence}
$\overline{\nabla}\cdot\mathbf{v}$ is formally defined as follows:
\[
\int_{X}\phi d\overline{\nabla}\cdot\mathbf{v}:=-\frac{1}{2}\iint_{G}\overline{\nabla}\phi(x,y)d\mathbf{v}(x,y).
\]
\end{defn}

\begin{rem}
Another reasonable choice for a ``nonlocal-divergence'' is the $\eta$-weighted
adjoint of $\overline{\nabla}$, i.e., 
\[
\int_{X}\phi d\overline{\nabla}\cdot\mathbf{j}=-\frac{1}{2}\iint_{G}\overline{\nabla}\phi(x,y)\eta(x,y)d\mathbf{j}(x,y).
\]
Indeed this operator is used in \cite{esposito2019nonlocal,slepcev2023nonlocal}.
We mention this alternate choice here only for the purpose of disambiguation.
Additionally, we warn the reader that other conventions for the definition
of nonlocal gradient and divergence operators are present in the literature,
in particular our definition is not the same as the one presented
in \cite{du2012analysis}.

It is the nonlocal divergence operator $\overline{\nabla}\cdot$ which
replaces the usual divergence operator in the Definition \ref{def:nce}.
However, in the sequel, we will simply write out the integral operator
$\frac{1}{2}\iint_{G}\overline{\nabla}\phi(x,y)d\mathbf{v}(x,y)$
explicitly; the definition here is just presented for easier comparison
with articles such as \cite{erbar2014gradient,esposito2019nonlocal,peletier2020jump}.
A technical reason for this is as follows: it may happen that $\mathbf{v}\in\mathcal{M}_{loc}(G)$
and $\iint_{G}\bar{\nabla}\varphi d\mathbf{v}<\infty$, but there
is no measure $\mathbf{m}\in\mathcal{M}_{loc}(X)$ such that $\int_{X}\phi d\mathbf{m}=-\frac{1}{2}\iint_{G}\overline{\nabla}\phi(x,y)d\mathbf{v}(x,y)$,
i.e. the object ``$\bar{\nabla}\cdot\mathbf{v}$'' does not actually
exist. This problem is avoided in \cite{peletier2020jump} by requiring
the jump kernel to be integrable in the sense that $\sup_{x}\int_{X}\eta(x,y)d\pi(y)<\infty$,
which is a stronger assumption than our Assumption \ref{assu:eta=000020properties}.
\end{rem}

\subsection{Action and nonlocal Fisher information}
\begin{defn}[Action]
\label{def:action}Let $\eta$ satisfy Assumption \ref{assu:eta=000020properties}
and $\theta$ satisfy Assumption \ref{assu:theta=000020properties}.
Let $\mu\in\mathcal{P}(X)$ and $\mathbf{v}\in\mathcal{M}_{loc}(G)$.
Let $\mathcal{\pi}\in\mathcal{M}_{loc}^{+}(X)$ be any \emph{reference
measure}. Define the \emph{action }of the pair $(\mu,\mathbf{v})$
by
\[
\mathcal{A}_{\theta,\eta}(\mu,\mathbf{v};\pi):=\iint_{G}\frac{\left(\frac{d\mathbf{v}}{d\lambda}(x,y)\right)^{2}}{2\theta\left(\frac{d\mu^{1}}{d\lambda}(x,y),\frac{d\mu^{2}}{d\lambda}(x,y)\right)}d\lambda(x,y)
\]
where $d\mu^{1}=\eta(x,y)d\mu(x)d\pi(y)$ and $d\mu^{2}=\eta(x,y)d\pi(x)d\mu(y)$,
and $\lambda$ is any measure in $\mathcal{M}_{loc}(G)$ dominating
$|\mathbf{v}|$, $\mu^{1}$, and $\mu^{2}$. Here, the fraction in
the integrand is understood with the convention that $\frac{0}{0}=0$.

If $\theta$ and $\eta$ and $\pi$ are obvious from context, we simply
write $\mathcal{A}(\mu,\mathbf{v})$. We also introduce the notation
\[
A(v,r,s)=\begin{cases}
\frac{v^{2}}{2\theta(r,s)} & \theta(r,s)>0\\
0 & v,\theta(r,s)=0\\
\infty & \text{else}
\end{cases}
\]
so that $\mathcal{A}(\mu,\mathbf{v})=\iint_{G}A\left(\frac{d\mathbf{v}}{d\lambda},\frac{d\mu^{1}}{d\lambda},\frac{d\mu^{2}}{d\lambda}\right)d\lambda$.
We note that $A$ is jointly convex and 1-homogeneous.
\end{defn}

\begin{rem*}
Note that $\mathcal{A}_{\theta,\eta,\pi}(\mu,\mathbf{v})$ does not
depend on the choice of $\lambda$ satisfying this domination condition
$|\mathbf{v}|,\mu\otimes\pi,\pi\otimes\mu\ll\lambda$, since $\theta$
is 1-homogeneous.
\end{rem*}
\begin{rem}
\label{rem:reference-measure} In particular, if we consider a ``weighted
measured graph'' $G=(V,E,w,\pi_{n})$ with finitely many vertices
located in $\mathbb{T}^{d}$ or $\mathbb{R}^{d}$, where $w(x_{i},x_{j})=\eta(x_{i},x_{j})$
for any $(x_{i},x_{j})\in E$, and $\pi_{n}$ is a measure supported
on $V$, then selecting $\pi_{n}$ as the reference measure in Definition
\ref{def:action} causes $\mathcal{A}_{\eta,\theta,m_{n}}$ to coincide
with the action associated to the \emph{ersatz Wasserstein distance
on finite Markov chains }from \cite{maas2011gradient}, if we also
restrict to the case where $\mu\ll\pi_{n}$ and $\mathbf{v}\ll\sum_{i,j}w(x_{i},x_{j})\pi_{n}(x_{i})\pi_{n}(x_{j})$.
\end{rem}

The action $\mathcal{A}$ has favorable convexity and lower semicontinuity
properties, as has already been observed in \cite{erbar2014gradient}.
The following proposition is a slight extension of existing convexity
and lower semicontinuity results for $\mathcal{A}$ already proved
in e.g. \cite{erbar2014gradient} or \cite{slepcev2023nonlocal},
in that the kernel $\eta$ is also allowed to vary. We also reiterate
that, as observed in \cite{slepcev2023nonlocal}, for fixed $\eta$,
$\mathcal{A}(\mu,\mathbf{v};\pi)$ is jointly \emph{topologically}
(not simply sequentially) l.s.c. in $(\mu,\mathbf{v};\pi)$, which
in particular establishes that $\mathcal{A}$ is jointly Borel measurable.
\begin{prop}
\label{prop:action=000020convex=000020lsc}The action $\mathcal{A}_{\theta,\eta}(\mu,\mathbf{v};\pi)$
is jointly convex in $(\mu,\mathbf{v},\pi)$, and is jointly sequentially
lower semicontinuous in $\eta$ and $(\mu,\mathbf{v},\pi)$ with respect
to:
\begin{itemize}
\item uniform convergence on compact subsets of $G$ for continuous kernels
$\eta$,
\item the narrow topology on $\mathcal{P}(X)$ for $\mu$,
\item the local weak{*} topology on $\mathcal{M}_{loc}^{+}(X)$ for $\pi$,
and
\item the local weak{*} topology on $\mathcal{M}_{loc}(G)$ for $\mathbf{v}$. 
\end{itemize}
Moreover, $\mathcal{A}_{\theta,\eta}(\mu,\mathbf{v};\pi)$ is jointly
topologically l.s.c. in $(\mu,\mathbf{v},\pi)$ for fixed $\eta$.
\end{prop}

\begin{proof}
Let $(\eta_{n}(x,y))$ be a sequence of kernels which are continuous
functions on $G$, and assume that $\eta_{n}(x,y)$ converges to some
kernel $\eta(x,y)$ uniformly on compact subsets $K$ of $G$. For
a given measure $\mathbf{m}\in\mathcal{M}_{loc}(G)$, we define the
measure $\eta_{n}\cdot\mathbf{m}$ by the following formula: for all
$\varphi\in C_{c}(G)$, 
\[
\iint_{G}\varphi(x,y)d(\eta_{n}\cdot\mathbf{m})(x,y)=\iint_{G}\varphi(x,y)\eta_{n}(x,y)d\mathbf{m}(x,y).
\]
Now, since $\eta_{n}(x,y)$ converges uniformly on compact sets to
$\eta(x,y)$, it follows for any compactly supported $\varphi$ that
$\varphi\eta_{n}$ converges uniformly to $\varphi\eta$. Hence, for
all $\varphi\in C_{c}(G)$ it holds that 
\begin{align*}
\iint_{G}\varphi(x,y)d(\eta_{n}\cdot\mathbf{m})(x,y) & =\iint_{G}\varphi(x,y)\eta_{n}(x,y)d\mathbf{m}(x,y)\\
 & \rightarrow\iint_{G}\varphi(x,y)\eta(x,y)d\mathbf{m}(x,y)\\
 & =\iint_{G}\varphi(x,y)d(\eta\cdot\mathbf{m})(x,y).
\end{align*}
Therefore, by the characterizing property of weak{*} convergence in
$\mathcal{M}_{loc}(G)$, it follows that $\eta_{n}\cdot\mathbf{m}\rightharpoonup^{*}\eta\cdot\mathbf{m}$.
More generally, suppose we have a sequence $\mathbf{m}_{n}\rightharpoonup^{*}\mathbf{m}$.
Fix a compact $K\subset G$. Recall that the duality pairing operation
between $C_{c}(K)$ and $\mathcal{M}(K)$ is jointly sequentially
continuous, so we observe that: given any $\varphi\in C_{c}(K)$,
$\varphi\eta_{n}$ and $\varphi\eta$ also belong to $C_{c}(K)$,
and $\varphi\eta_{n}$ converges to $\varphi\eta$ in $C_{c}(K)$,
 and thus
\[
\iint_{G}\varphi(x,y)\eta_{n}(x,y)d\mathbf{m}_{n}(x,y)\rightarrow\iint_{G}\varphi(x,y)\eta(x,y)d\mathbf{m}(x,y).
\]
Since $\varphi$ and $K$ were arbitrary, this implies that $\eta_{n}\cdot\mathbf{m}_{n}\rightharpoonup^{*}\eta\cdot\mathbf{m}$.

In particular, consider the case where $\mathbf{m}=\mu\otimes\pi\llcorner G$
and $\mathbf{m}_{n}=\mu_{n}\otimes\pi_{n}\llcorner G$ with $\mu_{n}$
converging narrowly to $\mu$ and $\pi_{n}$ converging locally weak{*}ly
to $\pi$. It follows that $\eta_{n}\cdot(\mu_{n}\otimes\pi_{n}\llcorner G)\rightharpoonup^{*}\eta\cdot(\mu\otimes\pi\llcorner G)$
in $\mathcal{M}_{loc}(G)$. Obviously the same applies if we replace
$\mu\otimes\pi$ with $\pi\otimes\mu$. 

Now recall that the \emph{action }of the pair $(\mu,\mathbf{v})$
is given by
\[
\iint_{G}A\left(\frac{d\mathbf{v}}{d\lambda},\frac{d\mu^{1}}{d\lambda},\frac{d\mu^{2}}{d\lambda}\right)d\lambda
\]
where $d\mu^{1}=\eta(x,y)d\mu(x)d\pi(y)$ and $d\mu^{2}=\eta(x,y)d\pi(x)d\mu(y)$,
and $\lambda$ is any measure in $\mathcal{M}_{loc}(G)$ dominating
$|\mathbf{v}|$, $\mu^{1}$, and $\mu^{2}$; in fact we can take $\lambda=|\mathbf{v}|+\mu^{1}+\mu^{2}$.
More generally, we can consider the integrand 
\[
\iint_{G}A\left(\frac{d\sigma^{1}}{d|\sigma|},\frac{d\sigma^{2}}{d|\sigma|},\frac{d\sigma^{3}}{d|\sigma|}\right)d|\sigma|
\]
where $\sigma\in\mathcal{M}_{loc}(G;\mathbb{R}^{3})$. Then, since
$A$ is jointly convex in its arguments, and positively 1-homogeneous,
we can apply Reshetnyak's theorem (in the form presented in \cite[Theorem A.1]{slepcev2023nonlocal})
to deduce that $\iint_{G}A\left(\frac{d\sigma^{1}}{d|\sigma|},\frac{d\sigma^{2}}{d|\sigma|},\frac{d\sigma^{3}}{d|\sigma|}\right)d|\sigma|$
is convex in $\sigma$, and sequentially lower-semicontinuous in $\sigma$
w.r.t. local weak{*} convergence in $\mathcal{M}_{loc}(G;\mathbb{R}^{3})$.
In our particular situation, it then suffices to consider the case
where 
\[
\sigma_{n}=\left(\mathbf{v}_{n},\eta_{n}\cdot(\mu_{n}\otimes\pi_{n}\llcorner G),\eta_{n}\cdot(\pi_{n}\otimes\mu_{n}\llcorner G)\right)
\]
which we have established converges locally weak{*}ly to $\sigma=\left(\mathbf{v},\eta\cdot(\mu\otimes\pi\llcorner G),\eta\cdot(\pi\otimes\mu\llcorner G)\right)$.
Indeed, for these choices of $\sigma_{n}$ and $\sigma$, we have
\[
\iint_{G}A\left(\frac{d\sigma^{1}}{d|\sigma|},\frac{d\sigma^{2}}{d|\sigma|},\frac{d\sigma^{3}}{d|\sigma|}\right)d|\sigma|=\iint_{G}A\left(\frac{d\mathbf{v}}{d\lambda},\frac{d\mu^{1}}{d\lambda},\frac{d\mu^{2}}{d\lambda}\right)d\lambda
\]
and similarly for $\sigma_{n}$. Lastly, we see that $\mathcal{A}$
is jointly topologically l.s.c. in $(\mu,\mathbf{v},\pi)$ by arguing
as in \cite[Corollary A.2]{slepcev2023nonlocal}.
\end{proof}
A very similar strategy can be used to study the convexity and semicontinuity
properties of the \emph{nonlocal relative Fisher information}, which
will turn out to coincide with the metric slope of the relative entropy
with respect to $\mathcal{W}_{\eta,\pi}$.
\begin{defn}
Let $\mu,\pi\in\mathcal{P}(X)$ and $\eta\in C_{\geq0}(G)$. The \emph{nonlocal
Fisher information} of $\mu$ with respect to $\pi$ is denoted $\mathcal{I}_{\eta}(\mu\mid\pi)$
and is denoted as follows: 
\[
\mathcal{I}_{\eta}(\mu\mid\pi):=\begin{cases}
\frac{1}{2}\iint_{G}\bar{\nabla}\frac{d\mu}{d\pi}(x,y)\bar{\nabla}\log\frac{d\mu}{d\pi}(x,y)\eta(x,y)d\pi(x)d\pi(y) & \mu\ll\pi\\
+\infty & \text{else}
\end{cases}
\]
where for the logarithm in the integrand above, we use the convention
that for $r>0$, $r(\log r-\log0)=(-r)(\log0-\log r)=\infty$, and
$\log0-\log0=0$. Note that with this convention $\mathcal{I}_{\eta}(\mu\mid\pi)\geq0$
for all $\mu,\pi\in\mathcal{P}(X)$. 
\end{defn}

\begin{prop}
\label{prop:fisher-convex-lsc}The nonlocal relative Fisher information
$\mathcal{I}_{\eta}(\mu\mid\pi)$ is jointly convex in $(\mu,\pi)$,
and jointly sequentially lower semicontinuous in the following sense:
if $(\mu_{n})_{n\in\mathbb{N}}$ converges narrowly to $\mu$ in in
$\mathcal{P}(X)$, and and $(\pi_{n})_{n\in\mathbb{N}}$ converging
locally weak{*}ly to $\pi$ in $\mathcal{M}_{loc}^{+}(X)$, with $\mu_{n}\ll\pi_{n}$
and $\mu\ll\pi$, and $(\eta_{n})_{n\in\mathbb{N}}$ is a sequence
of continuous kernels, converging uniformly on compact subsets of
$G$, then 
\[
\liminf_{n\rightarrow\infty}\mathcal{I}_{\eta_{n}}(\mu_{n}\mid\pi_{n})\geq\mathcal{I}_{\eta}(\mu\mid\pi).
\]
Furthermore, if $\eta_{n}$ and $\eta$ are strictly positive on $G$,
then we can drop the assumption that $\mu_{n}\ll\pi_{n}$ and $\mu\ll\pi$.
\end{prop}

\begin{proof}
We claim that 
\[
\mathcal{I}_{\eta}(\mu\mid\pi)=F\left(\left(\eta\cdot(\mu\otimes\pi\llcorner G),\eta\cdot(\pi\otimes\mu\llcorner G)\right)\right)
\]
where for any vector-valued measure $\lambda=(\lambda_{1},\lambda_{2})$
in $\mathcal{M}_{loc}(G;\mathbb{R}^{2})$, 
\[
F(\lambda):=\iint_{G}f\left((x,y),\frac{d\lambda}{d|\lambda|}\right)d|\lambda|(x,y)
\]
and 
\[
f\left((x,y),\frac{d\lambda}{d|\lambda|}\right):=\frac{1}{2}\left(\frac{d\lambda_{1}}{d|\lambda|}(x,y)-\frac{d\lambda_{2}}{d|\lambda|}(x,y)\right)\left(\log\frac{d\lambda_{1}}{d|\lambda|}(x,y)-\log\frac{d\lambda_{2}}{d|\lambda|}(x,y)\right).
\]
(Here as above, we use the convention that for $r>0$, $r(\log r-\log0)=(-r)(\log0-\log r)=\infty$,
and $\log0-\log0=0$.) This representation of $\mathcal{I}_{\eta}(\mu\mid\pi)$
as $F(\lambda)$ will allow us to invoke the Reshetnyak theorem (in
the form presented in \cite[Theorem A.1]{slepcev2023nonlocal}) to
deduce the desired convexity and lower semicontinuity result ---
together with the fact that $\eta_{n}\cdot(\mu_{n}\otimes\pi_{n}\llcorner G)\rightharpoonup^{*}\eta\cdot(\mu\otimes\pi\llcorner G)$
in $\mathcal{M}_{loc}(G)$ (and similarly for $\pi\otimes\mu$), which
we have already observed in the proof of Proposition \ref{prop:action=000020convex=000020lsc}
above. 

Indeed, write $\lambda=\left(\eta\cdot(\mu\otimes\pi\llcorner G),\eta\cdot(\pi\otimes\mu\llcorner G)\right)$.
Then, under the assumption that $\mu\ll\pi$, we compute that
\begin{multline*}
\frac{1}{2}\iint_{G}\bar{\nabla}\frac{d\mu}{d\pi}\bar{\nabla}\log\frac{d\mu}{d\pi}\eta d\pi\otimes\pi\\
\begin{aligned} & =\frac{1}{2}\iint_{G}\left(\left(\frac{d\mu\otimes\pi}{d\pi\otimes\pi}-\frac{d\pi\otimes\mu}{d\pi\otimes\pi}\right)\left(\log\frac{d\mu\otimes\pi}{d\pi\otimes\pi}-\log\frac{d\pi\otimes\mu}{d\pi\otimes\pi}\right)\right)\eta\frac{d\pi\otimes\pi}{d|\lambda|}d|\lambda|\\
 & =\frac{1}{2}\iint_{G}\left(\left(\frac{d\mu\otimes\pi}{d|\lambda|}-\frac{d\pi\otimes\mu}{d|\lambda|}\right)\left(\log\frac{d\mu\otimes\pi}{d|\lambda|}-\log\frac{d\pi\otimes\mu}{d|\lambda|}\right)\right)\eta d|\lambda|.
\end{aligned}
\end{multline*}
Then, using the 1-homogeneity of $(r-s)(\log r-\log s)$, and the
fact that $\eta\frac{d\mu\otimes\pi}{d|\lambda|}=\frac{d\eta\cdot(\mu\otimes\pi)}{d|\lambda|}$,
we see that 
\begin{multline*}
\frac{1}{2}\iint_{G}\left(\left(\frac{d\mu\otimes\pi}{d|\lambda|}-\frac{d\pi\otimes\mu}{d|\lambda|}\right)\left(\log\frac{d\mu\otimes\pi}{d|\lambda|}-\log\frac{d\pi\otimes\mu}{d|\lambda|}\right)\right)\eta d|\lambda|\\
=\frac{1}{2}\iint_{G}\left(\left(\frac{d\eta\cdot(\mu\otimes\pi)}{d|\lambda|}-\frac{d\eta\cdot(\pi\otimes\mu)}{d|\lambda|}\right)\left(\log\frac{d\eta\cdot(\mu\otimes\pi)}{d|\lambda|}-\log\frac{d\eta\cdot(\pi\otimes\mu)}{d|\lambda|}\right)\right)d|\lambda|
\end{multline*}
which is what we wanted to show, at least for the case where $\mu\ll\pi$.

Now, what happens if $\mu\not\ll\pi$? Then there exists a set $A$
on which $\mu(A)>0$ but $\pi(A)=0$; hence picking any set $B$ such
that $\pi(B)>0$, we have $\mu\otimes\pi(A\times B)>0$, hence $(\mu\otimes\pi+\pi\otimes\mu)(A\times B)>0$;
however, $\pi\otimes\mu(A\times B)=0$. 

Next, let us further assume that $B\cap A=\emptyset$. (This assumption
is permissible because we can start with any $B$ satisfying $\pi(B)>0$,
then instead take $B\backslash A$, which by assumption has $\pi(B\backslash A)=\pi(B)$.)
Note in particular that this means $A\times B\cap\{x=y\}=\emptyset$.
Consequently, still keeping $\lambda=\left(\eta\cdot(\mu\otimes\pi\llcorner G),\eta\cdot(\pi\otimes\mu\llcorner G)\right)$,
we see that for $|\lambda|$-almost all $(x,y)\in A\times B\subset G$,
\[
\frac{d\eta\cdot(\mu\otimes\pi)}{d|\lambda|}(x,y)\neq0,\text{ and }\frac{d\eta\cdot(\pi\otimes\mu)}{d|\lambda|}(x,y)=0.
\]
 Therefore,  (still using the convention that for $r>0$, $r(\log r-\log0)=(-r)(\log0-\log r)=\infty$,
and $\log0-\log0=0$) on $A\times B$, it holds $|\lambda|$-almost
always that
\[
\left(\left(\frac{d\eta\cdot(\mu\otimes\pi)}{d|\lambda|}-\frac{d\eta\cdot(\pi\otimes\mu)}{d|\lambda|}\right)\left(\log\frac{d\eta\cdot(\mu\otimes\pi)}{d|\lambda|}-\log\frac{d\eta\cdot(\pi\otimes\mu)}{d|\lambda|}\right)\right)=\infty;
\]
and since $|\lambda|(A\times B)>0$, and $A\times B\cap\{x=y\}=\emptyset$,
we have that 
\[
\iint_{G}\left(\left(\frac{d\mu\otimes\pi}{d|\lambda|}-\frac{d\pi\otimes\mu}{d|\lambda|}\right)\left(\log\frac{d\mu\otimes\pi}{d|\lambda|}-\log\frac{d\pi\otimes\mu}{d|\lambda|}\right)\right)\eta d|\lambda|=\infty
\]
also. (Here we use that $\eta>0$ on $G$.)%
\begin{comment}
Hmm, do we maybe need that $G$ is the entire space here? Look carefully
at this later
\end{comment}
{} Since $\mathcal{I}(\mu\mid\pi)$ is defined to equal $\infty$ whenever
$\mu\not\ll\pi$, this establishes the claim for general $\mu$ and
$\pi$ in $\mathcal{P}(X)$ and $\mathcal{M}_{loc}^{+}(X)$ respectively.
\begin{comment}
Thus, since $(r-s)(\log r-\log s)$ is convex and 1-homogeneous, we
can apply the Reshetnyak theorem to deduce that $\mathcal{I}(\mu\mid\pi)$
is jointly convex and weak{*} l.s.c. in both $\mu$ and $\pi$.
\end{comment}
\end{proof}
\begin{example}
\label{exa:relative-fisher-is-finite}Given $\eta$ and $\pi$, one
might reasonably ask, for what sort of probability measures $\mu$
does it hold that $\mathcal{I}_{\eta}(\mu\mid\pi)<\infty$? If the
density $\frac{d\mu}{d\pi}$ is sufficiently regular, it turns out
an explicit estimate can be obtained. Indeed, suppose that $\left\Vert \frac{d\mu}{d\pi}\right\Vert _{C^{1}}<\infty$
and also $\left\Vert \log\frac{d\mu}{d\pi}\right\Vert _{C^{1}}<\infty$.
(These assumptions are satisfied, for instance, by probability measures
$\mu$ whose density $\frac{d\mu}{d\pi}$ is proportional to $e^{-b|x|}+c$
for some constants $b,c>0$; this condition makes sense in the case
where $\pi$ itself is a probability measure.) Then, we can use the
fact that $|\bar{\nabla}\varphi(x,y)|\leq\Vert\varphi\Vert_{C^{1}}(1\wedge|x-y|)$
to compute that 
\begin{align*}
\int_{X}\bar{\nabla}\frac{d\mu}{d\pi}(x,y)\bar{\nabla}\log\frac{d\mu}{d\pi}(x,y)\eta(x,y)d\pi(x) & \le\left\Vert \frac{d\mu}{d\pi}\right\Vert _{C^{1}}\int_{X}(1\wedge|x-y|)\bar{\nabla}\log\frac{d\mu}{d\pi}(x,y)\eta(x,y)d\pi(x)\\
 & \leq\left\Vert \frac{d\mu}{d\pi}\right\Vert _{C^{1}}\left\Vert \log\frac{d\mu}{d\pi}\right\Vert _{C^{1}}\int_{X}(1\wedge|x-y|)^{2}\eta(x,y)d\pi(x)
\end{align*}
and the last line is bounded, uniformly in $y$, by Assumption \ref{assu:eta=000020properties}(2).
This implies directly that $\mathcal{I}_{\eta}(\mu\mid\pi)<\infty$.
\end{example}

\subsection{Nonlocal continuity equation.}

Our nonlocal continuity equation is the same as \cite[Section 3]{erbar2014gradient}
when the underlying space $X$ is taken to be $\mathbb{R}^{d}$ (respectively,
\cite[Section 4]{ferreira2019minimizing} for $X=\mathbb{T}^{d}$).
Namely, we consider the \emph{weak formulation} equation
\begin{equation}
\forall\varphi\in C_{c}^{\infty}([0,T]\times X)\qquad\int_{0}^{T}\int_{X}\partial_{t}\varphi_{t}(x)d\mu_{t}(x)dt+\frac{1}{2}\int_{0}^{T}\iint_{G}\bar{\nabla}\varphi_{t}(x,y)d\mathbf{v}_{t}(x,y)dt=0.\label{eq:nce}
\end{equation}

\begin{defn}[Nonlocal continuity equation]
\label{def:nce}Let $T>0$. We say that $(\mu_{t},\mathbf{v}_{t})_{t\in[0,T]}$
\emph{solves the nonlocal continuity equation} provided that
\begin{enumerate}
\item $\mu_{(\cdot)}:[0,T]\rightarrow\mathcal{P}(X)$ is continuous when
$\mathcal{P}(X)$ is equipped with the narrow topology,
\item $\mathbf{v}_{(\cdot)}:[0,T]\rightarrow\mathcal{M}_{loc}(G)$ is Borel
when $\mathcal{M}_{loc}(G)$ is equipped with the local weak{*} topology,
\item $(\mu_{t},\mathbf{v}_{t})_{t\in[0,T]}$ solves \ref{eq:nce}.
\end{enumerate}
We write $(\mu_{t},\mathbf{v}_{t})\in\mathcal{CE}_{T}$ to indicate
that $(\mu_{t},\mathbf{v}_{t})_{t\in[0,T]}$ satisfies conditions
(1), (2), and (3) above. Furthermore, we write $\mathcal{CE}$ to
denote $\mathcal{CE}_{1}$, and write $(\mu_{t},\mathbf{v}_{t})\in\mathcal{CE}_{T}[\nu,\sigma]$
to indicate that $(\mu_{t},\mathbf{v}_{t})\in\ensuremath{\mathcal{CE}_{T}}$
and $\mu_{0}=\nu$ and $\mu_{1}=\sigma$. 
\end{defn}

It is sometimes useful to consider test functions $\varphi$ for equation
\ref{eq:nce} which have lower regularity. In particular, since $|\bar{\nabla}\varphi(x,y)|\leq\Vert\varphi\Vert_{C^{1}(X)}(1\wedge|x-y|)$,
it makes sense to consider test functions which are $C^{1}$ in space
and differentiable in time provided that $(\mu_{t},\mathbf{v}_{t})_{t\in[0,T]}$
satisfies $\int_{0}^{1}\sqrt{\mathcal{A}_{\eta}(\mu_{t},\mathbf{v}_{t};\pi)}dt$,
in view of Lemma \ref{lem:C-eta=000020upper=000020bound=000020of=000020flux}
below; see also discussion in \cite[Section 3]{erbar2014gradient}.

\subsection{Nonlocal Wasserstein metric}
\begin{defn}
\label{def:nlw}Let $\eta$ satisfy Assumption \ref{assu:eta=000020properties}
and $\theta$ satisfy Assumption \ref{assu:theta=000020properties},
and let $\pi\in\mathcal{M}_{loc}^{+}(X)$. The \emph{nonlocal Wasserstein
distance} $\mathcal{W}_{\eta,\theta,m}$ on $\mathcal{P}(X)$ is defined
by 
\[
\mathcal{W}_{\eta,\theta,\pi}^{2}(\nu,\sigma):=\inf\left\{ \int_{0}^{1}\mathcal{A}_{\eta,\theta}(\mu_{t},\mathbf{v}_{t};\pi)dt:(\mu_{t},\mathbf{v}_{t})\in\mathcal{CE}[\nu,\sigma]\right\} .
\]
We will write $\mathcal{W}_{\eta,\pi}$ to denote the case where $\theta$
is the logarithmic mean.
\end{defn}

\begin{fact}
$\mathcal{W}_{\eta,\theta,\pi}$ is a pseudometric on $\mathcal{P}(X)$.
On $\mathcal{P}(X)$, $\mathcal{W}_{\eta,\theta,\pi}^{2}$ is jointly
convex, and $\mathcal{W}_{\eta,\theta,\pi}$ is jointly narrowly lower
semicontinuous. 
\end{fact}

The proof of this fact is exactly as in \cite{erbar2014gradient}
(resp. \cite{ferreira2019minimizing} when $X=\mathbb{T}^{d}$) and
is therefore omitted.

\begin{comment}
\emph{Summary of key results from Erbar and EPSS}.

\emph{Facts concerning convolutions}.
\end{comment}

In the case where $\pi\in\mathcal{P}(X)$, we use the notation $\mathcal{P}_{\pi}(X):=\{\mu\in\mathcal{P}(X):\mathcal{W}(\mu,\pi)<\infty\}$. 
\begin{defn}
\label{def:ac-curves} We call a time-dependent probability measure
$\mu_{(.)}\in C([0,T];\mathcal{P}(X))$ an \emph{absolutely continuous
(AC) curve} for $\mathcal{W}_{\eta,\theta,\pi}$ if there exists a
Borel $\mathbf{v}_{t}:[0,T]\rightarrow\mathcal{M}_{loc}(G)$ such
that $(\mu_{t},\mathbf{v}_{t})_{t\in[0,T]}\in\mathcal{CE}_{T}$ in
the sense of Definition \ref{def:nce} such that $\int_{0}^{1}\sqrt{\mathcal{A}_{\eta,\theta}(\mu_{t},\mathbf{v}_{t};\pi)}dt<\infty$.
If moreover we can take $\int_{0}^{1}\mathcal{A}_{\eta,\theta}(\mu_{t},\mathbf{v}_{t};\pi)dt<\infty$
we say that $\mu_{t}$ is \emph{2-absolutely continuous }or a \emph{curve
of finite total action}. 
\end{defn}

The terminology in the definition above is justified by the following
result, which makes use of the notion of \emph{metric speed}: for
an absolutely continuous curve $x_{t}$ in a metric space $(X,d)$,
the metric speed $|\dot{x}_{t}|_{d}$ is defined by $\lim_{s\rightarrow t}\frac{d(x_{t},x_{s})}{|t-s|}$,
see \cite{ambrosio2008gradient}. 
\begin{fact}[{\cite[Prop. 4.9]{erbar2014gradient}}]
\label{fact:tangent-flux} A curve $\mu_{(.)}\in C([0,T];\mathcal{P}(X))$
is absolutely continuous in the sense of metric spaces w.r.t. the
metric $\mathcal{W}_{\eta,\theta,\pi}$ iff it is absolutely continuous
in the sense of Definition \ref{def:ac-curves}. Moreover, it holds
that 
\[
\vert\dot{\mu}_{t}|_{\mathcal{W}_{\eta,\theta,\pi}}\leq\sqrt{\mathcal{A}_{\eta,\theta}(\mu_{t},\mathbf{v}_{t};\pi)}\qquad t-a.e..
\]
Furthermore there exists a $t-$a.e. unique Borel $\mathbf{v}_{t}:[0,T]\rightarrow\mathcal{M}_{loc}(G)$
such that $(\mu_{t},\mathbf{v}_{t})_{t\in[0,T]}\in\mathcal{CE}_{T}$
and $\vert\dot{\mu}_{t}|_{\mathcal{W}_{\eta,\theta,\pi}}=\sqrt{\mathcal{A}_{\eta,\theta}(\mu_{t},\mathbf{v}_{t};\pi)}\qquad t-a.e.$.
We call this $\mathbf{v}_{t}$ the \emph{tangent flux} for the AC
curve $\mu_{t}$.
\end{fact}

We do not make such direct use of the metric $\mathcal{W}$ in the
sequel, except that the metric speed $\vert\dot{\mu}_{t}|_{\mathcal{W}_{\eta,\theta,\pi}}$
appears in the synthetic notions of gradient flow in a metric space
from \cite{ambrosio2008gradient}. Accordingly, we do sometimes appeal
to the existence of tangent fluxes for AC curves to relate this notion
back to solutions to the nonlocal continuity equation with finite
total action.

\subsection{Weak solutions for nonlocal diffusion equations}

Our theory of weak solutions for equation ($\dagger$) makes use of
the nonlocal continuity equation in an essential way.
\begin{defn}
\label{def:nde-weak-solution}We say that a narrowly continuous curve
$\rho_{t}:[0,T]\rightarrow\mathcal{P}(X)$ is a \emph{weak solution}
for equation ($\dagger$) provided that $(\rho_{t},\mathbf{v}_{t})_{t\in[0,T]}$
solve the nonlocal continuity equation in the sense of Definition
\ref{def:nce}, where for almost all $t\in[0,T]$, 
\[
d\mathbf{v}_{t}(x,y)=-\bar{\nabla}\frac{d\rho_{t}}{d\pi}(x,y)\eta(x,y)d\pi(x)d\pi(y).
\]
In other words, $\forall\varphi\in C_{c}^{\infty}((0,T)\times X)$,
\begin{equation}
\int_{0}^{T}\int_{X}\partial_{t}\varphi_{t}(x)d\rho_{t}(x)dt-\frac{1}{2}\int_{0}^{T}\iint_{G}\bar{\nabla}\varphi_{t}(x,y)\bar{\nabla}\frac{d\rho_{t}}{d\pi}(x,y)\eta(x,y)d\pi(x)d\pi(y)dt=0.\label{eq:nde-weak-form}
\end{equation}
Furthermore, we say that $\rho_{t}$ is an \emph{entropy weak solution
}if additionally, $\mathcal{H}(\rho_{t}\mid\pi)$ is non-increasing
in $t$.
\end{defn}

A large part of the analysis that follows is ultimately directed towards
understanding the properties of solutions to equation \ref{eq:nde-weak-form}
(viz. existence, uniqueness, stability w.r.t. initial condition, and
so on). However, we do so indirectly, by relating solutions equation
\ref{eq:nde-weak-form} to \emph{curves of maximal slope} for the
relative entropy $\mathcal{H}(\cdot\mid\pi)$ (see Definition \ref{def:curve-of-maximal-slope}
below). 
\begin{rem*}
One might also ask whether it is possible to study solutions to equation
\ref{eq:nde-weak-form} more directly. Indeed, we note that equation
\ref{eq:nde-weak-form} can also be viewed as a \emph{nonlocal conservation
law}, and so one might try and adapt general tools for (for instance)
existence theory of conservation laws to our nonlocal setting. Indeed,
preliminary results concerning nonlocal conservation laws --- meaning,
solutions to the nonlocal continuity equation where the flux $\mathbf{v}_{t}$
is some measurable function of $\rho_{t}$ --- have been established
in \cite{esposito2022class} for a stronger notion of ``solution
to the nonlocal continuity equation'' than our Definition \ref{def:nce}.
(In particular, the setting of \cite{esposito2022class} considers
curves $\rho_{t}$ in the space of positive Radon measures which are
absolutely continuous w.r.t. the \emph{total variation} metric.) Therefore,
we cannot simply apply the existing results in \cite{esposito2022class}
to get results concerning equation \ref{eq:nde-weak-form} ``for
free''.

The notion of entropy weak solution is, of course, loosely borrowed
from the conservation law literature. We will see below in Proposition
\ref{prop:weak-soln-vs-gradient-flow} that entropy weak solutions
to \ref{eq:nde-weak-form} correspond to curves of maximal slope of
$\mathcal{H}(\cdot\mid\pi)$ with respect to the nonlocal Wasserstein
metric $\mathcal{W}_{\eta,\pi}$. 
\end{rem*}

\section{Compactness theorem for AC curves\protect\label{sec:Compactness}}

In this section, we provide a refinement of Erbar's compactness result
from \cite[Prop. 3.4]{erbar2014gradient}, which allows for also varying
the jump kernel $\eta$ and reference measure $\pi$; a similar result,
allowing for variable $\pi$ but with fixed $\eta$ has already appeared
in \cite{esposito2019nonlocal}. The required modifications to the
proof of \cite[Prop. 3.4]{erbar2014gradient} are routine, but sufficiently
intricate that we provide a somewhat detailed proof in our setting.

We also requires two preparatory lemmas; the proof of the first of
which is nearly identical to that of \cite[Lemma 2.6]{erbar2014gradient}. 
\begin{lem}
\label{lem:C-eta=000020upper=000020bound=000020of=000020flux}Let
$\{\eta_{i},\pi_{i}\}_{i\in I}$ satisfy Assumption \ref{assu:eta=000020properties}
(1-3). Then, for all $\mu\in\mathcal{P}(X)$ and $\mathbf{v}\in\mathcal{M}_{loc}(G)$,
and $i\in I$, we have 
\[
\iint_{G}(1\wedge|x-y|)d|\mathbf{v}|(x,y)\leq C_{\eta}\sqrt{\mathcal{A}_{\eta_{i}}(\mu,\mathbf{v};\pi_{i})}
\]
where 
\[
C_{\eta}^{2}:=2\sup_{i\in I}\sup_{x\in\text{supp}(\pi_{i})}\int(1\wedge|x-y|^{2})\eta_{i}(x,y)d\pi_{i}(y)<\infty.
\]
Moreover, for every compact set $K\subset G$, we have 
\[
|\mathbf{v}|(K)\leq\frac{C_{\eta}}{\min\{|x-y|:(x,y)\in K\}}\sqrt{\mathcal{A}_{\eta_{i}}(\mu,\mathbf{v};\pi_{i})}.
\]
\end{lem}

\begin{proof}[Proof of Lemma \ref{lem:C-eta=000020upper=000020bound=000020of=000020flux}]
We compute as follows. Fix $i\in I$. Let $\lambda$ be any measure
in $\mathcal{M}_{loc}^{+}(G)$ which dominates $|\mathbf{v}|$, $\mu^{1}:=\eta_{i}\cdot((\mu\otimes\pi_{i})\llcorner G)$,
and $\mu^{2}:=\eta_{i}\cdot((\pi_{i}\otimes\mu)\llcorner G)$. Then
since $\left(\frac{d|\mathbf{v}|}{d\lambda}\right)^{2}=\left(\frac{d\mathbf{v}}{d\lambda}\right)^{2}$
and hence
\[
\left(\frac{d|\mathbf{v}|}{d\lambda}\right)^{2}=\theta\left(\frac{d\mu^{1}}{d\lambda},\frac{d\mu^{2}}{d\lambda}\right)\frac{\left(\frac{d\mathbf{v}}{d\lambda}\right)^{2}}{\theta\left(\frac{d\mu^{1}}{d\lambda},\frac{d\mu^{2}}{d\lambda}\right)}
\]
we have that 
\begin{multline*}
\iint_{G}(1\wedge|x-y|)d|\mathbf{v}|(x,y)\\
\begin{aligned} & =\iint_{G}(1\wedge|x-y|)\left(\frac{d|\mathbf{v}|}{d\lambda}\right)d\lambda(x,y)\\
 & =\iint_{G}(1\wedge|x-y|)\sqrt{2\theta\left(\frac{d\mu^{1}}{d\lambda},\frac{d\mu^{2}}{d\lambda}\right)}\sqrt{\frac{\left(\frac{d\mathbf{v}}{d\lambda}\right)^{2}}{2\theta\left(\frac{d\mu^{1}}{d\lambda},\frac{d\mu^{2}}{d\lambda}\right)}}d\lambda(x,y)\\
 & \leq\left(\iint_{G}(1\wedge|x-y|^{2})2\theta\left(\frac{d\mu^{1}}{d\lambda},\frac{d\mu^{2}}{d\lambda}\right)d\lambda(x,y)\right)^{1/2}\left(\iint_{G}\frac{\left(\frac{d\mathbf{v}}{d\lambda}\right)^{2}}{2\theta\left(\frac{d\mu^{1}}{d\lambda},\frac{d\mu^{2}}{d\lambda}\right)}d\lambda(x,y)\right)^{1/2}.
\end{aligned}
\end{multline*}
The integral in the second set of parentheses is none other than $\mathcal{A}_{\eta_{i}}(\mu,\mathbf{v})$.
Likewise, using the fact that $\theta(r,s)\leq(r+s)/2$, we have that
\[
\iint_{G}(1\wedge|x-y|^{2})2\theta\left(\frac{d\mu^{1}}{d\lambda},\frac{d\mu^{2}}{d\lambda}\right)d\lambda(x,y)\leq\iint_{G}(1\wedge|x-y|^{2})d\mu^{1}(x,y)+\iint_{G}(1\wedge|x-y|^{2})d\mu^{2}(x,y).
\]
Then, observe that $\mu^{1}=\eta_{i}\cdot((\mu\otimes\pi_{i})\llcorner G)$,
and so 
\begin{align*}
\iint_{G}(1\wedge|x-y|^{2})d\mu^{1}(x,y) & =\iint_{G}(1\wedge|x-y|^{2})\eta_{i}(x,y)d\mu(x)d\pi_{i}(y)\\
 & \leq\sup_{i\in I}\sup_{x\in\text{supp}(\pi_{i})}\int(1\wedge|x-y|^{2})\eta_{i}(x,y)d\pi_{i}(y).
\end{align*}
Applying the same reasoning to $\mu^{2}$, we see that 
\[
\iint_{G}(1\wedge|x-y|^{2})2\theta\left(\frac{d\mu^{1}}{d\lambda},\frac{d\mu^{2}}{d\lambda}\right)d\lambda(x,y)\leq C_{\eta}^{2}.
\]
Altogether this shows that 
\[
\iint_{G}(1\wedge|x-y|)d|\mathbf{v}|(x,y)\leq C_{\eta}\sqrt{\mathcal{A}_{\eta_{i}}(\mu,\mathbf{v};\pi_{i})}
\]
as desired.

Finally, to see that 
\[
|\mathbf{v}|(K)\leq\frac{C_{\eta}}{\min\{|x-y|:(x,y)\in K\}}\sqrt{\mathcal{A}_{\eta_{i}}(\mu,\mathbf{v};\pi_{i})},
\]
we simply observe that 
\[
|\mathbf{v}|(K)\leq\frac{1}{\min\{|x-y|:(x,y)\in K\}}\iint_{G}(1\wedge|x-y|)d|\mathbf{v}|(x,y).
\]
\end{proof}
\begin{lem}
\label{lem:total-flux-convergence}Let $\{\eta_{n},\pi_{n}\}_{n\in\mathbb{N}}\cup\{\eta,\pi\}$
satisfy Assumption \ref{assu:eta=000020properties} (1-3). Let $\sup_{n}\mathcal{A}_{\eta_{n}}(\mu_{n},\mathbf{v}_{n};\pi_{n})<\infty$,
and suppose that $\mu_{n}\rightharpoonup\mu$, $\mathbf{v}_{n}\rightharpoonup^{*}\mathbf{v}$
in $\mathcal{M}_{loc}(G)$, and $\pi_{n}\rightharpoonup^{*}\pi$,
and $\eta_{n}\rightarrow\eta$ uniformly on compact subsets of $G$.
Then, 
\[
\iint_{G}(1\wedge|x-y|)d|\mathbf{v}_{n}|(x,y)\rightarrow\iint_{G}(1\wedge|x-y|)d|\mathbf{v}|(x,y).
\]
Similarly, let $\xi_{n}\in C^{1}(\mathbb{R}^{d})$ and let $\xi_{n}\rightarrow\xi$
uniformly on compact sets, with bounded $C^{1}$ norm. Then, 
\[
\iint_{G}\bar{\nabla}\xi_{n}(x,y)d|\mathbf{v}_{n}|(x,y)\rightarrow\iint_{G}\bar{\nabla}\xi(x,y)d|\mathbf{v}|(x,y).
\]
\end{lem}

\begin{proof}
Let $\varepsilon>0$. By Prokhorov's theorem, for $\delta>0$ sufficiently
small, it holds that that for all $n\in\mathbb{N}$,
\[
\mu_{n}(\{x\in X:|x|>\delta^{-1}\})<\varepsilon
\]
and also $\mu(\{x\in X:|x|>\delta^{-1}\})<\varepsilon$. Let us also
take $\delta>0$ small enough that
\[
\sup_{x\in X}\int_{A_{\delta^{-1}}(x)}(1\wedge|x-y|^{2})\eta_{i}(x,y)d\pi_{i}(y)<\varepsilon
\]
where $A_{(\delta)^{-1}}(x)=\{y\in X:|x-y|<\delta\text{ or }|x-y|>\delta^{-1}\}$.
(Note that this does occur for sufficiently small $\delta$ under
Assumption \ref{assu:eta=000020properties}(3).) Define also
\[
K_{\delta}:=\{(x,y)\in G:|x|\leq2\delta^{-1},|y|\leq2\delta^{-1},\text{ and }|x-y|\geq\delta\}.
\]
Observe that $K_{\delta}$ is a compact subset of $G$. At the same
time, for fixed $x\in X$ satisfying $|x|\leq\delta^{-1}$ we have
that 
\[
\{y\in X:(x,y)\in K_{\delta}^{C}\}\subset A_{\delta^{-1}}(x).
\]
Now let $\chi_{K_{\delta}}$ belong to $C_{c}^{\infty}(G)$, be bounded
between $0$ and $1$, and satisfy $\chi_{K_{\delta}}=1$ on $K_{\delta}$.
Since $\mathbf{v}_{n}\rightharpoonup^{*}\mathbf{v}$ in $\mathcal{M}_{loc}(G)$,
we have that 
\[
\iint_{G}(1\wedge|x-y|)\chi_{K_{\delta}}(x,y)d|\mathbf{v}_{n}|(x,y)\stackrel{n\rightarrow\infty}{\longrightarrow}\iint_{G}(1\wedge|x-y|)\chi_{K_{\delta}}(x,y)d|\mathbf{v}|(x,y).
\]
It then remains to estimate
\[
\iint_{G}(1\wedge|x-y|)(1-\chi_{K_{\delta}}(x,y))d|\mathbf{v}_{n}|(x,y)\leq\iint_{K_{\delta}^{C}}(1\wedge|x-y|)d|\mathbf{v}_{n}|(x,y).
\]
Arguing as in the previous lemma, defining $\mu_{n}^{1}=\eta_{n}\cdot((\mu_{n}\otimes\pi_{n})\llcorner G)$
and $\mu_{n}^{2}=\eta_{n}\cdot((\pi_{n}\otimes\mu_{n})\llcorner G)$
as before, we have
\begin{multline*}
\iint_{K_{\delta}^{C}}(1\wedge|x-y|)d|\mathbf{v}_{n}|(x,y)\\
\begin{aligned} & \leq\left(\iint_{K_{\delta}^{C}}(1\wedge|x-y|^{2})2\theta\left(\frac{d\mu_{n}^{1}}{d\lambda},\frac{d\mu_{n}^{2}}{d\lambda}\right)d\lambda(x,y)\right)^{1/2}\left(\iint_{K_{\delta}^{C}}\frac{\left(\frac{d\mathbf{v}_{n}}{d\lambda}\right)^{2}}{2\theta\left(\frac{d\mu_{n}^{1}}{d\lambda},\frac{d\mu_{n}^{2}}{d\lambda}\right)}d\lambda(x,y)\right)^{1/2}\\
 & \leq\left(\iint_{K_{\delta}^{C}}(1\wedge|x-y|^{2})2\theta\left(\frac{d\mu_{n}^{1}}{d\lambda},\frac{d\mu_{n}^{2}}{d\lambda}\right)d\lambda(x,y)\right)^{1/2}\sqrt{\mathcal{A}_{\eta_{n}}(\mu_{n},\mathbf{v}_{n};\pi_{n})}.
\end{aligned}
\end{multline*}
Note also that we have $\mathcal{A}_{\eta}(\mu,\mathbf{v};\pi)\leq\sup_{n\in\mathbb{N}}\mathcal{A}_{\eta_{n}}(\mu_{n},\mathbf{v}_{n};\pi_{n})$
by the joint lower semicontinuity of the action, and so
\[
\iint_{K_{\delta}^{C}}(1\wedge|x-y|)d|\mathbf{v}_{n}|(x,y)\leq\left(\iint_{K_{\delta}^{C}}(1\wedge|x-y|^{2})2\theta\left(\frac{d\mu^{1}}{d\lambda},\frac{d\mu^{2}}{d\lambda}\right)d\lambda(x,y)\right)^{1/2}\sup_{n\in\mathbb{N}}\sqrt{\mathcal{A}_{\eta_{n}}(\mu_{n},\mathbf{v}_{n};\pi_{n})}.
\]
Next, still reasoning as in Lemma \ref{lem:C-eta=000020upper=000020bound=000020of=000020flux},
we have 
\begin{multline*}
\iint_{K_{\delta}^{C}}(1\wedge|x-y|^{2})2\theta\left(\frac{d\mu_{n}^{1}}{d\lambda},\frac{d\mu_{n}^{2}}{d\lambda}\right)d\lambda(x,y)\\
\leq\iint_{K_{\delta}^{C}}(1\wedge|x-y|^{2})d\mu_{n}^{1}(x,y)+\iint_{K_{\delta}^{C}}(1\wedge|x-y|^{2})d\mu_{n}^{2}(x,y).
\end{multline*}
Then, observe that $\mu_{n}^{1}=\eta_{n}\cdot((\mu_{n}\otimes\pi_{n})\llcorner G)$,
and so 
\begin{align*}
\iint_{K_{\delta}^{C}}(1\wedge|x-y|^{2})d\mu_{n}^{1}(x,y)= & \iint_{K_{\delta}^{C}}(1\wedge|x-y|^{2})\eta_{n}(x,y)d\pi_{n}(y)d\mu_{n}(x)\\
\leq & \int_{|x|\leq\delta^{-1}}\int_{A_{\delta^{-1}}(x)}(1\wedge|x-y|^{2})\eta_{n}(x,y)d\pi_{n}(y)d\mu_{n}(x)\\
 & +\int_{|x|>\delta^{-1}}\int_{X}(1\wedge|x-y|^{2})\eta_{n}(x,y)d\pi_{n}(y)d\mu_{n}(x).\\
\leq & \varepsilon+C_{\eta}\varepsilon.
\end{align*}
Here $C_{\eta}$ is the same constant from Lemma \ref{lem:C-eta=000020upper=000020bound=000020of=000020flux};
the estimate above only used Assumption \ref{assu:eta=000020properties}(1-3)
and the fact that $\mu_{n}(|x|>\delta^{-1})<\varepsilon$. Since we
also have that $\mu(|x|>\delta^{-1})<\varepsilon$, the same estimate
applies to $\mu$ in place of $\mu_{n}$.

Applying the same reasoning to $\mu_{n}^{2}$ and $\mu^{2}$, we see
that 
\[
\iint_{K_{\delta}^{C}}(1\wedge|x-y|^{2})2\theta\left(\frac{d\mu_{n}^{1}}{d\lambda},\frac{d\mu_{n}^{2}}{d\lambda}\right)d\lambda(x,y)\leq(2+2C_{\eta})\varepsilon
\]
for all $n$, and also for $\mu$ in the place of $\mu_{n}$. Therefore
\begin{multline*}
\limsup_{n\rightarrow\infty}\left|\iint_{G}(1\wedge|x-y|)d|\mathbf{v}_{n}|(x,y)-\iint_{G}(1\wedge|x-y|)d|\mathbf{v}|(x,y)\right|\\
\leq2\sqrt{(2+2C_{\eta})\varepsilon}\sup_{n}\sqrt{\mathcal{A}_{\eta_{n}}(\mu_{n},\mathbf{v}_{n};\pi_{n})}.
\end{multline*}
As $\varepsilon>0$ was arbitrary, the first claim holds.

Next we consider the $\iint_{G}\bar{\nabla}\xi(x,y)d|\mathbf{v}|(x,y)$
case. This reduces to the first case: indeed let $K_{\delta}$ and
$\chi_{X_{\delta}}$ be chosen as before. Since $\mathbf{v}_{n}\rightharpoonup^{*}\mathbf{v}$
in $\mathcal{M}_{loc}(G)$, and $\bar{\nabla}\xi_{n}(x,y)$ converges
uniformly on compact subsets of $G$, we have that 
\[
\iint_{G}\bar{\nabla}\xi_{n}(x,y)\chi_{K_{\delta}}(x,y)d|\mathbf{v}_{n}|(x,y)\stackrel{n\rightarrow\infty}{\longrightarrow}\iint_{G}\bar{\nabla}\xi(x,y)\chi_{K_{\delta}}(x,y)d|\mathbf{v}|(x,y).
\]
On the other hand, 
\[
\left|\iint_{G}\bar{\nabla}\xi_{n}(x,y)(1-\chi_{K_{\delta}}(x,y))d|\mathbf{v}_{n}|(x,y)\right|\leq\Vert\xi_{n}\Vert_{C^{1}}\iint_{K_{\delta}^{C}}(1\wedge|x-y|)d|\mathbf{v}_{n}|(x,y)
\]
and similarly 
\[
\left|\iint_{G}\bar{\nabla}\xi(x,y)(1-\chi_{K_{\delta}}(x,y))d|\mathbf{v}|(x,y)\right|\leq\Vert\xi\Vert_{C^{1}}\iint_{K_{\delta}^{C}}(1\wedge|x-y|)d|\mathbf{v}|(x,y).
\]
Since $\sup_{n}\Vert\xi_{n}\Vert_{C^{1}}\geq\Vert\xi\Vert_{C^{1}}$,
we can estimate exactly as before that 
\begin{multline*}
\limsup_{n\rightarrow\infty}\left|\iint_{G}\bar{\nabla}\xi_{n}(x,y)d|\mathbf{v}_{n}|(x,y)-\iint_{G}\bar{\nabla}\xi(x,y)d|\mathbf{v}|(x,y)\right|\\
\leq2\sqrt{(2+2C_{\eta})\varepsilon}\sup_{n}\Vert\xi_{n}\Vert_{C^{1}}\sqrt{\mathcal{A}_{\eta_{n}}(\mu_{n},\mathbf{v}_{n};\pi_{n})}.
\end{multline*}
\end{proof}
\begin{prop}[Compactness of finite-action solutions to nonlocal continuity equation]
\label{prop:compactness-nce}Let $(\pi_{n})_{\in\mathbb{N}}$ be
a sequence of reference measures with $\pi_{n}\rightharpoonup^{*}\pi$.
Suppose also that $\{\eta,\pi\}\cup\{\eta_{n},\pi_{n}\}_{n\in\mathbb{N}}$
satisfy Assumption \ref{assu:eta=000020properties} (1-3) and that
$\eta_{n}$ converges uniformly on compact sets $K\subset G$ to $\eta$. 

For each $n$, let $(\mu_{t,n},\mathbf{v}_{t,n})_{t\in[0,T]}$ be
a solution to the nonlocal continuity equation, such that 
\[
A:=\sup_{n}\int_{0}^{T}\mathcal{A}_{\eta_{n}}(\mu_{t,n},\mathbf{v}_{t,n};\pi_{n})dt<\infty.
\]
Suppose that $\mu_{0,n}\rightharpoonup\mu_{0}$. Then, there exists
a solution $(\mu_{t},\mathbf{v}_{t})_{t\in[0,1]}$ to the nonlocal
continuity equation, such that along some subsequence $n_{k}$, we
have that: $\mu_{t,n_{k}}$ converges narrowly in $\mathcal{P}(X)$
to $\mu_{t}$ for all $t\in[0,T]$, $\int_{0}^{T}\mathbf{v}_{t,n}dt$
converges locally weak{*}ly in $\mathcal{M}_{loc}(G\times[0,T])$
to $\int_{0}^{T}\mathbf{v}_{t}dt$, and moreover, 
\[
\int_{0}^{T}\mathcal{A}_{\eta}(\mu_{t},\mathbf{v}_{t};\pi)dt\leq\liminf_{n}\int_{0}^{T}\mathcal{A}_{\eta_{n}}(\mu_{t,n},\mathbf{v}_{t,n};\pi_{n})dt.
\]
\end{prop}

\begin{proof}
The proof is simply a careful extension of the argument from \cite[Prop. 3.4]{erbar2014gradient},
which in turn is descended from \cite[Lem. 4.5]{dolbeault2009new}.

For each $n$, define $\mathbf{v}^{n}\in\mathcal{M}_{loc}(G\times[0,T])$
by $d\mathbf{v}_{t}^{n}(x,y,t)=d\mathbf{v}^{n}(x,y)dt$. Then, from
Lemma \ref{lem:C-eta=000020upper=000020bound=000020of=000020flux},
we infer that 
\[
\sup_{n}\int_{0}^{T}\int_{G}(1\wedge|x-y|)d|\mathbf{v}_{t}^{n}|(x,y)dt\leq C_{\eta}\sqrt{A}<\infty
\]
and, for any compact $K\subset G$ and $B\subset[0,T]$, we have 
\[
\sup_{n}|\mathbf{v}^{n}|(K\times B)\leq\sup_{n}\int_{B}|\mathbf{v}_{t}^{n}|(K)dt\leq\frac{C_{\eta}\sqrt{A}\sqrt{\text{Leb}(B)}}{\min\{|x-y|:(x,y)\in K\}}<\infty.
\]
This implies that $\mathbf{v}^{n}$ has bounded total variation on
every compact set in $G\times[0,T]$, uniformly in $n$; this implies
that (up to a subsequence, which we do not relabel) there exists a
$\mathbf{v}\in\mathcal{M}_{loc}(G\times[0,T])$ such that $\mathbf{v}^{n}\stackrel{*}{\rightharpoonup}\mathbf{v}$,
such that 
\[
|\mathbf{v}|(K\times B)\leq\frac{C_{\eta}\sqrt{A}\sqrt{\text{Leb}(B)}}{\min\{|x-y|:(x,y)\in K\}}
\]
as well. Consequently, $\mathbf{v}$ can be disintegrated w.r.t. the
Lebesgue measure on $[0,T]$; that is, there exists a Borel family
$(\mathbf{v}_{t})_{t\in[0,T]}\in\mathcal{M}_{loc}(G)$ such that $d\mathbf{v}(x,y,t)=d\mathbf{v}_{t}(x,y)dt$.
Moreover, local weak convergence of $\mathbf{v}^{n}$ to $\mathbf{v}$
implies that 
\[
\int_{0}^{T}\int_{G}(1\wedge|x-y|)d|\mathbf{v}_{t}|(x,y)dt\leq\lim_{n\rightarrow\infty}\int_{0}^{T}\int_{G}(1\wedge|x-y|)d|\mathbf{v}_{t}^{n}|(x,y)dt\leq C_{\eta}\sqrt{A}.
\]
Indeed, let $K_{\delta}$ be the same set from the argument from Lemma
\ref{lem:total-flux-convergence} and $\chi_{K_{\delta}}$ the same
smooth cutoff function. By the same reasoning as in Lemma \ref{lem:total-flux-convergence},
we have that
\[
\int_{0}^{T}\iint_{G}(1\wedge|x-y|)\chi_{K_{\delta}}(x,y)d|\mathbf{v}_{t,n}|(x,y)dt\stackrel{n\rightarrow\infty}{\longrightarrow}\int_{0}^{T}\iint_{G}(1\wedge|x-y|)\chi_{K_{\delta}}(x,y)d|\mathbf{v}_{t}|(x,y)dt
\]
since $(1\wedge|x-y|)\chi_{K_{\delta}}(x,y)$ is also compactly supported
on $[0,T]\times G$. Likewise, 
\[
\int_{0}^{T}\iint_{G}(1\wedge|x-y|)(1-\chi_{K_{\delta}}(x,y))d|\mathbf{v}_{t,n}|(x,y)dt\leq\int_{0}^{T}\iint_{K_{\delta}^{C}}(1\wedge|x-y|)d|\mathbf{v}_{t,n}|(x,y)dt
\]
and 
\begin{align*}
\int_{0}^{T}\iint_{K_{\delta}^{C}}(1\wedge|x-y|)d|\mathbf{v}_{t,n}|(x,y) & \leq\left(\int_{0}^{T}\iint_{K_{\delta}^{C}}(1\wedge|x-y|^{2})2\theta\left(\frac{d\mu^{1}}{d\lambda},\frac{d\mu^{2}}{d\lambda}\right)d\lambda(x,y)\right)^{1/2}\sqrt{A}\\
 & \leq\sqrt{T(2+2C_{\eta})\varepsilon}\cdot\sqrt{A}
\end{align*}
(where $\varepsilon$ is the same as in the proof Lemma \ref{lem:total-flux-convergence})
of so that
\[
\limsup_{n\rightarrow\infty}\left|\int_{0}^{T}\iint_{G}(1\wedge|x-y|)d|\mathbf{v}_{t,n}|(x,y)dt-\int_{0}^{T}\iint_{G}(1\wedge|x-y|)d|\mathbf{v}_{t}|(x,y)dt\right|\leq\sqrt{T(2+2C_{\eta})\varepsilon}\cdot\sqrt{A}.
\]
Sending $\varepsilon\rightarrow0$ establishes the claim. In particular
$t\mapsto\mathbf{v}_{t}$ is admissible for the nonlocal continuity
equation. 

Next we argue that for any $\xi\in C^{1}(X)$ and $t_{0}\leq t_{1}$
between $0$ and $T$, it holds that 
\[
\int_{t_{0}}^{t_{1}}\iint_{G}\bar{\nabla}\xi d\mathbf{v}_{t,n}dt\overset{n\rightarrow\infty}{\longrightarrow}\int_{t_{0}}^{t_{1}}\iint_{G}\bar{\nabla}\xi d\mathbf{v}_{t}dt.
\]
Again we argue similarly to Lemma \ref{lem:total-flux-convergence},
except that we need a smooth cutoff for the time-indicator $1_{(t_{0},t_{1})}$
as well. Accordingly, let $\tilde{\chi}_{t_{0}+\epsilon,t_{1}-\epsilon}$
be a smooth cutoff function in time for the indicator function $1_{(t_{0}+\epsilon,t_{1}-\epsilon)}$
supported compactly inside $(t_{0},t_{1})$. It then holds that 
\[
\int_{t_{0}}^{t_{1}}\iint_{G}\tilde{\chi}_{t_{0}+\epsilon,t_{1}-\epsilon}\chi_{K_{\delta}}\bar{\nabla}\xi d\mathbf{v}_{t,n}dt\overset{n\rightarrow\infty}{\longrightarrow}\int_{t_{0}}^{t_{1}}\iint_{G}\tilde{\chi}_{t_{0}+\epsilon,t_{1}-\epsilon}\chi_{K_{\delta}}\bar{\nabla}\xi d\mathbf{v}_{t}dt.
\]
At the same time, 
\[
|(1-\tilde{\chi}_{t_{0}+\epsilon,t_{1}-\epsilon}\chi_{K_{\delta}})\bar{\nabla}\xi|\leq\Vert\xi\Vert_{C^{1}}(1\wedge|x-y|)(1_{\{t\in[0,t_{0}+\epsilon]\cup[t_{1}-\epsilon,T]\}}\vee1_{K_{\delta}^{C}}).
\]
Hence 
\begin{align*}
\left|\int_{t_{0}}^{t_{1}}\iint_{G}(1-\tilde{\chi}_{t_{0}+\epsilon,t_{1}-\epsilon}\chi_{K_{\delta}})\bar{\nabla}\xi d\mathbf{v}_{t,n}dt\right|\leq & \Vert\xi\Vert_{C^{1}}\int_{[0,t_{0}+\epsilon]\cup[t_{1}-\epsilon,T]}\iint_{G}(1\wedge|x-y|)d|\mathbf{v}_{t,n}|(x,y)dt\\
 & +\Vert\xi\Vert_{C^{1}}\int_{0}^{T}\int_{K_{\delta}^{C}}(1\wedge|x-y|)d|\mathbf{v}_{t,n}|(x,y)dt.
\end{align*}
The first term can be estimated as 
\begin{multline*}
\int_{[0,t_{0}+\epsilon]\cup[t_{1}-\epsilon,T]}\iint_{G}(1\wedge|x-y|)d|\mathbf{v}_{t,n}|(x,y)dt\\
\begin{aligned} & \leq\left(\int_{[0,t_{0}+\epsilon]\cup[t_{1}-\epsilon,T]}2\iint_{G}(1\wedge|x-y|^{2}\eta_{n}(x,y)d\pi_{n}(x)d\mu_{n}(y)\right)^{1/2}\sqrt{A}\\
 & \leq\left(4\epsilon\sup_{n\in\mathbb{N}}\sup_{y\in\text{supp}(\pi_{n})}\iint_{G}(1\wedge|x-y|^{2}\eta_{n}(x,y)d\pi_{n}(x)\right)^{1/2}\sqrt{A}\\
 & =\sqrt{2\epsilon C_{\eta}}\sqrt{A}.
\end{aligned}
\end{multline*}
And we have already seen the second term is bounded by $\sqrt{T(2+2C_{\eta})\varepsilon}\cdot\sqrt{A}$.
Altogether it holds that 
\[
\limsup_{n\rightarrow\infty}\left|\int_{t_{0}}^{t_{1}}\iint_{G}(1-\tilde{\chi}_{t_{0}+\epsilon,t_{1}-\epsilon}\chi_{K_{\delta}})\bar{\nabla}\xi d\mathbf{v}_{t,n}dt\right|\leq\sqrt{2\epsilon C_{\eta}}\sqrt{A}+\sqrt{T(2+2C_{\eta})\varepsilon}\cdot\sqrt{A}
\]
so sending both $\epsilon,\varepsilon\rightarrow0$ the claim is established. 

We can now deduce that for any $\xi\in C^{1}(X)$, 
\begin{align*}
\int_{X}\xi d\mu_{t,n} & =\int_{X}\xi d\mu_{0,n}+\frac{1}{2}\int_{0}^{t}\iint_{G}\bar{\nabla}\xi d\mathbf{v}_{t,n}dt\\
 & \rightarrow\int_{X}\xi d\mu_{0}+\frac{1}{2}\int_{0}^{t}\iint_{G}\bar{\nabla}\xi d\mathbf{v}_{t}dt.
\end{align*}
Define the distribution $\mu_{t}$ by $\int\xi d\mu_{t}:=\int_{X}\xi d\mu_{0}+\int_{0}^{t}\iint_{G}\bar{\nabla}\xi d\mathbf{v}_{t}dt$.
Since $\int_{X}\xi d\mu_{t,n}$ is nonnegative for every nonnegative
$\xi$, the same is true in the limit, and so $\int\xi d\mu_{t}$
is actually a nonnegative Radon measure. Moreover we have that $\mu_{t,n}\rightharpoonup\mu_{t}$
since $C^{1}(X)$ is dense inside $C_{b}(X)$. (Note that $\mu_{t}$
is automatically a probability measure.) Moreover it holds that $t\mapsto\int\xi\mu_{t}$
is continuous for every $\xi\in C^{1}(X)$, in fact is absolutely
continuous, since $t\mapsto\frac{1}{2}\iint_{G}\bar{\nabla}\xi d\mathbf{v}_{t}$
is an $L^{1}$ function and 
\[
\limsup_{s\rightarrow t}\frac{|\int_{X}\xi d\mu_{t}-\int_{X}\xi d\mu_{s}|}{|t-s|}\leq\frac{1}{2}\iint_{G}\bar{\nabla}\xi d\mathbf{v}_{t}.
\]
Therefore $t\mapsto\mu_{t}$ is narrowly continuous since $C^{1}(X)$
is dense in $C_{b}(X)$. 

We now argue that $(\mu_{t},\mathbf{v}_{t})_{t\in[0,T]}\in\mathcal{CE}_{T}$.
So far we have shown that the nonlocal continuity equation holds in
the weak sense when tested against functions of the form $\xi\in C^{1}(X)$,
in other words (using absolute continuity of $t\mapsto\int\xi d\mu_{t}$)
\[
\frac{d}{dt}\int_{X}\xi(x)d\mu_{t}(x)dt=\frac{1}{2}\iint_{G}\bar{\nabla}\xi d\mathbf{v}_{t}dt\text{ }t-a.e.
\]
and so we must upgrade this to test functions $\varphi_{t}\in C_{c}^{\infty}((0,T)\times X)$
which also depend on time. To see this holds, we compute as follows.
(See \cite[Prop. 4.2]{santambrogio2015optimal} for a similar argument
in the case of the usual continuity equation.) Take $\tilde{\varphi}(t)\in C_{c}^{\infty}((0,T))$.
Multiplying and integrating we have that 
\[
\int_{0}^{T}\tilde{\varphi}(t)\frac{d}{dt}\int_{X}\xi(x)d\mu_{t}(x)dt=\frac{1}{2}\int_{0}^{T}\tilde{\varphi}(t)\iint_{G}\bar{\nabla}\xi d\mathbf{v}_{t}dt.
\]
Integrating by parts, we find that
\begin{align*}
\int_{0}^{T}\tilde{\varphi}(t)\frac{d}{dt}\int_{X}\xi(x)d\mu_{t}(x)dt & =-\int_{0}^{T}\frac{d}{dt}\tilde{\varphi}(t)\cdot\int_{X}\xi(x)d\mu_{t}(x)dt\\
 & =-\int_{0}^{T}\frac{d}{dt}\left(\tilde{\varphi}(t)\xi(x)\right)d\mu_{t}(x)dt.
\end{align*}
At the same time, $\tilde{\varphi}(t)\bar{\nabla}\xi(x,y)=\bar{\nabla}\left(\tilde{\varphi}(t)\xi\right)(x,y)$
and so
\[
\int_{0}^{T}\tilde{\varphi}(t)\iint_{G}\bar{\nabla}\xi d\mathbf{v}_{t}dt=\int_{0}^{T}\iint_{G}\bar{\nabla}\left(\tilde{\varphi}(t)\xi\right)d\mathbf{v}_{t}dt.
\]
As linear combination of functions of the form $\tilde{\varphi}(t)\xi(x)$
are dense in $C_{c}^{\infty}((0,T)\times X)$, this implies that for
all $\varphi_{t}\in C_{c}^{\infty}(X\times(0,T))$,
\[
\int_{0}^{T}\int_{X}\partial_{t}\varphi_{t}(x)d\rho_{t}(x)dt+\frac{1}{2}\int_{0}^{T}\iint_{G}\bar{\nabla}\varphi_{t}d\mathbf{v}_{t}dt=0
\]
as desired.

Finally, we observe that $\int_{0}^{T}\mathcal{A}_{\eta}(\mu_{t},\mathbf{v}_{t};\pi)dt\leq\liminf_{n}\int_{0}^{T}\mathcal{A}_{\eta_{n}}(\mu_{t,n},\mathbf{v}_{t,n};\pi_{n})dt$
by the same reasoning as in Proposition \ref{prop:action=000020convex=000020lsc},
by viewing $\int_{0}^{T}\mathcal{A}_{\eta}(\mu_{t},\mathbf{v}_{t};\pi)dt$
as an integral functional on the space $\mathcal{M}_{loc}(G\times[0,T])$,
with arguments $d\mathbf{v}_{t}(x,y)dt$, $\eta(x,y)d\mu(x)d\pi(y)dt$,
and $\eta(x,y)d\pi(x)d\mu(y)dt$.

\end{proof}

\section{Identifying the gradient flow of the relative entropy}

\label{sec:Identifying}

\subsection{Curves of maximal slope}

We recall (from say \cite{ambrosio2008gradient}) that on any metric
space $(X,d)$, a \emph{strong upper gradient $G$} for a functional
$F:X\rightarrow(-\infty,\infty]$ is characterized by the property
\[
\forall(x_{t})_{t\in[0,T]}\in AC([0,T];X)\qquad\forall0\leq t_{0}\leq t_{1}\leq T\qquad|F(x_{t_{0}})-F(x_{t_{1}})|\leq\int_{t_{0}}^{t_{1}}G(x_{t})|\dot{x}_{t}|dt
\]
where $|\dot{x}_{t}|$ denotes the metric speed of $x_{t}$, namely
$\lim_{s\rightarrow t}\frac{d(x_{t},x_{s})}{|t-s|}$.

In fact, for our purposes, it will be convenient to sometimes make
use of a slightly weaker notion: we say that $G$ is a ``a.e. upper
gradient'' provided that: for almost all $\tau_{0},\tau_{1}\in[0,T]$,
\[
\left|F(x_{\tau_{0}})-F(x_{\tau_{1}})\right|\leq\int_{\tau_{0}}^{\tau_{1}}G(x_{t})|\dot{x}_{t}|dt.
\]

Note that by Young's inequality, if $G$ is a strong upper gradient
for $F$, then it holds that: for all $t_{0},t_{1}\in[0,T]$, 
\[
\left|F(x_{\tau_{0}})-F(x_{\tau_{1}})\right|\leq\frac{1}{2}\int_{t_{0}}^{t_{1}}G(x_{t})^{2}dt+\frac{1}{2}\int_{t_{0}}^{t_{1}}|\dot{x}_{t}|^{2}dt.
\]
(Likewise, for almost all $\tau_{0},\tau_{1}\in[0,T]$ provided that
$G$ is an a.e. upper gradient for $F$.) This motivates the following
synthetic definition of a gradient flow in $(X,d)$. 
\begin{defn}[Curve of maximal slope]
\label{def:curve-of-maximal-slope} Let $x_{t}\in AC([0,T];X)$,
and $F:X\rightarrow(-\infty,\infty]$ with strong upper gradient $G$.
We say that $(x_{t})_{t\in[0,T]}$ is a \emph{curve of maximal slope
}for $F$ provided that for all $0\leq t_{0}\leq t_{1}\leq T$, the
following \emph{entropy dissipation inequality }(EDI) holds: 
\[
F(x_{t_{0}})-F(x_{t_{1}})\geq\frac{1}{2}\int_{t_{0}}^{t_{1}}G(x_{t})^{2}dt+\frac{1}{2}\int_{t_{0}}^{t_{1}}|\dot{x}_{t}|^{2}dt.
\]
In the case where $G$ is merely an a.e. upper gradient, we say that
$(x_{t})_{t\in[0,T]}$ is a curve of maximal slope\emph{ }for $F$
provided that $t\mapsto F(x_{t})$ is non-increasing, and for almost
all $0\leq t_{0}\leq t_{1}\leq T$ 
\[
F(x_{t_{0}})-F(x_{t_{1}})\geq\frac{1}{2}\int_{t_{0}}^{t_{1}}G(x_{t})^{2}dt+\frac{1}{2}\int_{t_{0}}^{t_{1}}|\dot{x}_{t}|^{2}dt.
\]
\end{defn}

\begin{rem*}
In the case of the a.e. upper gradient, the condition that $t\mapsto F(x_{t})$
is non-increasing implies that for \emph{all} $t_{0}$ and \emph{almost
all} $t_{1}$,
\[
F(x_{t_{0}})-F(x_{t_{1}})\geq\frac{1}{2}\int_{t_{0}}^{t_{1}}G(x_{t})^{2}dt+\frac{1}{2}\int_{t_{0}}^{t_{1}}|\dot{x}_{t}|^{2}dt
\]
by the monotone convergence theorem, provided that $F$ is l.s.c.
on $X$. This further implies that the inequality obtains for all
$t_{1}$ as well: in view of Young's inequality, we have that $t_{1}\mapsto F(x_{t_{0}})-F(x_{t_{1}})$
is monotone, and equal almost everywhere to the monotone and continuous
function $t_{1}\mapsto\frac{1}{2}\int_{t_{0}}^{t_{1}}G(x_{t})^{2}dt+\frac{1}{2}\int_{t_{0}}^{t_{1}}|\dot{x}_{t}|^{2}dt$.

Our notion of a.e. upper gradient deviates from the definition of
``weak upper gradient'' given in \cite[Def. 1.3.2]{ambrosio2008gradient}
in several ways. Notably we assume that $F(x_{t})$ is non-increasing,
whereas \cite{ambrosio2008gradient} takes $F(x_{t})$ be almost everywhere
equal to a non-increasing function. Our definition of curve of maximal
slope with respect to an a.e. upper gradient turns out to be more
suitable for our purposes, in particular it gives the right relation
between curves of maximal slope and entropy weak solutions to equation
($\dagger$), as we see below in Proposition \ref{prop:weak-soln-vs-gradient-flow}.

\end{rem*}
We will relate curves of maximal slope with entropy weak solutions
to equation ($\dagger$) momentarily. First, however, we present some
results which can already be obtained from the definition of curve
of maximal slope and the variational properties of $\mathcal{H},\mathcal{I},$
and $\mathcal{W}$. In the case where $\sqrt{\mathcal{I}(\cdot\mid\pi)}$
is an a.e. upper gradient for $\mathcal{H}(\cdot\mid\pi)$, we can
exploit an argument of Gigli (originally from \cite[Thm. 15]{gigli2010heat})
to deduce the following. 
\begin{prop}[Uniqueness of curves of maximal slope]
\label{prop:uniqueness} Fix $\rho_{0}\in\mathcal{P}(X)$ such that
$\mathcal{H}(\rho_{0}\mid\pi)<\infty$. Let $\sqrt{\mathcal{I}(\cdot\mid\pi)}$
be an a.e. upper gradient for $\mathcal{H}(\cdot\mid\pi)$. (In particular
it suffices to suppose the same assumptions as in Proposition \ref{prop:one=000020sided=000020chain=000020rule}
below.) Then, there is at most one AC curve $(\rho_{t})_{t\in[0,T]}$
satisfying the following for almost all $\tau\in[0,\infty)$:
\[
0\geq\mathcal{H}(\rho_{\tau}\mid\pi)-\mathcal{H}(\rho_{0}\mid\pi)+\frac{1}{2}\int_{0}^{\tau}|\dot{\rho}_{t}|^{2}dt+\frac{1}{2}\int_{0}^{\tau}\mathcal{I}(\rho_{t}\mid\pi)dt.
\]
In particular there is at most one curve of maximal slope beginning
at $\rho_{0}$.
\end{prop}

\begin{proof}
Suppose that there are two distinct AC curves $(\rho_{t})_{t\in[0,T]},(\hat{\rho}_{t})_{t\in[0,T]}$
both beginning at $\rho_{0}$, which are curves of maximal slope.
For each curve, by \cite[Prop. 4.9]{erbar2014gradient} we can select
a tangent flux $\mathbf{v}_{t}$ (resp. $\hat{\mathbf{v}}_{t}$) such
that $t$-almost surely, 
\[
|\dot{\rho}_{t}|^{2}=\mathcal{A}(\rho_{t},\mathbf{v}_{t};\pi)\text{ and }|\dot{\hat{\rho}}_{t}|^{2}=\mathcal{A}(\rho_{t},\hat{\mathbf{v}}_{t};\pi).
\]
Now, consider the curve $\tilde{\rho}_{t}:=\frac{1}{2}\rho_{t}+\frac{1}{2}\hat{\rho}_{t}$
with tangent flux $\tilde{\mathbf{v}}_{t}:=\frac{1}{2}\mathbf{v}_{t}+\frac{1}{2}\hat{\mathbf{v}}_{t}$.
Observe that $\left(\tilde{\rho}_{t},\tilde{\mathbf{v}}_{t}\right)_{t\in[0,\tau]}\in\mathcal{CE}_{\tau}$
for all $T\geq\tau>0$, and since the action is jointly convex, and
the relative Fisher information is convex in the first argument, and
the relative entropy is \emph{strictly} convex in the first argument,
we find that for almost all $\tau>0$,
\begin{align*}
0 & >\mathcal{H}(\tilde{\rho}_{\tau}\mid\pi)-\mathcal{H}(\rho_{0}\mid\pi)+\frac{1}{2}\int_{0}^{\tau}\mathcal{A}\left(\tilde{\rho}_{t},\tilde{\mathbf{v}}_{t};\pi\right)dt+\frac{1}{2}\int_{0}^{\tau}\mathcal{I}(\tilde{\rho}_{t}\mid\pi)dt\\
 & >\mathcal{H}(\tilde{\rho}_{\tau}\mid\pi)-\mathcal{H}(\rho_{0}\mid\pi)+\frac{1}{2}\int_{0}^{\tau}|\dot{\tilde{\rho}}_{t}|^{2}dt+\frac{1}{2}\int_{0}^{\tau}\mathcal{I}(\tilde{\rho}_{t}\mid\pi)dt.
\end{align*}
On the other hand, since $\sqrt{\mathcal{I}(\cdot\mid\pi)}$ is an
a.e. upper gradient for $\mathcal{H}(\cdot\mid\pi)$, we see that
for \emph{any} AC curve $(\mu_{t})_{t\in[0,T]}$ with $\mathcal{H}(\mu_{0}\mid\pi)<\infty$,
it holds for almost all $\tau\in[0,\infty)$ that 
\[
0\leq\mathcal{H}(\mu_{\tau}\mid\pi)-\mathcal{H}(\mu_{0}\mid\pi)+\frac{1}{2}\int_{0}^{\tau}|\dot{\mu}_{t}|^{2}dt+\frac{1}{2}\int_{0}^{\tau}\mathcal{I}(\mu_{t}\mid\pi)dt
\]
since $\mathcal{H}(\mu_{0}\mid\pi)\leq\liminf_{t_{n}\rightarrow0}\mathcal{H}(\mu_{t}\mid\pi)$.
We therefore obtain a contradiction.

It remains to show the last claim of the proposition. The argument
above shows that the two curves of maximal slope $(\rho_{t})_{t\in[0,T]},(\hat{\rho}_{t})_{t\in[0,T]}$
must agree on almost every interval of the form $[0,\tau]$ with $0<\tau\leq T$.
By density, and the fact that both $(\rho_{t})$ and $(\hat{\rho}_{t})$
are narrowly continuous curves, it follows that $\rho_{t}=\hat{\rho}_{t}$
for all $t\in[0,T]$.
\end{proof}
The following proposition provides ``evolutionary $\Gamma$-convergence''
result (in the sense of \cite{sandier2004gamma,serfaty2011gamma})
as the kernel $\eta$ and reference measure $\pi$ are varied; the
only caveat is that the proposition does not presuppose that $\sqrt{\mathcal{I}(\cdot\mid\pi)}$
is a strong/a.e. upper gradient for $\mathcal{H}(\cdot\mid\pi)$.
Note that in our setting, this is an improvement over what is already
done in \cite[Theorem 5.10]{peletier2020jump}, since their stability
result considers fixed jump kernel and reference measure only.
\begin{prop}[Stability of curves satisfying the EDI]
\label{prop:stability} Fix $\rho_{0},\pi\in\mathcal{P}(X)$ with
$\mathcal{H}(\rho_{0}\mid\pi)<\infty$. Additionally, let $(\rho_{0}^{n})_{n\in\mathbb{N}}$
and $(\pi_{n})_{n\in\mathbb{N}}$ be sequences in $\mathcal{P}(X)$
such that $\rho_{0}^{n}\rightharpoonup\rho_{0}$, $\pi_{n}\rightharpoonup\pi$,
and $\mathcal{H}(\rho_{0}^{n}\mid\pi_{n})\rightarrow\mathcal{H}(\rho_{0}\mid\pi)$;
and let $\eta_{n}$ be a sequence of kernels converging uniformly
on compacts $K\subset G$ to $\eta$, such that $\{\eta_{n},\pi_{n}\}_{n\in\mathbb{N}}\cup\{\eta,\pi\}$
satisfy Assumption \ref{assu:eta=000020properties}. 

Fix $T>0$. Suppose that for each $n$, $(\rho_{t}^{n})_{t\in[0,T]}$
satisfies 
\[
0\geq\mathcal{H}(\rho_{T}^{n}\mid\pi_{n})-\mathcal{H}(\rho_{0}^{n}\mid\pi_{n})+\frac{1}{2}\int_{0}^{T}|\dot{\rho}_{t}^{n}|_{\mathcal{W}_{\eta_{n},\pi_{n}}}^{2}dt+\frac{1}{2}\int_{0}^{T}\mathcal{I}_{\eta_{n}}(\rho_{t}^{n}\mid\pi)dt.
\]
Assume also that $t\mapsto\mathcal{H}(\rho_{t}^{n}\mid\pi_{n})$ is
non-increasing for each $n\in\mathbb{N}$. Then, there exists a subsequence
along $n$ (not relabelled) such that as $n\rightarrow\infty$, $(\rho_{t}^{n})_{t\in[0,T]}$
converges narrowly, uniformly in $t$, to some $(\rho_{t})_{t\in[0,T]}$
beginning at $\rho_{0}$, and this $(\rho_{t})_{t\in[0,T]}$ satisfies
\[
0\geq\mathcal{H}(\rho_{T}\mid\pi)-\mathcal{H}(\rho_{0}\mid\pi)+\frac{1}{2}\int_{0}^{T}|\dot{\rho}_{t}|_{\mathcal{W}_{\eta,\pi}}^{2}dt+\frac{1}{2}\int_{0}^{T}\mathcal{I}_{\eta}(\rho_{t}\mid\pi)dt.
\]
In particular, if $\sqrt{\mathcal{I}_{\eta_{n}}(\cdot\mid\pi)}$ is
an a.e. upper gradient for $\mathcal{H}(\cdot\mid\pi)$ with respect
to $\mathcal{W}_{\eta_{n},\pi_{n}}$ for each $n$, and $\sqrt{\mathcal{I}_{\eta}(\cdot\mid\pi)}$
is an a.e. upper gradient for $\mathcal{H}(\cdot\mid\pi)$ with respect
to $\mathcal{W}_{\eta,\pi}$, this establishes evolutionary $\Gamma$-convergence
of curves of maximal slope as $n\rightarrow\infty$.
\end{prop}

\begin{proof}
Let $C=\sup_{n}\mathcal{H}(\rho_{0}^{n}\mid\pi_{n})$. For each $n$,
it holds in particular that
\[
\frac{1}{2}\int_{0}^{T}|\dot{\rho}_{t}^{n}|_{\mathcal{W}_{\eta_{n},\pi_{n}}}^{2}dt\leq\mathcal{H}(\rho_{0}^{n}\mid\pi^{n})\leq C.
\]
Thus, choosing $\mathbf{v}_{t}^{n}$ such that $t$-almost surely,
$|\dot{\rho}_{t}^{n}|_{\mathcal{W}_{\eta_{n},\pi_{n}}}^{2}=\mathcal{A}_{\eta_{n}}(\rho_{t}^{n},\mathbf{v}_{t}^{n};\pi_{n})$,
we see that
\[
\sup_{n}\int_{0}^{T}\mathcal{A}_{\eta_{n}}(\rho_{t}^{n},\mathbf{v}_{t}^{n};\pi^{n})dt<\infty,
\]
and so applying the compactness result of Proposition \ref{prop:compactness-nce},
we see that along some subsequence (not relabelled), $(\rho_{t}^{n},\mathbf{v}_{t}^{n})_{t\in[0,T]}$
converges to some $(\rho_{t},\mathbf{v}_{t})_{t\in[0,T]}$ which solves
the nonlocal continuity equation, and such that 
\[
\liminf_{n\rightarrow\infty}\int_{0}^{T}|\dot{\rho}_{t}^{n}|^{2}dt=\liminf_{n\rightarrow\infty}\int_{0}^{T}\mathcal{A}_{\eta_{n}}(\rho_{t}^{n},\mathbf{v}_{t}^{n};\pi_{n})dt\geq\int_{0}^{T}\mathcal{A}_{\eta}(\rho_{t},\mathbf{v}_{t};\pi)dt\geq\int_{0}^{T}|\dot{\rho}_{t}|_{\mathcal{W}_{\eta,\pi}}^{2}dt.
\]
At the same time, using the joint weak{*} lower semicontinuity of
the relative entropy and the assumption that $t\mapsto\mathcal{H}(\rho_{t}^{n}\mid\pi_{n})$
is non-increasing for each $n\in\mathbb{N}$, we see that for all
$t\in[0,T]$, 
\[
\mathcal{H}(\rho_{t}\mid\pi)\leq\liminf_{n\rightarrow\infty}\mathcal{H}(\rho_{t}^{n}\mid\pi^{n})\leq\sup_{n\rightarrow\infty}\mathcal{H}(\rho_{0}^{n}\mid\pi)=C
\]
and so in particular $\rho_{t}^{n}\ll\pi_{n}$ and $\rho_{t}\ll\pi$.
Therefore, by Fatou's lemma and the joint sequential lower semicontinuity
property of $\mathcal{I}$ from Proposition \ref{prop:fisher-convex-lsc},
we find that 
\[
\liminf_{n\rightarrow\infty}\int_{0}^{T}\mathcal{I}_{\eta_{n}}(\rho_{t}^{n}\mid\pi^{n})dt\geq\int_{0}^{T}\liminf_{n\rightarrow\infty}\mathcal{I}_{\eta_{n}}(\rho_{t}^{n}\mid\pi^{n})dt\geq\int_{0}^{T}\mathcal{I}_{\eta}(\rho_{t}\mid\pi)dt.
\]
Lastly, since we assumed that $\mathcal{H}(\rho_{0}^{n}\mid\pi^{n})\rightarrow\mathcal{H}(\rho_{0}\mid\pi)$,
we conclude that 
\[
0\geq\mathcal{H}(\rho_{T}\mid\pi)-\mathcal{H}(\rho_{0}\mid\pi)+\frac{1}{2}\int_{0}^{T}|\dot{\rho}_{t}|_{\mathcal{W}_{\eta,\pi}}^{2}dt+\frac{1}{2}\int_{0}^{T}\mathcal{I}_{\eta}(\rho_{t}\mid\pi)dt
\]
as desired.
\end{proof}
In the next lemma, we note that the assumption in the second sentence
of the preceding proposition is not vacuous. The proof is a routine
application of semi-discrete optimal transport and is omitted. 
\begin{lem}[Existence of recovery sequence for relative entropy]
\label{lem:existence-well-prepared-initial} Let $\rho,\pi\in\mathcal{P}_{2}(X)$
with $\pi\ll\text{Leb}^{d}$, and suppose that $\mathcal{H}(\rho\mid\pi)<\infty$.
Then, there exists a pair of sequences $(\rho^{n})_{n\in\mathbb{N}}$
and $(\pi^{n})_{n\in\mathbb{N}}$ in $\mathcal{P}(X)$ such that $\rho^{n}\stackrel{*}{\rightharpoonup}\rho$,
$\pi^{n}\stackrel{*}{\rightharpoonup}\pi$, and $\mathcal{H}(\rho^{n}\mid\pi^{n})\rightarrow\mathcal{H}(\rho\mid\pi)$.
\end{lem}

\subsection{Chain rule}

In this section we provide a ``chain rule'' for the relative entropy
along AC curves $\rho_{t}$; this will allow us to show that $\sqrt{\mathcal{I}(\cdot\mid\pi)}$
is an a.e./strong upper gradient, and that curves of maximal slope
can be identified with entropy weak solutions to $(\dagger)$.
\begin{lem}
\label{lem:Cauchy-Schwarz}Let $(\rho_{t},\mathbf{v}_{t})_{t\in[0,T]}\in\mathcal{CE}_{T}$.
 Suppose that $\rho_{t}\ll\pi$ for almost all $t\in[0,T]$. Then
\[
\left|\iint_{G}\bar{\nabla}\log\frac{d\rho_{t}}{d\pi}d\mathbf{v}_{t}\right|\leq\sqrt{\mathcal{I}(\rho_{t}\mid\pi)}\mathcal{A}(\rho_{t},\mathbf{v}_{t})^{1/2}\qquad t\text{-a.s.}.
\]
In the case where $\mathbf{v}_{t}$ at a given $t\in[0,T]$ is a tangent
flux we have for almost all $t$ that
\[
\left|\iint_{G}\bar{\nabla}\log\frac{d\rho_{t}}{d\pi}d\mathbf{v}_{t}\right|\leq\sqrt{\mathcal{I}(\rho_{t}\mid\pi)}|\rho_{t}^{\prime}|
\]
with equality iff $d\mathbf{v}_{t}(x,y)=\lambda\bar{\nabla}\frac{d\rho_{t}}{d\pi}(x,y)\eta(x,y)d\pi(x)d\pi(y)$
for some $\lambda\in\mathbb{R}$. 
\end{lem}

\begin{proof}
Recall the notation $\hat{\rho}_{t}(x,y)=\theta\left(\frac{d\rho_{t}}{d\pi}(x),\frac{d\rho_{t}}{d\pi}(y)\right)$.
We see that  under the condition that $\mathbf{v}_{t}=v_{t}(x,y)\hat{\rho}_{t}(x,y)\eta(x,y)d\pi(x)d\pi(y)$
(which holds $t$-a.s. by \cite[Lemma 2.3]{erbar2014gradient}) we
have
\[
\iint_{G}\bar{\nabla}\log\frac{d\rho_{t}}{d\pi}d\mathbf{v}_{t}=\iint_{G}\bar{\nabla}\left(\log\frac{d\rho_{t}}{d\pi}\right)(x,y)v_{t}(x,y)\hat{\rho}_{t}(x,y)\eta(x,y)d\pi(x)d\pi(y).
\]
Now, applying Cauchy-Schwarz to $L^{2}(\hat{\rho}\eta(\pi\otimes\pi))$,
we find that 
\begin{align*}
\iint_{G}\left|\bar{\nabla}\log\frac{d\rho_{t}}{d\pi}(v_{t})\right|\hat{\rho}_{t}\eta\pi\otimes\pi & \leq\left[\iint_{G}\left(\bar{\nabla}\log\frac{d\rho_{t}}{d\pi}\right)^{2}\hat{\rho}_{t}\eta\pi\otimes\pi\right]^{1/2}\left[\iint_{G}v_{t}^{2}\hat{\rho}_{t}\eta\pi\otimes\pi\right]^{1/2}\\
 & =\left[\iint_{G}\left(\bar{\nabla}\log\frac{d\rho_{t}}{d\pi}\right)^{2}\hat{\rho}\eta\pi\otimes\pi\right]^{1/2}\mathcal{A}(\rho_{t},\mathbf{v}_{t})^{1/2}.
\end{align*}
In the case where $\theta$ is the logarithmic mean, we have $\hat{\rho}=\frac{\bar{\nabla}\rho}{\bar{\nabla}\log\rho}$,
hence $\iint\left(\bar{\nabla}\log\frac{d\rho_{t}}{d\pi}\right)^{2}\hat{\rho}_{t}\eta\pi\otimes\pi=\mathcal{I}(\rho_{t}\mid\pi)$.
Now, if we have that $|\rho_{t}^{\prime}|=\mathcal{A}(\rho_{t},\mathbf{v}_{t})^{1/2}$
for a specific $t\in[0,T]$, we deduce
\[
\left|\iint_{G}\bar{\nabla}\log\frac{d\rho_{t}}{d\pi}d\mathbf{v}_{t}\right|\leq\sqrt{\mathcal{I}(\rho_{t}\mid\pi)}|\rho_{t}^{\prime}|.
\]
Lastly, we have equality in the instance of the Cauchy-Schwarz inequality
used above, precisely if $v_{t}(x,y)=\lambda\bar{\nabla}\log\frac{d\rho_{t}}{d\pi}(x,y)$
for some $\lambda\in\mathbb{R}$. Hence 
\begin{align*}
d\mathbf{v}_{t}(x,y) & =\lambda\bar{\nabla}\log\frac{d\rho_{t}}{d\pi}(x,y)\hat{\rho}_{t}(x,y)\eta(x,y)d\pi(x)d\pi(y)\\
 & =\lambda\bar{\nabla}\frac{d\rho_{t}}{d\pi}(x,y)\eta(x,y)d\pi(x)d\pi(y).
\end{align*}
\end{proof}
The following lemma is an application of Young's inequality at the
level of integrands. By integrating, it implies that $\left|\iint_{G}\bar{\nabla}\log\rho_{t}d\mathbf{v}_{t}\right|\leq\mathcal{I}(\rho_{t}\mid\pi)+\mathcal{A}(\rho_{t},\mathbf{v}_{t})$,
which can also be deduced by applying Young's inequality to the previous
lemma's statement.
\begin{lem}
\label{lem:Young's=000020inequality=000020action=000020fisher=000020integrands}Let
$\rho\in\mathcal{P}(X)$, $\pi\in\mathcal{M}_{loc}^{+}(X)$, and $\mathbf{v}\in\mathcal{M}_{loc}(G)$,
and let $\lambda$ be any dominating measure for the integrand of
the action $\mathcal{A}_{\eta}(\rho,\mathbf{v};\pi)$. Then, at the
level of integrands, we have pointwise in $(x,y)\in G$ that
\begin{align*}
\left|\bar{\nabla}\log\frac{d\rho}{d\pi}(x,y)\frac{d\mathbf{v}}{d\lambda}(x,y)\right| & \leq\frac{1}{2}\left(\bar{\nabla}\log\frac{d\rho}{d\pi}(x,y)\right)^{2}\theta\left(\frac{d(\rho\otimes\pi)}{d\lambda}(x,y),\frac{d(\pi\otimes\rho)}{d\lambda}(x,y)\right)\eta(x,y)\\
 & \qquad+\frac{\left(\frac{d\mathbf{v}}{d\lambda}(x,y)\right)^{2}}{2\theta\left(\frac{d(\rho\otimes\pi)}{d\lambda}(x,y),\frac{d(\pi\otimes\rho)}{d\lambda}(x,y)\right)\eta(x,y)}.
\end{align*}
\end{lem}

\begin{prop}[A.e. chain rule for the relative entropy]
 \label{prop:one=000020sided=000020chain=000020rule}Let $\pi\ll\text{Leb}^{d}$.%
\begin{comment}
Furthermore, let $\text{KL}(\rho_{0}\mid\pi)<\infty$.
\end{comment}
{} Suppose that $\eta$ and $\pi$ satisfy Assumption \ref{assu:eta=000020properties}.
Let $\theta$ be the logarithmic mean. Lastly, suppose that there
exists a $\mu_{0}\in\mathcal{P}(X)$ with $C^{1}$ density such that
$\mathcal{H}(\mu_{0}\mid\text{Leb}),\mathcal{I}_{\eta}(\mu_{0}\mid\text{Leb})<\infty$. 

Then, for any $\mathcal{W}_{\eta,\pi}$ AC curve $(\rho_{t})_{t\in[0,T]}$
in $\mathcal{P}(X)$ with tangent flux $(\mathbf{v}_{t})_{t\in[0,T]}$,
satisfying 
\[
\int_{0}^{T}\mathcal{H}(\rho_{t}\mid\pi)dt<\infty,\;\int_{0}^{T}\mathcal{I}(\rho_{t}\mid\pi)dt<\infty,\text{ and }\int_{0}^{T}\mathcal{A}_{\eta}(\rho_{t},\mathbf{v}_{t};\pi)dt<\infty
\]
\begin{comment}
and therefore $\text{esssup}_{\tau\in[0,T]}\mathcal{H}(\rho_{\tau}\mid\pi)<\infty,$
\end{comment}
it holds for almost all $\tau_{0}\leq\tau_{1}\in[0,T]$ that 
\[
\mathcal{H}(\rho_{\tau_{0}}\mid\pi)-\mathcal{H}(\rho_{\tau_{1}}\mid\pi)=-\frac{1}{2}\int_{\tau_{0}}^{\tau_{1}}\iint_{G}\bar{\nabla}\log\frac{d\rho_{t}}{d\pi}d\mathbf{v}_{t}dt
\]
and so $\sqrt{\mathcal{I}(\cdot\mid\pi)}$ is an a.e. upper gradient
for $\mathcal{H}(\cdot\mid\pi)$. Furthermore, for almost all $\tau\in[0,T]$,
it holds that 
\[
\mathcal{H}(\rho_{0}\mid\pi)-\mathcal{H}(\rho_{\tau}\mid\pi)\leq-\frac{1}{2}\int_{0}^{\tau}\iint_{G}\bar{\nabla}\log\frac{d\rho_{t}}{d\pi}d\mathbf{v}_{t}dt.
\]
\end{prop}

\begin{rem*}
The assumption that there exists a probability measure $\mu_{0}$
(with smooth density) such that $\mathcal{I}_{\eta}(\mu_{0}\mid\text{Leb})<\infty$
automatically holds in the case where $X=\mathbb{T}^{d}$, since if
one takes $\mu_{0}$ to be a constant multiple of the Lebesgue measure
one has $\mathcal{I}(\mu_{0}\mid\text{Leb})=0$. In the case where
$X=\mathbb{R}^{d}$, the same reasoning does not apply. Nonetheless
for certain specific kernels $\eta$ the existence of such a $\mu_{0}$
can be verified via extrinsic arguments: in particular, the results
of \cite[Section 5]{erbar2014gradient} show that if $\eta(x,y)$
is the fractional heat kernel on $\mathbb{R}^{d}$ for some $\alpha\in(0,2)$,
then if $\varphi_{t}(x)$ is the fundamental solution to the fractional
heat equation, it holds that $\mathcal{H}(\varphi_{t}dx\mid\text{Leb}),\mathcal{I}_{\eta}(\varphi_{t}dx\mid\text{Leb})<\infty$
for all positive $t$; more generally the existence of such a $\mu_{0}$
is implied by the technical assumption \cite[Assu. 5.5, eq. (5.5)]{erbar2014gradient}.
Finally, by Lemma \ref{lem:convolution-regularity-of-Fisher}, if
there exists a $\mu_{0}\in\mathcal{P}(X)$ such that $\mathcal{I}_{\eta}(\mu_{0}\mid\text{Leb})<\infty$,
then $k_{\delta}*\mu_{0}$ has smooth density and satisfies $\mathcal{I}_{\eta}(k_{\delta}*\mu_{0}\mid\text{Leb})\leq S\mathcal{I}_{\eta}(\mu_{0}\mid\text{Leb})<\infty$,
and likewise $\mathcal{H}(\cdot\mid\text{Leb})$ is stable w.r.t.
convolution with smooth mollifiers.
\end{rem*}
We first give a formal argument. Indeed, were it the case that the
relative entropy was differentiable in the first argument, we would
simply differentiate under the integral sign to see that
\[
\mathcal{H}(\rho_{T}\mid\pi)-\mathcal{H}(\rho_{0}\mid\pi)=\int_{0}^{T}\frac{d}{dt}\mathcal{H}(\rho_{t}\mid\pi)dt
\]
whence (setting $f_{t}=\frac{d\rho_{t}}{d\pi}$) we differentiate
under the integral sign again to compute that
\[
\frac{d}{dt}\mathcal{H}(\rho_{t}\mid\pi)=\frac{d}{dt}\int_{X}f_{t}\log f_{t}d\pi=\int_{X}(1+\log f_{t})\partial_{t}f_{t}d\pi.
\]
\begin{comment}
We claim that as measures, $\partial_{t}f_{t}d\pi=\partial_{t}\rho_{t}$.
Indeed, for any test function $\phi_{t}(x)$, 
\[
\int_{[0,1]\times\mathbb{R}^{d}}\phi\partial_{t}f_{t}d\pi=-\int_{[0,1]\times\mathbb{R}^{d}}\partial_{t}\phi f_{t}d\pi=-\int_{[0,1]\times\mathbb{R}^{d}}\partial_{t}\phi d\rho_{t}=\int_{[0,1]\times\mathbb{R}^{d}}\phi\partial_{t}\rho_{t}.
\]
\end{comment}
Next, we plug in $\partial_{t}f_{t}d\pi=\partial_{t}\rho_{t}=-\bar{\nabla}\cdot\mathbf{v}_{t}$,
which formally establishes that 
\[
\int_{X}(1+\log f_{t})\partial_{t}f_{t}d\pi=-\int_{X}(1+\log f_{t})d(\bar{\nabla}\cdot\mathbf{v}_{t})=\frac{1}{2}\iint_{G}\bar{\nabla}\log\frac{d\rho_{t}}{d\pi}d\mathbf{v}_{t}
\]
and so 
\[
\mathcal{H}(\rho_{0}\mid\pi)-\mathcal{H}(\rho_{T}\mid\pi)=-\frac{1}{2}\int_{0}^{T}\iint_{G}\bar{\nabla}\log\frac{d\rho_{t}}{d\pi}d\mathbf{v}_{t}dt.
\]
  Hence the lemma is ``proved''. 

The actual proof we give below goes according to the same strategy,
albeit with a careful regularization procedure. We defer the proof
to Appendix \ref{sec:Deferred-proofs}.

The preceding ``chain rule'' type result is enough for most of our
results below, but slightly weaker than is typically established (for
instance in \cite{erbar2016boltzmann,carrillo2022landau,peletier2020jump});
in particular it is not enough (when combined with Lemma \ref{lem:Cauchy-Schwarz})
to show that $\sqrt{\mathcal{I}(\cdot\mid\pi)}$ is indeed a strong
upper gradient for $\mathcal{H}(\cdot\mid\pi)$. Examining the proof
of Proposition \ref{prop:one=000020sided=000020chain=000020rule}
(in Appendix \ref{sec:Deferred-proofs} below), we see that the essential
stumbling block is in establishing that for \emph{every }$t\in[0,1]$,
$\mathcal{H}^{n}(h_{\sigma}*k_{\delta}*\rho_{t}\mid k_{\delta}*\pi)\rightarrow\mathcal{H}^{n}(k_{\delta}*\rho_{t}\mid\pi)$
as $\sigma\rightarrow0$ (where this notation is defined during the
proof).

In the following Corollary, we indicate alternate assumptions which
allow us to upgrade the result from Proposition \ref{prop:one=000020sided=000020chain=000020rule}
to a ``true'' chain rule, thereby establishing that $\sqrt{\mathcal{I}(\cdot\mid\pi)}$
is a strong upper gradient for $\mathcal{H}(\cdot\mid\pi)$. 
\begin{cor}
\label{cor:upgraded-chain-rule} Let $\pi\ll\text{Leb}^{d}$. Furthermore,
let $\mathcal{H}(\rho_{s_{0}}\mid\pi)<\infty$ for some $s_{0}\in[0,T]$.
Suppose that $\eta$ and $\pi$ satisfy Assumption \ref{assu:eta=000020properties}.
Take any $\mathcal{W}_{\eta,\pi}$ AC curve $(\rho_{t})_{t\in[0,T]}$
in $\mathcal{P}(X)$ satisfying 
\[
\int_{0}^{T}\mathcal{I}(\rho_{t}\mid\pi)dt<\infty,
\]
and assume at least one of the following holds:
\begin{enumerate}
\item $\sup_{x}\int\eta(x,y)d\pi(y)<\infty$,
\item $\sup_{t\in[0,T]}M_{2}(\rho_{t})<\infty$. 
\end{enumerate}
Then, for all $t_{0}\leq t_{1}\in[0,T]$, 
\[
\mathcal{H}(\rho_{t_{0}}\mid\pi)-\mathcal{H}(\rho_{t_{1}}\mid\pi)=-\int_{t_{0}}^{t_{1}}\iint_{G}\bar{\nabla}\log\rho_{t}d\mathbf{v}_{t}dt.
\]
\end{cor}

\begin{rem*}
We note that when $X=\mathbb{T}^{d}$, it automatically holds that
$\sup_{t\in[0,T]}M_{2}(\rho_{t})<\infty$. However, when $X$ is unbounded,
we expect that additional assumptions on $\eta$ and $\pi$ are required
to ensure propagation of 2nd moments along AC curves. In particular,
we mention \cite[Lem. 2.16]{esposito2019nonlocal}, which (in a slightly
different setting than ours) proves proves that one can bound $\sup_{t\in[0,T]}M_{2}(\rho_{t})$
in terms of $M_{2}(\rho_{0})$ and $\int_{0}^{T}\mathcal{A}(\rho_{t},\mathbf{v}_{t})dt$,
but under the additional assumption that $\eta$ satisfies a 4th moment
tail bound, namely 
\[
\sup_{x}\int\left(|x-y|^{2}\vee|x-y|^{4}\right)\eta(x,y)d\pi(y)<\infty.
\]
\end{rem*}
\begin{proof}
(sketch) The case where $\sup_{x}\int\eta(x,y)d\pi(y)<\infty$ is
already established in \cite[Thm. 4.16]{peletier2020jump}: note that
when $\sup_{x}\int\eta(x,y)d\pi(y)<\infty$, then ``Assumption $(V\pi\kappa)$''
in that article is satisfied.

If $\sup_{t\in[0,T]}M_{2}(\rho_{t})<\infty$, then one can argue as
in either of \cite[Section 4]{carrillo2022landau} or \cite[Section 4]{erbar2016boltzmann}.
In particular, the lemma of Carlen \& Carvallo (see \cite{carrillo2022landau}
for a statement) shows that one can use the uniform 2nd moment bound
to prove that $\mathcal{H}(k_{\delta}*\rho_{t}\mid\pi)$ is absolutely
continuous in $t$; one then uses the generalized dominated convergence
theorem, exactly as in the proof given in \cite{carrillo2022landau},
to show that for every $t\in[0,T]$, $t\in[0,1]$, $\mathcal{H}^{n}(h_{\sigma}*k_{\delta}*\rho_{t}\mid\pi)\rightarrow\mathcal{H}^{n}(k_{\delta}*\rho_{t}\mid\pi)$
as $\sigma\rightarrow0$.
\end{proof}
Finally we relate curves of maximal slope to weak solutions for equation
$(\dagger)$.
\begin{prop}
\label{prop:weak-soln-vs-gradient-flow}Let $(\rho_{t})_{t\in[0,T]}$
be a curve of maximal slope for $\mathcal{H}(\cdot\mid\pi)$, with
$\sqrt{\mathcal{I}_{\eta}(\cdot\mid\pi)}$ as an a.e. upper gradient
for $\mathcal{H}(\cdot\mid\pi)$. Additionally, suppose the same assumptions
of Proposition \ref{prop:one=000020sided=000020chain=000020rule}
or Corollary \ref{cor:upgraded-chain-rule}. Then, 
\[
(\rho_{t},\mathbf{v}_{t})_{t\in[0,T]};\qquad d\mathbf{v}_{t}(x,y)=-\bar{\nabla}\rho_{t}(x,y)\eta(x,y)d\pi(x)d\pi(y)
\]
solves the nonlocal continuity equation, and so $(\rho_{t})_{t\in[0,T]}$
is an entropy weak solution to the nonlocal diffusion equation $(\dagger)$.

Furthermore, the following converse holds. We consider the following
cases: 
\begin{enumerate}
\item $\sqrt{\mathcal{I}(\cdot\mid\pi)}$ is a strong upper gradient for
$\mathcal{H}(\cdot\mid\pi)$;
\item $\sqrt{\mathcal{I}(\cdot\mid\pi)}$ is an a.e. upper gradient for
$\mathcal{H}(\cdot\mid\pi)$. 
\end{enumerate}
Then, in case (1), any weak solution to $(\dagger)$  is the unique
curve of maximal slope for $\mathcal{H}(\cdot\mid\pi)$. In particular,
there is at most one weak solution to $(\dagger)$ and it is also
an entropy weak solution. In case (2), we merely have that for almost
all $\tau_{0}\leq\tau_{1}\in[0,T]$, 
\[
\mathcal{H}(\rho_{\tau_{0}}\mid\pi)-\mathcal{H}(\rho_{\tau_{1}}\mid\pi)=\frac{1}{2}\int_{\tau_{0}}^{\tau_{1}}\mathcal{I}(\rho_{t}\mid\pi)dt+\frac{1}{2}\int_{\tau_{0}}^{\tau_{1}}|\rho_{t}^{\prime}|^{2}dt.
\]
In this case, $(\rho_{t})_{t\in[0,T]}$ is the unique weak solution
to $(\dagger)$, and if we assume that $(\rho_{t})_{t\in[0,T]}$
is an entropy weak solution, then $(\rho_{t})_{t\in[0,T]}$ is the
unique curve of maximal slope for $\mathcal{H}(\cdot\mid\pi)$. Alternatively,
if we do not assume that $(\rho_{t})_{t\in[0,T]}$ is an entropy weak
solution, but we know that there exists a curve of maximal slope beginning
at $\rho_{0}$, then we can conclude (as in case (1)) that $(\rho_{t})_{t\in[0,T]}$
is the unique curve of maximal slope for $\mathcal{H}(\cdot\mid\pi)$
and is therefore an entropy weak solution to $(\dagger)$ .
\end{prop}

\begin{rem*}
The converse to this proposition leaves open the possibility that:
when $\sqrt{\mathcal{I}(\cdot\mid\pi)}$ is merely an a.e. upper gradient
for $\mathcal{H}(\cdot\mid\pi)$, if no curve of maximal slope for
$\mathcal{H}(\cdot\mid\pi)$ beginning at $\rho_{0}$ exists, then
there can be a (possibly unique) weak solution to $(\dagger)$ which
is not an entropy weak solution. While we still have that the entropy
dissipation equality is satisfied for almost all pairs of timepoints
$(\tau_{0},\tau_{1})$, it may be that $t\mapsto\mathcal{H}(\rho_{t}\mid\pi)$
has a nullset of timepoints which are jump discontinuities. This situation
would in some sense be unphysical, and we are not aware of an actual
example of such a non-entropy weak solution to $(\dagger)$.
\end{rem*}
\begin{proof}
Let $(\mathbf{v}_{t})_{t\in[0,T]}$ be chosen so that $\mathcal{A}(\rho_{t},\mathbf{v}_{t})=|\rho_{t}^{\prime}|^{2}$
for almost all $t\in[0,T]$. From the fact that $(\rho_{t})_{t\in[0,T]}$
is a curve of maximal slope, we know that $\mathcal{H}(\rho_{t_{0}})-\mathcal{H}(\rho_{t_{1}})\geq0$
for all $t_{0}\leq t_{1}\in[0,T]$. Likewise, the other hypotheses
of either Proposition \ref{prop:one=000020sided=000020chain=000020rule}
or Corollary \ref{cor:upgraded-chain-rule} are all satisfied, so
we know that for almost all $\tau\in[0,T]$, 
\[
\mathcal{H}(\rho_{t_{0}})-\mathcal{H}(\rho_{t_{1}})\leq-\int_{t_{0}}^{t_{1}}\iint_{G}\bar{\nabla}\log\rho_{t}d\mathbf{v}_{t}dt\leq\int_{t_{0}}^{t_{1}}\sqrt{\mathcal{I}(\rho_{t}\mid\pi)}|\dot{\rho}_{t}|dt\leq\frac{1}{2}\int_{t_{0}}^{t_{1}}\mathcal{I}(\rho_{t}\mid\pi)dt+\frac{1}{2}\int_{t_{0}}^{t_{1}}|\dot{\rho}_{t}|^{2}dt.
\]
But since $(\rho_{t})_{t\in[0,T]}$ is a curve of maximal slope, that
is,
\[
\mathcal{H}(\rho_{t_{0}})-\mathcal{H}(\rho_{t_{1}})\geq\frac{1}{2}\int_{t_{0}}^{t_{1}}\mathcal{I}(\rho_{t}\mid\pi)dt+\frac{1}{2}\int_{t_{0}}^{t_{1}}|\dot{\rho}_{t}|^{2}dt,
\]
all of these inequalities must in fact be equalities. Hence, by the
equality condition for Young's inequality, for almost all $t\in[0,T]$
it holds that
\[
\sqrt{\mathcal{I}(\rho_{t}\mid\pi)}=|\dot{\rho}_{t}|.
\]
Together with the equality condition in Lemma \ref{lem:Cauchy-Schwarz},
this implies that for almost all $t\in[0,T]$, $d\mathbf{v}_{t}(x,y)=-\bar{\nabla}\rho_{t}\eta(x,y)d\pi(x)d\pi(y)$.

Conversely, suppose that $(\rho_{t})_{t\in[0,T]}$ is a weak solution
to $(\dagger)$%
\begin{comment}
, and that for almost all $t\in[0,T]$, $\mathcal{H}(\rho_{t}\mid\pi)\leq\mathcal{H}(\rho_{0}\mid\pi)<\infty$
\end{comment}
. In this case, for almost all $t\in[0,T]$, $d\mathbf{v}_{t}(x,y)=-\bar{\nabla}\rho_{t}\eta(x,y)d\pi(x)d\pi(y)$,
and we compute that $\mathcal{A}(\rho_{t},\mathbf{v}_{t})=\mathcal{I}(\rho_{t}\mid\pi)$. 

We first consider case (1), where $\mathcal{I}(\cdot\mid\pi)$ is
a strong upper gradient for $\mathcal{H}(\cdot\mid\pi)$.%
\begin{comment}
Moreover, we know that $\rho_{t}\ll\pi$ $t$-a.e., hence by \cite[Cor. 4.10 and Prop. 4.11]{erbar2014gradient}
it holds $t$-a.e. that $\mathcal{A}(\rho_{t},\mathbf{v}_{t})=|\rho_{t}^{\prime}|$.
{[}more precisely, need to check by hand that $\bar{\nabla}\log\rho_{t}$
belongs to $\overline{\{\bar{\nabla}\varphi:\varphi\in C_{c}^{\infty}(\mathbb{R}^{d})\}}^{L^{2}(\hat{\rho}\eta\pi\otimes\pi)}$.
Is this obvious??{]}
\end{comment}
{} Fix any $0\leq t_{0}\leq t_{1}\leq T$. Since 
\[
\int_{t_{0}}^{t_{1}}\mathcal{I}(\rho_{t}\mid\pi)dt=\int_{t_{0}}^{t_{1}}\mathcal{A}(\rho_{t},\mathbf{v}_{t})dt<\infty
\]
we can apply Corollary \ref{cor:upgraded-chain-rule} to deduce that
\begin{align*}
\mathcal{H}(\rho_{t_{0}}\mid\pi)-\mathcal{H}(\rho_{t_{1}}\mid\pi) & =-\int_{t_{0}}^{t_{1}}\bar{\nabla}\log\rho_{t}d\mathbf{v}_{t}dt.
\end{align*}
In turn, we compute that 
\[
-\int_{t_{0}}^{t_{1}}\bar{\nabla}\log\rho_{t}d\mathbf{v}_{t}dt=\int_{t_{0}}^{t_{1}}\mathcal{I}(\rho_{t}\mid\pi)dt=\frac{1}{2}\int_{t_{0}}^{t_{1}}\mathcal{I}(\rho_{t}\mid\pi)dt+\frac{1}{2}\int_{\tau_{0}}^{\tau_{1}}\mathcal{A}(\rho_{t},\mathbf{v}_{t})dt.
\]
 But since $\mathcal{A}(\rho_{t},\mathbf{v}_{t})\geq|\dot{\rho}_{t}|^{2}$
for almost all $t$ by \cite[Prop. 4.9]{erbar2014gradient}, we have
that 
\[
\mathcal{H}(\rho_{t_{0}}\mid\pi)-\mathcal{H}(\rho_{t_{1}}\mid\pi)\geq\frac{1}{2}\int_{t_{0}}^{t_{1}}\mathcal{I}(\rho_{t}\mid\pi)dt+\frac{1}{2}\int_{t_{0}}^{t_{1}}|\rho_{t}^{\prime}|^{2}dt
\]
so that $\rho_{t}$ is a curve of maximal slope, as desired. On the
other hand, by using the fact that (thanks to Lemma \ref{lem:Cauchy-Schwarz})
$\mathcal{I}(\cdot\mid\pi)$ is a strong upper gradient for $\mathcal{H}(\cdot\mid\pi)$,
and applying Young's inequality, we have that
\[
\mathcal{H}(\rho_{t_{0}}\mid\pi)-\mathcal{H}(\rho_{t_{1}}\mid\pi)\leq\frac{1}{2}\int_{t_{0}}^{t_{1}}\mathcal{I}(\rho_{t}\mid\pi)dt+\frac{1}{2}\int_{t_{0}}^{t_{1}}|\rho_{t}^{\prime}|^{2}dt,
\]
so in fact 
\[
\mathcal{H}(\rho_{t_{0}}\mid\pi)-\mathcal{H}(\rho_{t_{1}}\mid\pi)=\frac{1}{2}\int_{t_{0}}^{t_{1}}\mathcal{I}(\rho_{t}\mid\pi)dt+\frac{1}{2}\int_{t_{0}}^{t_{1}}|\rho_{t}^{\prime}|^{2}dt.
\]
 Uniqueness now follows from Proposition \ref{prop:uniqueness}.

Now, consider case (2), where it is only known that $\mathcal{I}(\cdot\mid\pi)$
is an a.e. upper gradient for $\mathcal{H}(\cdot\mid\pi)$. Hence,
for almost all $0\leq\tau_{0}\leq\tau_{1}\leq T$, it holds that 
\begin{align*}
\mathcal{H}(\rho_{\tau_{0}}\mid\pi)-\mathcal{H}(\rho_{\tau_{1}}\mid\pi) & =-\int_{\tau_{0}}^{\tau_{1}}\bar{\nabla}\log\rho_{t}d\mathbf{v}_{t}dt\\
 & \geq\frac{1}{2}\int_{\tau_{0}}^{\tau_{1}}\mathcal{I}(\rho_{t}\mid\pi)dt+\frac{1}{2}\int_{\tau_{0}}^{\tau_{1}}|\rho_{t}^{\prime}|^{2}dt.
\end{align*}
In particular, $(\rho_{t})_{t\in[0,T]}$ is a curve of maximal slope
when restricted to $[\tau_{0},\tau_{1}]$. Arguing as in case (1),
by applying Lemma \ref{lem:Cauchy-Schwarz} and Young's inequality,
we deduce that in fact, 
\[
\mathcal{H}(\rho_{\tau_{0}}\mid\pi)-\mathcal{H}(\rho_{\tau_{1}}\mid\pi)=\frac{1}{2}\int_{\tau_{0}}^{\tau_{1}}\mathcal{I}(\rho_{t}\mid\pi)dt+\frac{1}{2}\int_{\tau_{0}}^{\tau_{1}}|\rho_{t}^{\prime}|^{2}dt.
\]
In turn, by weak{*} lower semicontinuity of $\mathcal{H}(\cdot\mid\pi)$,
this implies that for almost all $\tau_{0}\in[0,T]$, 
\[
\mathcal{H}(\rho_{\tau_{0}}\mid\pi)-\mathcal{H}(\rho_{T}\mid\pi)\geq\frac{1}{2}\int_{\tau_{0}}^{T}\mathcal{I}(\rho_{t}\mid\pi)dt+\frac{1}{2}\int_{\tau_{0}}^{T}|\rho_{t}^{\prime}|^{2}dt.
\]
Now, suppose that there exists some AC curve $(\tilde{\rho}_{t})_{t\in[0,T]}$
with $\rho_{0}=\tilde{\rho}_{0}$, which is also a weak solution for
$(\dagger)$. Then by the same reasoning, it also holds that for almost
all $\tau_{0},\tau_{1}\in[0,T]$, 
\[
\mathcal{H}(\tilde{\rho}_{\tau_{0}}\mid\pi)-\mathcal{H}(\tilde{\rho}_{\tau_{1}}\mid\pi)=\frac{1}{2}\int_{\tau_{0}}^{\tau_{1}}\mathcal{I}(\tilde{\rho}_{t}\mid\pi)dt+\frac{1}{2}\int_{\tau_{0}}^{\tau_{1}}|\tilde{\rho}_{t}^{\prime}|^{2}dt.
\]
Hence by Proposition \ref{prop:uniqueness}, we deduce that for almost
all $\tau_{0}\leq\tau_{1}\in[0,T]$, $\rho_{t}=\tilde{\rho}_{t}$
when $t\in[\tau_{0},\tau_{1}]$. Quantifying over all such $\tau_{0},\tau_{1}$,
and using the fact that $(\rho_{t})_{t\in[0,T]}$ and $(\tilde{\rho}_{t})_{t\in[0,T]}$
are both narrowly continuous, we conclude that $\rho_{t}=\tilde{\rho}_{t}$
for all $t\in[0,T]$. If furthermore, $(\rho_{t})_{t\in[0,T]}$ is
an entropy weak solution and so $t\mapsto\mathcal{H}(\rho_{t}\mid\pi)$
is non-increasing, we can conclude that $(\rho_{t})_{t\in[0,T]}$
is a curve of maximal slope. In this case we can apply Proposition
\ref{prop:uniqueness} and conclude this curve of maximal slope is
unique.

Lastly, we argue that $(\rho_{t})_{t\in[0,T]}$ is a curve of maximal
slope on the whole of $[0,T]$, assuming that a curve of maximal slope
beginning at $\rho_{0}$ exists. That is, there exists some AC curve
$(\tilde{\rho}_{t})_{t\in[0,T]}$ with $\rho_{0}=\tilde{\rho}_{0}$,
which is a curve of maximal slope for $\mathcal{H}(\cdot\mid\pi)$;
note in this case this means that $t\mapsto\mathcal{H}(\rho_{t}\mid\pi)$
is non-increasing (which the computation from the previous paragraph
does not actually establish). Since $(\rho_{t})_{t\in[0,T]}$ is at
least a curve of maximal slope when restricted to almost every interval
of the form $[\tau_{0},\tau_{1}]$, we know from Proposition \ref{prop:uniqueness}
that $(\rho_{t})$ and $(\tilde{\rho}_{t})$ must agree on each such
interval. By density, and the fact that both $(\rho_{t})$ and $(\tilde{\rho}_{t})$
are narrowly continuous curves, it follows that $\rho_{t}=\tilde{\rho}_{t}$
for all $t\in[0,T]$.
\end{proof}

\section{Existence and Discrete-to-Continuum\protect\label{sec:Existence-and-Discrete-to-Continuum}}

Now we are in a position to prove the unique existence of $\mathcal{W}_{\eta,\pi}$
gradient flows of $\mathcal{H}(\cdot\mid\pi)$.  Our proof goes by
way of a ``finite volume'' discrete-to-continuum result in the case
where $X=\mathbb{T}^{d}$, Proposition \ref{prop:eta-properties-sufficient},
of independent interest.
\begin{thm}[Existence and uniqueness of curves of maximal slope]
\label{thm:existence} Let $\rho_{0},\pi\in\mathcal{P}(X)$ with
$\mathcal{H}(\rho_{0}\mid\pi)<\infty$ and $\pi\ll\text{Leb}^{d}$.
Suppose that Assumption \ref{assu:eta=000020properties} is satisfied.
Then, there exists a unique AC curve $(\rho_{t})_{t\in[0,\infty)}$
beginning at $\rho_{0}$, satisfying
\[
(\forall T>0)\qquad0\geq\mathcal{H}(\rho_{T}\mid\pi)-\mathcal{H}(\rho_{0}\mid\pi)+\frac{1}{2}\int_{0}^{T}|\dot{\rho}_{t}|^{2}dt+\frac{1}{2}\int_{0}^{T}\mathcal{I}_{\eta}(\rho_{t}\mid\pi)dt.
\]
In particular, if $\sqrt{\mathcal{I}_{\eta}(\cdot\mid\pi)}$ is an
a.e. upper gradient for $\mathcal{H}(\cdot\mid\pi)$ with respect
to $\mathcal{W}_{\eta,\pi}$, there exists a unique curve of maximal
slope for $\mathcal{H}(\cdot\mid\pi)$ with respect to $\mathcal{W}_{\eta,\pi}$.
\end{thm}

\begin{proof}
\emph{Step 1}: First we consider the case where $X=\mathbb{T}^{d}$,
$\eta$ is strictly positive, and $\frac{d\pi}{dx}$ is uniformly
bounded from below. Let $\pi_{n}$ be a sequence of discrete measures
formed according to the construction from Proposition \ref{prop:eta-properties-sufficient}
below, and let $\rho_{0}^{n}$ be as in the proof of Lemma \ref{lem:existence-well-prepared-initial}.
For each $n$, we consider a weighted graph on the support of $\pi_{n}$,
with reference probability measure $\pi_{n}$ and edge weights given
as in Proposition \ref{prop:eta-properties-sufficient} below. Note
that $\{\eta,\pi\}\cup\{\eta_{n},\pi_{n}\}_{n\in\mathbb{N}}$ satisfies
Assumption \ref{assu:eta=000020properties} uniformly, by Proposition
\ref{prop:eta-properties-sufficient}.

By the main theorem in \cite{maas2011gradient}, there exists an AC
curve $(\rho_{t}^{n})_{t\in[0,\infty)}$ beginning at $\rho_{0}^{n}$,
which for all $t>0$ is a $\mathcal{W}_{\eta_{n},\pi_{n}}$ gradient
flow of $\mathcal{H}(\cdot\mid\pi_{n})$ , in the usual sense of gradient
flows on finite-dimensional Riemannian manifolds. So \emph{a fortiori}
this gradient flow is also a curve of maximal slope for $\mathcal{H}(\cdot\mid\pi_{n})$
with respect to $\mathcal{W}_{\eta_{n},\pi_{n}}$. Now, for each $T>0$,
we see that Proposition \ref{prop:stability} allows us to extract
a subsequence (not relabelled) such that, as $n\rightarrow\infty$,
$(\rho_{t}^{n})_{t\in[0,T]}\rightharpoonup(\rho_{t})_{t\in[0,T]}$,
and $(\rho_{t})_{t\in[0,T]}$ is a curve of maximal slope for $\mathcal{H}(\cdot\mid\pi)$
with respect to $\mathcal{W}_{\eta,\pi}$. We can then conclude by
invoking uniqueness and quantifying over all $T>0$.

We then easily deduce existence even when $\frac{d\pi}{dx}$ is not
bounded from below. Indeed, the previous argument applies if we take
$(1-\varepsilon)\pi+\varepsilon\text{Leb}^{d}$ as the reference measure
instead; it then suffices to send $\varepsilon\rightarrow0$ and apply
Proposition \ref{prop:stability}, since it is readily seen that $\mathcal{H}(\rho_{0}\mid(1-\varepsilon)\pi+\varepsilon\text{Leb}^{d})\rightarrow\mathcal{H}(\rho_{0}\mid\pi)$
under the assumption that $\mathcal{H}(\rho_{0}\mid\pi)<\infty$.
By similar reasoning, we can also drop the assumption that $\eta$
is strictly positive, by considering a sequence of $\eta_{n}$'s which
\emph{are} strictly positive, and converge pointwise monotonically
downwards to $\eta$.

\emph{Step 2}: Assume that $\pi\ll\text{Leb}^{d}$ and is compactly
supported in $\mathbb{R}^{d}$. The construction from Step 1 also
applies to dilations $N\mathbb{T}^{d}$ of $\mathbb{T}^{d}$ where
$N\in\mathbb{N}$. Let $N$ be sufficiently large that the quotient
map $Q$ from $\mathbb{R}^{d}$ to $N\mathbb{T}^{d}$ is injective
when restricted to $\text{supp}(\pi)$; in this case there is no distinction
between viewing $\pi$ as a measure on $\mathbb{R}^{d}$ or on $N\mathbb{T}^{d}$.
(Note also that $\text{supp}(\rho_{0})\subseteq\text{supp}(\pi)$
by assumption.) Likewise we can modify $\eta$ away from the support
of $\pi$ to be periodic\footnote{More precisely, we reason as follows.  Observe that $\text{supp}(\pi)\times\text{supp}(\pi)\backslash\{x=y\}$
is closed inside $\mathbb{T}^{d}\times\mathbb{T}^{d}\backslash\{x=y\}$,
so we can first restrict $\eta$ to this set, then apply the Tietze
extension theorem  on the space $\mathbb{T}^{d}\times\mathbb{T}^{d}\backslash\{x=y\}$
to extend $\eta\upharpoonright\text{supp}(\pi)\times\text{supp}(\pi)\backslash\{x=y\}$
to a continuous function on all of $\mathbb{T}^{d}\times\mathbb{T}^{d}\backslash\{x=y\}$.
This is equivalent to modifying $\eta$ away from $\text{supp}(\pi)\times\text{supp}(\pi)\backslash\{x=y\}$
to be periodic  on $\mathbb{R}^{d}\times\mathbb{R}^{d}\backslash\{x=y\}$.
Note that this extension need not be symmetric off of $\text{supp}(\pi)\times\text{supp}(\pi)\backslash\{x=y\}$,
but this is irrelevant to Assumption \ref{assu:eta=000020properties}(1).} with respect to cubes of the form $[-N/2,N/2]^{d}$, so that the
quotient map from $\mathbb{R}^{d}$ to $N\mathbb{T}^{d}$ leaves $\eta$
unchanged on $\text{supp}(\pi)$ as well. (In particular modifying
$\eta$ away from the support of $\pi$ does not change whether Assumption
\ref{assu:eta=000020properties}(1-3) is satisfied.) 

Step 1 establishes the existence of a curve of maximal slope $\tilde{\rho}_{t}$
for $\mathcal{H}(\cdot\mid Q_{\#}\pi)$ on $N\mathbb{T}^{d}$ with
respect to $\mathcal{W}_{(Q\times Q)_{\#}\eta,Q_{\#}\pi}$ starting
from $Q_{\#}\rho_{0}$. Now, since $\mathcal{H}(\tilde{\rho}_{t}\mid Q_{\#}\pi)$
is non-increasing, we have also that the support of $\tilde{\rho}_{t}$
is contained within the support of $Q_{\#}\pi$. Now let $\tilde{\mathbf{v}}_{t}$
denote a tangent flux solving the nonlocal continuity equation for
$\tilde{\rho}_{t}$. By \cite[Lem. 3.5]{ferreira2019minimizing} (which
states the periodic version of \cite[Lem. 2.3]{erbar2014gradient}),
we also obtain that $\tilde{\mathbf{v}}_{t}$ is supported on $\text{supp}(Q_{\#}\pi)\times\text{supp}(Q_{\#}\pi)\backslash\{x=y\}$
for almost all $t\in[0,T]$. Accordingly, we can lift $(\tilde{\rho}_{t},\tilde{\mathbf{v}}_{t})_{t\in[0,T]}$
to a solution to the nonlocal continuity equation $(\rho_{t},\mathbf{v}_{t})_{t\in[0,T]}$
on $\mathbb{R}^{d}$ without changing the action, relative entropy,
or relative Fisher information. It follows that $(\rho_{t})_{t\in[0,T]}$
is a curve of maximal slope for $\mathcal{H}(\cdot\mid\pi)$ w.r.t.
$\mathcal{W}_{\eta,\pi}$ beginning at $\rho_{0}$.

\emph{Step 3}: Finally we perform a compact exhaustion. Now let $\pi\in\mathcal{P}(\mathbb{R}^{d})$
be arbitrary apart from assuming that $\pi\ll\text{Leb}^{d}$. Let
$\rho_{0}\in\mathcal{P}(\mathbb{R}^{d})$ such that $\mathcal{H}(\rho_{0}\mid\pi)<\infty$.
Let $(K_{n})_{n\in\mathbb{N}}$ denote a compact exhaustion of $\mathbb{R}^{d}$;
without loss of generality, assume that the restrictions $\pi\llcorner K_{n}$
and $\rho_{0}\llcorner K_{n}$ have nonzero mass for every $n\in\mathbb{N}$.
Let $\pi_{n}$ denote the normalization of $\pi\llcorner K_{n}$ and
similarly for $\rho_{0,n}$. Then $\rho_{0,n}\rightharpoonup\rho_{0}$,
$\pi_{n}\rightharpoonup\pi$, and $\mathcal{H}(\rho_{0,n}\mid\pi_{n})\rightarrow\mathcal{H}(\rho_{0}\mid\pi)$.
 Finally let $\eta_{n}$ be a kernel which agrees with $\eta$ on
$K_{n}\times K_{n}\backslash\{x=y\}$ and is periodic with respect
to cubes of the form $[-N/2,N/2]^{d}$ (chosen to be sufficiently
large that $K_{n}$ belongs to a single such cube) ; such an $\eta_{n}$
exists by the same extension argument as in Step 2. Observe that
Step 2 establishes the existence of a curve of maximal slope for $\mathcal{H}(\cdot\mid\pi_{n})$
w.r.t. $\mathcal{W}_{\eta_{n},\pi_{n}}$ beginning at $\rho_{0,n}$.

It is obvious that $\eta_{n}\rightarrow\eta$ uniformly on compacts
$K\subset G$. It therefore remains to check that $\{\eta_{n},\pi_{n}\}_{n\in\mathbb{N}}\cup\{\eta,\pi\}$
satisfy Assumption \ref{assu:eta=000020properties}, since we can
then apply Proposition \ref{prop:stability}. But this is routine,
since for all bounded continuous $f:\mathbb{R}^{d}\rightarrow\mathbb{R}$,
$x\in\mathbb{R}^{d}$ and $n\in\mathbb{N}$, 
\[
\int f(y)(1\wedge|x-y|^{2})\eta_{n}(x,y)d\pi_{n}(y)=\frac{1}{\pi(K_{n})}\int_{K_{n}}f(y)(1\wedge|x-y|^{2})\eta(x,y)d\pi(y)
\]
and 
\[
\iint_{A_{R}(x)}(1\wedge|x-y|^{2})\eta_{n}(x,y)d\pi_{n}(y)=\frac{1}{\pi(K_{n})}\iint_{A_{R}(x)\cap K_{n}}(1\wedge|x-y|^{2})\eta(x,y)d\pi(y).
\]
Since we have assumed that $\pi(K_{n})\geq\pi(K_{1})>0$, this shows
that Assumption \ref{assu:eta=000020properties} is satisfied uniformly,
as desired.
\end{proof}
\begin{prop}
\label{prop:eta-properties-sufficient}Suppose that $\pi\in\mathcal{P}(\mathbb{T}^{d})$.
Suppose that $\eta$ satisfies Assumption \ref{assu:eta=000020properties}
(1-3) with respect to $\pi$.  For each $n$, let $\{x_{j}^{n}\}_{j=1}^{n^{d}}$
denote an evenly spaced grid of points in $\mathbb{T}^{d}$, such
that for $m\leq n$, $\{x_{j}^{m}\}_{j=1}^{m^{d}}\subseteq\{x_{j}^{n}\}_{j=1}^{n^{d}}$.
Let $T_{n}$ denote a map which sends $x\in\mathbb{T}^{d}$ to the
nearest point $x_{j}^{n}$ amongst the set $\{x_{j}^{n}\}_{j=1}^{n^{d}}$
(in the case of ties, we take $T_{n}$ to break the tie in an arbitrary
measurable fashion). Let $\pi_{n}=\pi\circ T_{n}^{-1}$ be the pushforward
of $\pi$ onto $\{x_{j}^{n}\}_{j=1}^{n^{d}}$ via $T_{n}$. Let $\delta_{n}=\text{diam}T_{n}^{-1}(x_{j}^{n})$
for some $j$ (the choice of $j$ is irrelevant since the grid is
regular). Lastly, on points contained in $\{x_{j}^{n}\}_{j=1}^{n^{d}}$,
we define 
\[
\eta_{n}(x_{j}^{n},x_{k}^{n})=\frac{1}{\pi_{n}(x_{j}^{n})\cdot\pi_{n}(x_{k}^{n})}\iint_{T_{n}^{-1}(x_{j}^{n})\times T^{-1}(x_{k}^{n})}1_{\{|x-y|\geq\delta_{n}/2\}}\eta(x,y)d\pi(x)d\pi(y).
\]
Then, if the pair $\{\eta,\pi\}$ satisfies Assumption \ref{assu:eta=000020properties}
(2-3), it holds that the entire family $\{\eta_{n},\pi_{n}\}_{n\in\mathbb{N}}\cup\{\eta,\pi\}$
satisfies Assumption \ref{assu:eta=000020properties} (2-3). 

Moreover, if $\eta$ is continuous and strictly positive on $\mathbb{T}^{d}\times\mathbb{T}^{d}\backslash\{x=y\}$
 and $\frac{d\pi}{dx}$ is uniformly bounded from below on $\mathbb{T}^{d}$,
it holds that each $\eta_{n}$ can be extended to a continuous function
on all of $G$ in such a way that $\eta_{n}\rightarrow\eta$ uniformly
on compact subsets of $G$.
\end{prop}

\begin{proof}
Given $\eta(x,y)$, we define the following, \emph{different} $\eta_{n}(x,y)$
to be used at the discrete level. Namely: for $\bar{x}\neq\bar{y}$
in the support of the $n$th diadic lattice in $\mathbb{T}^{d}$,
define 
\[
\eta_{n}(\bar{x},\bar{y}):=\frac{1}{\pi_{n}(\bar{x})\pi_{n}(\bar{y})}\iint_{T_{n}^{-1}(\bar{x})\times T_{n}^{-1}(\bar{y})\backslash\{(\text{diam}T_{n}^{-1})/2\}}\eta(x,y)d\pi(x)d\pi(y).
\]
(Here $\pi_{n}:=(T_{n})_{\#}\pi$.) That is, $\eta_{n}$ is a ``finite
volume'' type approximation of $\eta$, where volume is defined w.r.t.
$\pi$, but with a cutoff for when $|x-y|$ is too small. (Note that
here, $T_{n}^{-1}(\bar{x})$ is the $n$th dyadic cube centered at
$\bar{x}$, which has diameter independent of $\bar{x}$.)

We wish to compare $\int(1\wedge|\bar{x}-\bar{y}|^{2})\eta_{n}(\bar{x},\bar{y})d\pi_{n}(\bar{y})$
with $\int(1\wedge|x-y|^{2})\eta(x,y)d\pi(y)$. To that end, we compute
as follows. Under the assumption that $|x-y|\geq|\bar{x}-\bar{y}|/2$
(which holds on the set $T_{n}^{-1}(\bar{x})\times T_{n}^{-1}(\bar{y})\backslash\{(\text{diam}T_{n}^{-1})/2\}$),
we have that $(1\wedge|\bar{x}-\bar{y}|^{2})\leq4(1\wedge|x-y|^{2})$.
Consequently, for fixed $\bar{x}$ we have that 
\begin{multline*}
\sum_{\bar{y}}(1\wedge|\bar{x}-\bar{y}|^{2})\eta_{n}(\bar{x},\bar{y})\pi_{n}(\bar{y})\\
\begin{aligned} & =\frac{1}{\pi_{n}(\bar{x})}\sum_{\bar{y}}(1\wedge|\bar{x}-\bar{y}|^{2})\eta_{n}(\bar{x},\bar{y})\pi_{n}(\bar{x})\pi_{n}(\bar{y})\\
 & =\frac{1}{\pi_{n}(\bar{x})}\sum_{\bar{y}}(1\wedge|\bar{x}-\bar{y}|^{2})\iint_{T_{n}^{-1}(\bar{x})\times T_{n}^{-1}(\bar{y})\backslash\{(\text{diam}T_{n}^{-1})/2\}}\eta(x,y)d\pi(x)d\pi(y)\\
 & \leq4\frac{1}{\pi_{n}(\bar{x})}\sum_{\bar{y}}\iint_{T_{n}^{-1}(\bar{x})\times T_{n}^{-1}(\bar{y})\backslash\{(\text{diam}T_{n}^{-1})/2\}}(1\wedge|x-y|^{2})\eta(x,y)d\pi(x)d\pi(y)\\
 & \leq4\frac{1}{\pi_{n}(\bar{x})}\int_{T_{n}^{-1}(\bar{x})}\left(\int_{\mathbb{T}^{d}}(1\wedge|x-y|^{2})\eta(x,y)d\pi(y)\right)d\pi(x)\\
 & \leq4\sup_{x\in\mathbb{T}^{d}}\int_{\mathbb{T}^{d}}(1\wedge|x-y|^{2})\eta(x,y)d\pi(y)
\end{aligned}
\end{multline*}
where in the last line we have used that $\pi_{n}(\bar{x})=\pi(T_{n}^{-1}(\bar{x}))$.
 This establishes that the family $\{\eta_{n},\pi_{n}\}_{n\in\mathbb{N}}\cup\{\eta,\pi\}$
satisfies Assumption \ref{assu:eta=000020properties} (2) provided
that $\{\eta,\pi\}$satisfies Assumption \ref{assu:eta=000020properties}
(2). (Note that the continuity of the map 
\[
\text{supp}(\pi_{n})\ni x\mapsto\int f(y)(1\wedge|x-y|^{2})\eta_{n}(x,y)d\pi_{n}(y)
\]
 holds trivially for us, since $\text{supp}(\pi_{n})$ is discrete.)

The argument to verify that Assumption \ref{assu:eta=000020properties}
(3) is preserved under discretization is similar. Let $A_{R}(x)=\{y\in X:|x-y|<1/R\text{ or }|x-y|>R\}$.
Let $\bar{x}\in\{x_{j}^{n}\}_{j=1}^{n^{d}}$. Observe that
\begin{multline*}
\sum_{\bar{y}\in\{x_{j}^{n}\}_{j=1}^{n^{d}}\cap A_{R}(\bar{x})\backslash\{\bar{x}\}}(1\wedge|\bar{x}-\bar{y}|^{2})\eta_{n}(\bar{x},\bar{y})\pi_{n}(\bar{y})\\
\begin{aligned} & =\frac{1}{\pi_{n}(\bar{x})}\sum_{\bar{y}\in\{x_{j}^{n}\}_{j=1}^{n^{d}}\cap A_{R}(\bar{x})\backslash\{\bar{x}\}}(1\wedge|\bar{x}-\bar{y}|^{2})\eta_{n}(\bar{x},\bar{y})\pi_{n}(\bar{x})\pi_{n}(\bar{y})\\
 & =\frac{1}{\pi_{n}(\bar{x})}\sum_{\bar{y}\in\{x_{j}^{n}\}_{j=1}^{n^{d}}\cap A_{R}(\bar{x})\backslash\{\bar{x}\}}(1\wedge|\bar{x}-\bar{y}|^{2})\iint_{T_{n}^{-1}(\bar{x})\times T_{n}^{-1}(\bar{y})\backslash\{(\text{diam}T_{n}^{-1})/2\}}\eta(x,y)d\pi(x)d\pi(y)\\
 & \leq4\frac{1}{\pi_{n}(\bar{x})}\sum_{\bar{y}\in\{x_{j}^{n}\}_{j=1}^{n^{d}}\cap A_{R}(\bar{x})\backslash\{\bar{x}\}}\iint_{T_{n}^{-1}(\bar{x})\times T_{n}^{-1}(\bar{y})\backslash\{(\text{diam}T_{n}^{-1})/2\}}(1\wedge|x-y|^{2})\eta(x,y)d\pi(x)d\pi(y)\\
 & \leq4\frac{1}{\pi_{n}(\bar{x})}\int_{T_{n}^{-1}(\bar{x})}\left(\sum_{\bar{y}\in\{x_{j}^{n}\}_{j=1}^{n^{d}}\cap A_{R}(\bar{x})\backslash\{\bar{x}\}}\int_{T_{n}^{-1}(\bar{y})}(1\wedge|x-y|^{2})\eta(x,y)d\pi(y)\right)d\pi(x).
\end{aligned}
\end{multline*}
Now, for any $x\in T_{n}^{-1}(\bar{x})$, we have $|x-\bar{x}|\leq\delta_{n}/2$.
Furthermore, if $\bar{y}\in\{x_{j}^{n}\}_{j=1}^{n^{d}}\cap A_{R}(\bar{x})\backslash\{\bar{x}\}$
(and hence: $|\bar{x}-\bar{y}|<1/R\text{ or }|\bar{x}-\bar{y}|>R$)
and $y\in T_{n}^{-1}(\bar{y})$, it follows that either $|x-y|<1/R+\delta_{n}$,
or $|x-y|>R-\delta_{n}$. 

We consider two cases. Observe that for all $\bar{y}\in\{x_{j}^{n}\}_{j=1}^{n^{d}}\cap A_{R}(\bar{x})\backslash\{\bar{x}\}$,
we have $|\bar{x}-\bar{y}|\geq\delta_{n}$. If $1/R<\delta_{n}$,
then this means that none of the $\bar{y}\in\{x_{j}^{n}\}_{j=1}^{n^{d}}\cap A_{R}(\bar{x})$
satisfy $|\bar{x}-\bar{y}|<1/R$, hence $\bar{y}\in\{x_{j}^{n}\}_{j=1}^{n^{d}}\cap A_{R}(\bar{x})\backslash\{\bar{x}\}$
implies that $|x-y|>R-\delta_{n}$. In this case, we have that for
all $x\in T_{n}^{-1}(\bar{x})$, 
\[
\sum_{\bar{y}\in\{x_{j}^{n}\}_{j=1}^{n^{d}}\cap A_{R}(\bar{x})\backslash\{\bar{x}\}}\int_{T_{n}^{-1}(\bar{y})}(1\wedge|x-y|^{2})\eta(x,y)d\pi(y)\leq\int_{|x-y|>R-\delta_{n}}(1\wedge|x-y|^{2})\eta(x,y)d\pi(y).
\]
Now, note that uniformly in $n$, $\delta_{n}\leq\text{diam}(\mathbb{T}^{d})$,
and so for $R>2\text{diam}(\mathbb{T}^{d})$, we have $\int_{|x-y|>R-\delta_{n}}(1\wedge|x-y|^{2})\eta(x,y)d\pi(y)=0$
since the domain of integration is empty. This shows that 
\[
\sum_{\bar{y}\in\{x_{j}^{n}\}_{j=1}^{n^{d}}\cap A_{R}(\bar{x})\backslash\{\bar{x}\}}\int_{T_{n}^{-1}(\bar{y})}(1\wedge|x-y|^{2})\eta(x,y)d\pi(y)=0
\]
trivially. 

Next, consider the case where instead $1/R\geq\delta_{n}$. In this
case, $1/R+\delta_{n}\leq1/(R/2)$. Likewise, if $R\geq2$, we have
that $R-\delta_{n}\geq R-1/R\geq R/2$. Together, these imply that
whenever $\bar{y}\in\{x_{j}^{n}\}_{j=1}^{n^{d}}\cap A_{R}(\bar{x})\backslash\{\bar{x}\}$
and $x\in T_{n}^{-1}(\bar{x})$ and $y\in T_{n}^{-1}(\bar{y})$, we
have that $y\in A_{R/2}(x)$. Consequently, for all $x\in T_{n}^{-1}(\bar{x})$
we have 
\[
\sum_{\bar{y}\in\{x_{j}^{n}\}_{j=1}^{n^{d}}\cap A_{R}(\bar{x})\backslash\{\bar{x}\}}\int_{T_{n}^{-1}(\bar{y})}(1\wedge|x-y|^{2})\eta(x,y)d\pi(y)\leq\int_{A_{R/2}(x)}(1\wedge|x-y|^{2})\eta(x,y)d\pi(y).
\]

Finally, combining the two cases, we deduce that: for $R>\max\{2,2\text{diam}(\mathbb{T}^{d})\}=2\text{diam}(\mathbb{T}^{d})$,
it holds uniformly in $n$ that: for all $\bar{x}\in\{x_{j}^{n}\}_{j=1}^{n^{d}}$,
and $x\in T_{n}^{-1}(\bar{x})$,
\begin{align*}
\sum_{\bar{y}\in\{x_{j}^{n}\}_{j=1}^{n^{d}}\cap A_{R}(\bar{x})\backslash\{\bar{x}\}}\int_{T_{n}^{-1}(\bar{y})}(1\wedge|x-y|^{2})\eta(x,y)d\pi(y) & \leq\max\left\{ 0,\int_{A_{R/2}(x)}(1\wedge|x-y|^{2})\eta(x,y)d\pi(y)\right\} \\
 & =\int_{A_{R/2}(x)}(1\wedge|x-y|^{2})\eta(x,y)d\pi(y).
\end{align*}
Consequently, 
\begin{multline*}
\sum_{\bar{y}\in\{x_{j}^{n}\}_{j=1}^{n^{d}}\cap A_{R}(\bar{x})\backslash\{\bar{x}\}}(1\wedge|\bar{x}-\bar{y}|^{2})\eta_{n}(\bar{x},\bar{y})\pi_{n}(\bar{y})\\
\begin{aligned} & \leq4\frac{1}{\pi_{n}(\bar{x})}\int_{T_{n}^{-1}(\bar{x})}\left(\int_{A_{R/2}(x)}(1\wedge|x-y|^{2})\eta(x,y)d\pi(y)\right)d\pi(x)\\
 & \leq4\frac{1}{\pi_{n}(\bar{x})}\int_{T_{n}^{-1}(\bar{x})}\left(\sup_{x\in\text{supp}(\pi)}\int_{A_{R/2}(x)}(1\wedge|x-y|^{2})\eta(x,y)d\pi(y)\right)d\pi(x)\\
 & =4\sup_{x\in\text{supp}(\pi)}\int_{A_{R/2}(x)}(1\wedge|x-y|^{2})\eta(x,y)d\pi(y).
\end{aligned}
\end{multline*}
Quantifying over all $\bar{x}$, this shows that 
\[
\sup_{n\in\mathbb{N}}\sup_{\bar{x}\in\text{supp}(\pi_{n})}\sum_{\bar{y}\in\{x_{j}^{n}\}_{j=1}^{n^{d}}\cap A_{R}(\bar{x})\backslash\{\bar{x}\}}(1\wedge|\bar{x}-\bar{y}|^{2})\eta_{n}(\bar{x},\bar{y})\pi_{n}(\bar{y})\leq4\sup_{x\in\text{supp}(\pi)}\int_{A_{R/2}(x)}(1\wedge|x-y|^{2})\eta(x,y)d\pi(y).
\]
Since by assumption $\lim_{R/2\rightarrow\infty}\sup_{x\in\text{supp}(\pi)}\int_{A_{R/2}(x)}(1\wedge|x-y|^{2})\eta(x,y)d\pi(y)=0$,
this establishes that 
\[
\lim_{R\rightarrow\infty}\sup_{n\in\mathbb{N}}\sup_{\bar{x}\in\text{supp}(\pi_{n})}\sum_{\bar{y}\in\{x_{j}^{n}\}_{j=1}^{n^{d}}\cap A_{R}(\bar{x})\backslash\{\bar{x}\}}(1\wedge|\bar{x}-\bar{y}|^{2})\eta_{n}(\bar{x},\bar{y})\pi_{n}(\bar{y})=0.
\]
Moreover, since $\int_{A_{R}(x)}(1\wedge|x-y|^{2})\eta(x,y)d\pi(y)\leq\int_{A_{R/2}(x)}(1\wedge|x-y|^{2})\eta(x,y)d\pi(y)$,
we actually have that $4\sup_{x\in\text{supp}(\pi)}\int_{A_{R/2}(x)}(1\wedge|x-y|^{2})\eta(x,y)d\pi(y)$
controls both 
\[
\sup_{x\in\text{supp}(\pi)}\int_{A_{R}(x)}(1\wedge|x-y|^{2})\eta(x,y)d\pi(y)
\]
and 
\[
\sup_{n\in\mathbb{N}}\sup_{\bar{x}\in\text{supp}(\pi_{n})}\sum_{\bar{y}\in\{x_{j}^{n}\}_{j=1}^{n^{d}}\cap A_{R}(\bar{x})\backslash\{\bar{x}\}}(1\wedge|\bar{x}-\bar{y}|^{2})\eta_{n}(\bar{x},\bar{y})\pi_{n}(\bar{y}).
\]
This establishes that the entire family $\{\eta_{n},\pi_{n}\}_{n\in\mathbb{N}}\cup\{\eta,\pi\}$
satisfies Assumption \ref{assu:eta=000020properties} (3), as desired.

It remains only to establish the second part of the statement of the
lemma. Namely, that if $\eta$ is continuous on $G$, then $\eta_{n}$
can be extended to a continuous function on all of $G$ in such a
way that $\eta_{n}$ converges to $\eta$ uniformly on compact subsets
of $G$. (Clearly, the preceding results still hold regardless of
how we extend $\eta_{n}$, since we are only altering $\eta_{n}$
off of the support of $\pi_{n}$.) Accordingly, let $K\subset\subset G$.
Let $D=\text{dist}(K,\{x=y\})$. Take $n$ to be sufficiently large
that $D\geq3\delta_{n}=3\text{Diam}(T_{n}^{-1}(x_{j}^{n}))$. In this
case, For each $(x_{j}^{n},x_{k}^{n})\in K$, it holds that 
\[
T_{n}^{-1}(x_{j}^{n})\times T_{n}^{-1}(x_{k}^{n})=T_{n}^{-1}(x_{j}^{n})\times T_{n}^{-1}(x_{k}^{n})\cap\{|x-y|\geq\delta_{n}/2\}
\]
and so we have 
\[
\eta_{n}(x_{j}^{n},x_{k}^{n})=\frac{1}{\pi_{n}(x_{j}^{n})\cdot\pi_{n}(x_{k}^{n})}\iint_{T_{n}^{-1}(x_{j}^{n})\times T^{-1}(x_{k}^{n})}\eta(x,y)d\pi(x)d\pi(y).
\]
Now, $\eta(x,y)$ is continuous, and so restricted to $K$, $\eta(x,y)$
is uniformly continuous. Let $\Phi^{K}(\varepsilon)$ denote a modulus
of uniform continuity for $\eta$ restricted to $K$, in the following
sense:
\[
\forall(x,y),(x^{\prime},y^{\prime})\in K,|x-x^{\prime}|+|y-y^{\prime}|<\Phi^{K}(\varepsilon)\implies|\eta(x,y)-\eta(x^{\prime},y^{\prime})|<\varepsilon.
\]
Accordingly, let $\varepsilon>0$ be arbitrary. It follows that, if
$\delta_{n}<\frac{1}{2}\Phi^{K}(\varepsilon)$, we have 
\[
\left|\eta(x_{j}^{n},x_{k}^{n})-\frac{1}{\pi_{n}(x_{j}^{n})\cdot\pi_{n}(x_{k}^{n})}\iint_{T_{n}^{-1}(x_{j}^{n})\times T^{-1}(x_{k}^{n})}\eta(x,y)d\pi(x)d\pi(y)\right|<\varepsilon.
\]
Moreover, given any $(x,y)\in K$, we have that the closest $(x_{j}^{n},x_{k}^{n})$
to $(x,y)$ satisfies $|x-x_{j}^{n}|+|y-x_{k}^{n}|\leq2\delta_{n}<\Phi^{K}(\varepsilon)$,
and so we also have that 
\[
\left|\eta(x,y)-\frac{1}{\pi_{n}(x_{j}^{n})\cdot\pi_{n}(x_{k}^{n})}\iint_{T_{n}^{-1}(x_{j}^{n})\times T^{-1}(x_{k}^{n})}\eta(x,y)d\pi(x)d\pi(y)\right|<2\varepsilon.
\]

It remains to extend $\eta_{n}(x_{j}^{n},x_{k}^{n})$ to all of $G$
in such a way that $\eta(x,y)$ and $\eta_{n}(x,y)$ are uniformly
close on $K$.  We do so explicitly using a \emph{singular kernel
interpolator}, as follows. (We learned of this construction from \cite{belkin2019does};
compare also the broadly similar construction in \cite[Section 5.2]{chen2013discrete}.)
Let $a>2$ and define
\[
K(z)=\begin{cases}
|z|^{-a}(1-|z|^{2}) & |z|\leq1\\
0 & |z|>1.
\end{cases}
\]
Then $K(z)$ is continuous except at $z=0$. Given a function $f$
defined at points $p_{1},\ldots,p_{k}$in $\mathbb{R}^{D}$ we set
\[
f(p)=\frac{\sum_{i=1}^{k}f(p_{i})K_{\tilde{\varepsilon}}(\Vert p-p_{i}\Vert)}{\sum_{i=1}^{k}K_{\tilde{\varepsilon}}(\Vert p-p_{i}\Vert)}.
\]
where $\tilde{\varepsilon}$ is a bandwidth parameter and $K_{\tilde{\varepsilon}}(\cdot):=\frac{1}{\tilde{\varepsilon}^{D}}K\left(\frac{\cdot}{\tilde{\varepsilon}}\right)$.
Note that $f(p)$ is bounded in between the minimum and maximum values
of $f(p_{i})$ for $\Vert p_{i}-p\Vert<\tilde{\varepsilon}$ and $f(p)$
is continuous and well-defined whenever there is at least one $p_{i}$
such that $\Vert p_{i}-p\Vert<\tilde{\varepsilon}$. (In particular,
$f(p)\rightarrow f(p_{i})$ as $p\rightarrow p_{i}$ since $K$ is
singular at 0, so this really is an interpolation.)

Applying this device to our particular case, we extend $\eta_{n}$
to all of $\mathbb{T}^{d}\times\mathbb{T}^{d}$ by defining 
\[
\eta_{n}(x,y)=\frac{\sum\eta_{n}(x_{j}^{n},x_{k}^{n})K_{\tilde{\varepsilon}}(|x-x_{j}^{n}|+|y-x_{k}^{n}|)}{\sum K_{\tilde{\varepsilon}}(|x-x_{j}^{n}|+|y-x_{k}^{n}|)}
\]
where the summation ranges over all $(x_{j}^{n},x_{k}^{n})$ such
that $|x-x_{j}^{n}|+|y-x_{k}^{n}|<\tilde{\varepsilon}$. This definition
extends $\eta_{n}(x,y)$ to all of $\mathbb{T}^{d}\times\mathbb{T}^{d}\backslash\{x=y\}$
provided that $\tilde{\varepsilon}$ is strictly bigger than the minimal
distance between distinct points $x_{j}^{n}$ and $x_{k}^{n}$, which
in turn is less than $2\delta_{n}$. Accordingly, we assume that $2\delta_{n}<\tilde{\varepsilon}$
below.

We have already observed that $\eta_{n}(x,y)$ is bounded in between
$\min\eta(x_{j}^{n},x_{k}^{n})$ and $\max\text{\ensuremath{\eta}(\ensuremath{x_{j}^{n}},\ensuremath{x_{k}^{n}})}$
where the min and max run over the $\tilde{\varepsilon}$ ball around
$(x,y)$. Write $(x_{j,\text{min}}^{n},x_{k,\text{min}}^{n})$ and
$(x_{j,\text{max}}^{n},x_{k,\text{max}}^{n})$ for these values respectively,
and note that $|x_{j,\text{min}}^{n}-x_{j,\text{max}}^{n}|+|x_{k,\text{min}}^{n}-x_{k,\text{max}}^{n}|<2\tilde{\varepsilon}$
by the triangle inequality. It follows that if $2\tilde{\varepsilon}<\Phi^{K}(\varepsilon)$
we have 
\[
|\eta(x_{j,\text{min}}^{n},x_{k,\text{min}}^{n})-\eta(x_{j,\text{max}}^{n},x_{k,\text{max}}^{n})|<\varepsilon.
\]
Combining this with the fact, seen above, that $\left|\eta(x_{j}^{n},x_{k}^{n})-\eta_{n}(x_{j}^{n},x_{k}^{n})\right|<\varepsilon$
whenever $\delta_{n}<\frac{1}{2}\Phi^{K}(\varepsilon)$, for all pairs
$(x_{j}^{n},x_{k}^{n})\in K$ we have that
\[
|\eta_{n}(x_{j,\text{min}}^{n},x_{k,\text{min}}^{n})-\eta_{n}(x_{j,\text{max}}^{n},x_{k,\text{max}}^{n})|<3\varepsilon.
\]
Now take $(x_{j}^{n},x_{k}^{n})$ to be the closest such point to
$(x,y)$. By construction we have $|x-x_{j}^{n}|+|y-x_{k}^{n}|<2\delta_{n}<\tilde{\varepsilon}$
and so, since $\eta_{n}(x,y)$ and $\eta_{n}(x_{j}^{n},x_{k})$ are
both contained in the interval $[\eta_{n}(x_{j,\text{min}}^{n},x_{k,\text{min}}^{n}),\eta_{n}(x_{j,\text{max}}^{n},x_{k,\text{max}}^{n})]$
we see that
\[
|\eta_{n}(x,y)-\eta_{n}(x_{j}^{n},x_{k}^{n})|<3\varepsilon.
\]
At the same time, we already saw that $|\eta(x,y)-\eta_{n}(x_{j}^{n},x_{k}^{n})|<2\varepsilon$
so altogether we have, for all $(x,y)\in K$, that
\[
|\eta(x,y)-\eta_{n}(x,y)|<5\varepsilon
\]
under the conditions that: $2\tilde{\varepsilon}<\Phi^{K}(\varepsilon)$
and $\delta_{n}<\frac{1}{2}\Phi^{K}(\varepsilon)$ and $2\delta_{n}<\tilde{\varepsilon}$.
In particular, we can first pick $K$ arbitrarily, then pick $\varepsilon>0$
arbitrarily, then pick $\tilde{\varepsilon}<\frac{1}{2}\Phi^{K}(\varepsilon)$,
then pick $n$ sufficiently large (hence $\delta_{n}$ sufficiently
small) that $\text{dist}(K,\{x=y\})\geq3\delta_{n}$ and $\tilde{\varepsilon}>2\delta_{n}$.
Then for all such $n$ we have $|\eta(x,y)-\eta_{n}(x,y)|<5\varepsilon$
on $K$, in other words $\eta_{n}$ converges uniformly on $K$. Since
$K$ was arbitrary the claim follows.
\end{proof}
More generally, we have the following sufficient condition for discrete-to-continuum
convergence, which is basically a special case of Proposition \ref{prop:stability}
together with Proposition \ref{prop:weak-soln-vs-gradient-flow}.
\begin{prop}
Let $\pi_{n}$ be a sequence of probability measures supported on
finitely many atoms in $X$, and let $\eta_{n}$ be defined on the
support of $\pi_{n}$. Suppose that $\pi_{n}\rightharpoonup\pi$ with
$\pi\ll\text{Leb}^{d}$. Suppose also that each $\eta_{n}$ can be
extended to a continuous function all of $G=X\times X\backslash\{x=y\}$
in such a way that $\eta_{n}$ converges to some continuous $\eta$
uniformly on compact subsets of $G$, and such that $\{\eta_{n},\pi_{n}\}_{n\in\mathbb{N}}\cup\{\eta,\pi\}$
satisfies Assumption \ref{assu:eta=000020properties} (1-3). 

Then, curves of maximal slope for $\mathcal{H}(\cdot\mid\pi_{n})$
with respect to $\mathcal{W}_{\eta_{n},\pi_{n}}$ converge to curves
of maximal slope for $\mathcal{H}(\cdot\mid\pi)$ with respect to
$\mathcal{W}_{\eta,\pi}$, in the same sense as in Proposition \ref{prop:stability};
and if $\sqrt{\mathcal{I}(\cdot\mid\pi)}$ is an a.e. upper gradient
for $\mathcal{H}(\cdot\mid\pi)$, and same assumptions as Proposition
\ref{prop:one=000020sided=000020chain=000020rule} hold, then the
resulting limit is also an entropy weak solution for $(\dagger)$.
\end{prop}

\section{Nonlocal log-Sobolev inequality and contractivity rate\protect\label{sec:contractivity}}

We recall from Section \ref{sec:Identifying} that if $\sqrt{\mathcal{I}_{\eta}(\cdot\mid\pi)}$
is an a.e. upper gradient for $\mathcal{H}(\cdot\mid\pi)$, the chain
rule holds, and $\rho_{t}$ is a curve of maximal slope for $\mathcal{H}(\cdot\mid\pi)$
with respect to $\mathcal{W}_{\eta,\pi}$, then (thanks to the equality
condition of Young's inequality) it actually holds that 
\[
|\dot{\rho}_{t}|_{\mathcal{W}_{\eta,\pi}}^{2}=\mathcal{I}_{\eta}(\rho_{t}\mid\pi)\text{ }t-a.s.
\]
and so we actually have the \emph{entropy production formula} 
\[
\mathcal{H}(\rho_{0}\mid\pi)-\mathcal{H}(\rho_{T}\mid\pi)=\int_{0}^{T}\mathcal{I}_{\eta}(\rho_{t}\mid\pi)dt.
\]
It therefore follows immediately from Gronwall's inequality that the
inequality
\[
\mathcal{H}(\rho\mid\pi)\leq\frac{1}{C_{LS}}\mathcal{I}_{\eta}(\rho\mid\pi)\qquad\forall\rho\in\mathcal{P}(X)
\]
corresponds to an entropy-entropy production inequality for $(\dagger)$,
and so ensures exponential convergence to equilibrium. We call this
the ``nonlocal log-Sobolev inequality'' because $\mathcal{I}_{\eta}(\rho\mid\pi)$
is a nonlocal analogue of the Fisher information.\footnote{In the Dirichlet forms literature, the preferred terminology is that
this is the ``modified log-Sobolev inequality'' (MLSI) instead.
This is because in general, for a Dirichlet form $\mathcal{E}(u,v)$
on the space $L^{2}(X,\pi)$, we say that $\mathcal{E}$ satisfies
the MLSI provided that $\mathcal{H}(\rho\mid\pi)\leq C^{-1}\mathcal{E}\left(\frac{d\rho}{d\pi},\log\frac{d\rho}{d\pi}\right)$
\cite{bobkov2006modified}. In our case this is an equivalent description
of the ``nonlocal log-Sobolev inequality'' above, provided that
$\mathcal{E}$ is chosen to be the nonlocal Dirichlet form $\mathcal{E}_{\eta,\pi}$
given in Section \ref{sec:nonlocal-Dirichlet} immediately below. } In this section,  in the case where $X=\mathbb{T}^{d}$, we give
an explicit sufficient criterion on $\eta$ which guarantees that
the nonlocal log-Sobolev inequality holds.  We do not attempt a serious
study of this inequality here, rather we simply point out that this
inequality is at least true in some explicit cases.

The following proposition reproduces a remarkably brief argument given
in \cite{wang2014simple}.
\begin{prop}
Let $X=\mathbb{T}^{d}$. Suppose that there exists some constant $C>0$
such that, for all $x,y\in\mathbb{T}^{d}$, $\eta(x,y)\geq C$. Then
$\mathcal{H}(\rho\mid\pi)\leq\frac{1}{C_{LS}}\mathcal{I}(\rho\mid\pi)$
with $C_{LS}=C$.
\end{prop}

\begin{rem*}
In the case where $X$ is unbounded, if one assumes that $\eta(x,y)\geq C$,
then Assumption \ref{assu:eta=000020properties} (3) is not satisfied.
We used Assumption \ref{assu:eta=000020properties} (3) to prove compactness
for AC curves in Proposition \ref{prop:compactness-nce}, among other
results. In particular, the proof we state below works just as well
for $X=\mathbb{R}^{d}$; but in this case, other arguments in this
article break down. 
\end{rem*}
\begin{proof}
First, note that since $\mu$ is a probability measure, 
\[
\mathcal{H}(\mu\mid\pi)=\int\frac{d\mu}{d\pi}\log\frac{d\mu}{d\pi}d\pi=\int\frac{d\mu}{d\pi}\log\frac{d\mu}{d\pi}d\pi-\int\frac{d\mu}{d\pi}d\pi\log\left(\int\frac{d\mu}{d\pi}d\pi\right).
\]
By Jensen's inequalty, $\log\left(\int\frac{d\mu}{d\pi}d\pi\right)\geq\int\log\frac{du}{d\pi}d\pi$.
Hence 
\begin{align*}
\mathcal{H}(\mu\mid\pi) & \leq\int\frac{d\mu}{d\pi}\log\frac{d\mu}{d\pi}d\pi-\int\frac{d\mu}{d\pi}d\pi\left(\int\log\frac{d\mu}{d\pi}d\pi\right)\\
 & =\frac{1}{2}\iint\left(\frac{d\mu}{d\pi}(y)-\frac{d\mu}{d\pi}(x)\right)\left(\log\frac{d\mu}{d\pi}(y)-\log\frac{d\mu}{d\pi}(x)\right)d\pi(x)d\pi(y)\\
 & =\frac{1}{2}\iint\overline{\nabla}\frac{d\mu}{d\pi}(x,y)\overline{\nabla}\log\frac{d\mu}{d\pi}(x,y)d\pi(x)d\pi(y).
\end{align*}
Now simply observe that if $C\leq\eta(x,y)$, it follows immediately
that 
\begin{align*}
\frac{1}{2}\iint\overline{\nabla}\frac{d\mu}{d\pi}(x,y)\overline{\nabla}\log\frac{d\mu}{d\pi}(x,y)d\pi(x)d\pi(y) & \leq\frac{1}{2C}\iint\overline{\nabla}\frac{d\mu}{d\pi}(x,y)\overline{\nabla}\log\frac{d\mu}{d\pi}(x,y)\eta(x,y)d\pi(x)d\pi(y)\\
 & =\frac{1}{C}\mathcal{I}_{\eta}(\mu\mid\pi)
\end{align*}
and hence $\mathcal{H}(\mu\mid\pi)\leq\frac{1}{C}\mathcal{I}_{\eta}(\mu\mid\pi)$
as desired.
\end{proof}
Finally, in the following example, we show how the PDE \ref{eq:intro-nonlocal-fokker-planck}
can be made precise and cast as a gradient flow of $\mathcal{H}(\cdot\mid\pi)$,
with an explicit lower bound on the contractivity rate depending on
the geometry of $\pi$, as advertised in the introduction. 
\begin{example}
 Consider the case where $X=\mathbb{T}^{d}$, $\frac{d\pi}{dx}=\frac{1}{C_{V}}e^{-V}$,
and $\eta(x,y)=\frac{1}{c_{V}}\left(e^{V(x)}+e^{V(y)}\right)\kappa(|x-y|)$;
such a choice of kernel has previously been studied, in the context
of nonlocal Dirichlet form theory, in \cite{chen2017weighted,wang2014simple}.
Note that \ref{assu:eta=000020properties} is satisfied in the case
where $V$ is continuous on $\mathbb{T}^{d}$, in particular the shift
constant $S$ can be estimated by bounding $\frac{e^{V(x)}+e^{V(y)}}{e^{V(x+z)}+e^{V(y+z)}}$.
The relevant continuity equation is
\[
\partial_{t}\left(\frac{d\mu_{t}}{d\pi}\right)(x)-\int_{\mathbb{T}^{d}}\overline{\nabla}\psi_{t}(x,y)\theta\left(\frac{d\mu_{t}}{d\pi}(x),\frac{d\mu_{t}}{d\pi}(y)\right)\tilde{J}(x,dy)=0
\]
so with $\tilde{J}(x,dy)=\frac{1}{c_{V}}\left(e^{V(x)}+e^{V(y)}\right)\kappa(|x-y|)d\pi(y)$
and $\psi_{t}=\log\left(\frac{d\mu_{t}}{d\pi}\right)$ (this is the
tangent potential for $\mathcal{H}(\cdot\mid\pi)$ by same computation
as in \cite{maas2011gradient}) we get 
\[
\partial_{t}\left(\frac{d\mu_{t}}{d\pi}\right)(x)-\int_{\mathbb{T}^{d}}\overline{\nabla}\left(\frac{d\mu_{t}}{d\pi}\right)(x,y)\frac{1}{c_{V}}\left(e^{V(x)}+e^{V(y)}\right)\kappa(|x-y|)d\pi(y)=0
\]
and, using $\pi=c_{V}e^{-V}dx$ and the Radon-Nikodym theorem, this
simplifies to
\[
\partial_{t}\left(\frac{d\mu_{t}}{d\pi}\right)(x)-\int_{\mathbb{T}^{d}}\overline{\nabla}\left(\frac{d\mu_{t}}{d\pi}\right)(x,y)\left(e^{V(x)-V(y)}+1\right)\kappa(|x-y|)dy=0.
\]
Finally, we replace $\frac{d\mu_{t}}{d\pi}$ with $\frac{d\mu_{t}}{dx}$.
Putting $\rho_{t}=\frac{d\mu_{t}}{dx}$ and using $\frac{d\mu_{t}}{d\pi}=\frac{d\mu_{t}}{dx}\frac{dx}{d\pi}$
we have 
\[
\frac{e^{V(x)}}{c_{V}}\partial_{t}\rho_{t}(x)-\int_{\mathbb{T}^{d}}\left(\frac{e^{V(y)}}{c_{V}}\rho_{t}(y)-\frac{e^{V(x)}}{c_{V}}\rho_{t}(x)\right)\left(e^{V(x)-V(y)}+1\right)\kappa(|x-y|)dy=0.
\]
Thus, dividing through by $\frac{e^{V(x)}}{c_{V}}$, we have 
\[
\partial_{t}\rho_{t}(x)-\int_{\mathbb{T}^{d}}\left(e^{V(y)-V(x)}\rho_{t}(y)-\rho_{t}(x)\right)\left(e^{V(x)-V(y)}+1\right)\kappa(|x-y|)dy=0.
\]
Note that if $C_{LS}\leq\frac{1}{c_{V}}\left(e^{V(x)}+e^{V(y)}\right)\kappa(|x-y|)$,
then the nonlocal log-Sobolev inequality is satisfied with constant
$C_{LS}$. %
\begin{comment}
The term $e^{V(x)-V(y)}+1$ has the following interpretation: compared
to the usual fractional heat equation, we have modified the jump kernel
to preferentially send mass from regions where $V$ is large, to regions
where $V$ is small. Formally, for a particle method, one might replace
this with Metropolis-type update step, where given a particle at $x_{0}$,
you sample $y_{0}$ from $\eta(x,y)dy$, then do accept/reject at
some rate depending on $V(x)/V(y)$.
\end{comment}
{} This shows that the rate of convergence to equilibrium is lower bounded
in terms of the \emph{distance between global minima }of $e^{V(\cdot)}$
(hence global maxima of $\frac{d\pi}{dx}$), as measured by $\kappa$,
as well as the value of $\max_{x\in\mathbb{T}^{d}}\frac{d\text{\ensuremath{\pi}}}{dx}$.

In particular, we can decompose the PDE as follows: (here we abuse
notation and write $\pi(x)$ to denote the Lebesgue density of the
equilibrium measure)
\[
\partial_{t}\rho_{t}(x)=\underbrace{\int_{\mathbb{T}^{d}}\overline{\nabla}\rho_{t}(x,y)\kappa(|x-y|)dy}_{\text{flat nonlocal diffusion term}}+\underbrace{\int_{\mathbb{T}^{d}}\left(\rho_{t}(y)\frac{\pi(x)}{\pi(y)}-\rho_{t}(x)\frac{\pi(y)}{\pi(x)}\right)\kappa(|x-y|)dy}_{\text{weighted nonlocal divergence term}}.
\]
Particularizing further, if we take $\kappa$ to be the fractional
heat kernel  on $\mathbb{T}^{d}$ for some $s\in(0,2)$, the resulting
equation can be written as 
\[
\partial_{t}\rho(x)=-(-\Delta)^{s/2}\rho_{t}(x)+\int_{\mathbb{T}^{d}}\left(\rho_{t}(y)\frac{\pi(x)}{\pi(y)}-\rho_{t}(x)\frac{\pi(y)}{\pi(x)}\right)\kappa(|x-y|)dy.
\]
Lastly, we mention that our results from the previous sections show
that: for this equation, weak solutions exist and are unique, are
curves of maximal slope for the relative entropy $\mathcal{H}(\cdot\mid\pi)$,
and are stable (in the sense of evolutionary $\Gamma$-convergence)
with respect to uniform perturbations of $\left(\frac{d\pi}{dx}\right)^{-1}$.
\end{example}

\section{Relation to nonlocal Dirichlet forms\protect\label{sec:nonlocal-Dirichlet}}

In this section, we briefly explain how our nonlocal transport-theoretic
approach to equation ($\dagger$) relates to a more classical perspective
drawing on Dirichlet form theory. We do not explain all the required
notions regarding Dirichlet forms in full detail; see \cite{fukushima2011dirichlet}
for general background.

Let $\pi\in\mathcal{P}(X)$. We define the symmetric bilinear form
$\mathcal{E}_{\eta,\pi}(u,v)$ with domain contained in $L^{2}(X,\pi)$
as follows:
\begin{align*}
\mathcal{E}_{\eta,\pi}(u,v) & :=\frac{1}{2}\iint_{X\times X\backslash\{x=y\}}(u(x)-u(y))(v(x)-v(y))\eta(x,y)d\pi(x)d\pi(y)\\
\mathcal{D}[\mathcal{E}_{\eta,\pi}] & :=\left\{ u\in L^{2}(X,\pi):\mathcal{E}_{\eta,\pi}(u,u)<\infty\right\} .
\end{align*}
Likewise we can define the ``$\mathcal{E}_{\eta,\pi}^{1}$ norm''
on functions $u\in L^{2}(X,\pi)$ by 
\[
\sqrt{\mathcal{E}_{\eta,\pi}^{1}(u,u)}:=\left(\mathcal{E}_{\eta,\pi}(u,u)+\Vert u\Vert_{L^{2}(\pi)}^{2}\right)^{1/2}.
\]
Let $\mathcal{F}$ be any subset of $\mathcal{D}[\mathcal{E}_{\eta,\pi}]$
which is closed with respect to the $\mathcal{E}_{\eta,\pi}^{1}$
norm. Then we say that $(\mathcal{E}_{\eta,\pi},\mathcal{F})$ is
a \emph{Dirichlet space} and $\mathcal{E}_{\eta,\pi}$ is a \emph{Dirichlet
form} provided that $\mathcal{E}_{\eta,\pi}$ is also ``Markovian'',
a technical condition we do not reproduce here but which is given
as condition ($\mathcal{E}.4$) in \cite{fukushima2011dirichlet}. 

We note that for $u$ to belong to $\mathcal{D}[\mathcal{E}_{\eta,\pi}]$,
under Assumption \ref{assu:eta=000020properties}(2) it suffices that
$u\in C^{1}(X)\cap L^{2}(X;\pi)$, since 
\begin{align*}
\iint_{X\times X\backslash\{x=y\}}(u(x)-u(y))^{2}\eta(x,y)d\pi(x)d\pi(y) & \leq\Vert u\Vert_{C^{1}}^{2}\iint_{X\times X\backslash\{x=y\}}(1\wedge|x-y|)^{2}\eta(x,y)d\pi(x)d\pi(y)
\end{align*}
and 
\begin{align*}
\iint_{X\times X\backslash\{x=y\}}(1\wedge|x-y|)^{2}\eta(x,y)d\pi(x)d\pi(y) & \leq\sup_{y\in\text{supp}(\pi)}\left(\int(1\wedge|x-y|)^{2}\eta(x,y)d\pi(x)\right)<\infty.
\end{align*}
In particular $\mathcal{D}[\mathcal{E}_{\eta,\pi}]$ automatically
includes the smooth functions of compact support. In fact, since $\pi$
is a probability measure, we actually have $C^{1}(X)\subset L^{2}(X;\pi)$
in our case. Additionally, in \cite[Example 1.2.4]{fukushima2011dirichlet}
it is verified that $\mathcal{E}_{\eta,\pi}$ is ``Markovian'' and
so is a Dirichlet form on any $\mathcal{E}_{\eta,\pi}^{1}$-closed
subset of $\mathcal{D}[\mathcal{E}_{\eta,\pi}]$, under the conditions
that: for any $\varepsilon>0$, $x\mapsto\int_{X\backslash B_{\varepsilon}(x)}\eta(x,y)d\pi(y)$
is locally integrable (which is implied by our \ref{assu:eta=000020properties}(3)),
and that for any compact $K\subset X$, $\iint_{K\times K\backslash\{x=y\}}(|x-y|)^{2}\eta(x,y)d\pi(x)d\pi(y)<\infty$,
which is implied by our Assumption \ref{assu:eta=000020properties}(2+4).
Moreover it is shown that $\mathcal{D}[\mathcal{E}_{\eta,\pi}]$ itself
is closed with respect to the $\mathcal{E}_{\eta,\pi}^{1}$ norm in
this case, so we can take $\mathcal{F}=\mathcal{D}[\mathcal{E}_{\eta,\pi}]$
for instance.

The Dirichlet form $\mathcal{E}_{\eta,\pi}$ gives rise to the nonlocal
parabolic equation\emph{ 
\[
\frac{d}{dt}\langle u(t),\cdot\rangle_{L^{2}(\pi)}=-\mathcal{E}_{\eta,\pi}(u(t),\cdot)
\]
}which we interpret as an ODE on the subspace $\mathcal{F}$ of $L^{2}(\pi)$.
Weak solutions to this equation are said to be \emph{caloric} (see
e.g. \cite{barlow2012equivalence,grigor2024parabolic} for this terminology),
as per the following definition.
\begin{defn}
Given the Dirichlet space $(\mathcal{E}_{\eta,\pi},\mathcal{F})$,
we say that a function $u(t):[0,T]\rightarrow\mathcal{F}$ is \emph{caloric}
for $\mathcal{E}_{\eta,\pi}$ provided that\emph{ }for all $\varphi\in\mathcal{F}$,
$\frac{d}{dt}\langle u(t),\varphi\rangle_{L^{2}(\pi)}$ exists for
all $t\in[0,T]$, and\emph{
\[
\frac{d}{dt}\int_{X}u(t,x)\varphi(x)d\pi+\frac{1}{2}\iint_{X\times X\backslash\{x=y\}}(u(t,x)-u(t,y))(\varphi(x)-\varphi(y))\eta(x,y)d\pi(x)d\pi(y)=0.
\]
}
\end{defn}

\begin{rem*}
Caloric functions are less commonly seen in the literature than semigroups
generated by Dirichlet forms; caloric functions are indeed a weaker
notion, see discussion in \cite{barlow2012equivalence,grigor2008off}.
The idea is that the generator $L$ of the Dirichlet form $\mathcal{E}$
is given by duality by $\langle Lu,\varphi\rangle_{L^{2}(\pi)}=-\mathcal{E}(u,\varphi)$
and so the caloric functions are a notion of weak solutions for the
evolution equation $\frac{d}{dt}u(t)=Lu(t)$.
\end{rem*}
Notice that this definition, is almost, but not quite, the same as
our Definition \ref{def:nde-weak-solution} of weak solutions to ($\dagger$),
as the class of test functions and notion of time derivative differ.
This is analogous to how $2$-Wasserstein gradient flows of the entropy
are measure-valued weak solutions to the heat equation, which also
has a subtly different $L^{2}$ theory of weak solutions.

The following proposition addresses whether caloric functions for
$\mathcal{E}_{\eta,\pi}$ are also weak solutions to $(\dagger)$
in the sense of Definition \ref{def:nde-weak-solution}, or vice versa.
Combining this with \ref{prop:weak-soln-vs-gradient-flow} allows
us to then relate caloric functions for $\mathcal{E}_{\eta,\pi}$
and curves of maximal slope for $\mathcal{H}(\cdot\mid\pi)$ w.r.t.
$\mathcal{W}_{\eta,\pi}$. 

Aside from being of analytic interest, this proposition allows for
a \emph{probabilistic} interpretation of curves of maximal slope for
$\mathcal{H}(\cdot\mid\pi)$. In the introduction, we mentioned that
equation $(\dagger)$ comes with a \emph{formal} interpretation as
the time-marginals of a pure-jump process, however that interpretation
can only directly be made rigorous in the case when $\sup_{x}\int_{X}\eta(x,y)d\pi(y)<\infty$
(see e.g. \cite[Thm. 9.3.27]{ccinlar2011probability}). However, in
the Dirichlet forms literature, existence theorems for underlying
stochastic processes hold much more generally, see especially \cite[Thm. 7.2.1]{fukushima2011dirichlet}.
\begin{prop}
\label{prop:dirichlet-caloric-equivalent}Let $\mathcal{E}_{\eta,\pi}$
be defined as above, and let $\mathcal{F}\subset\mathcal{D}[\mathcal{E}_{\eta,\pi}]$
be $\mathcal{E}_{\eta,\pi}^{1}$-closed. Assume $C^{1}(X)\subset\mathcal{F}$,
and suppose Assumption \ref{assu:eta=000020properties}(1-4) holds. 

(1) Let $u_{0}\in\mathcal{F}$ satisfy $\int_{X}u_{0}d\pi=1$. Then,
given any nonnegative caloric function $u(t)$ beginning at $u_{0}$,
the time-dependent probability measure $\rho_{t}$ given by $u(t)=\frac{d\rho_{t}}{d\pi}$
is a weak solution to equation ($\dagger$) in the sense of Definition
\ref{def:nde-weak-solution}. 

(2) Suppose that $\sqrt{\mathcal{I}_{\eta}(\cdot\mid\pi)}$ is an
a.e. upper gradient for $\mathcal{H}(\cdot\mid\pi)$ w.r.t. $\mathcal{W}_{\eta,\pi}$,
and that $\rho_{0}\in\mathcal{P}(X)$ with $\mathcal{H}(\rho_{0}\mid\pi)<\infty$
and $\frac{d\rho_{0}}{d\pi}=u_{0}\in\mathcal{F}$. Suppose also the
assumptions of either Proposition \ref{prop:one=000020sided=000020chain=000020rule}
or Corollary \ref{cor:upgraded-chain-rule}. Then given both a nonnegative
caloric function $u(t)$ for $\mathcal{E}_{\eta,\pi}$ beginning at
$u_{0}$, and a curve of maximal slope $\rho_{t}$ for $\mathcal{H}(\cdot\mid\pi)$
w.r.t. $\mathcal{W}_{\eta,\pi}$ beginning at $\rho_{0}$, it holds
that $u(t)=\frac{d\rho_{t}}{d\pi}$ for all $t\in[0,T]$, and so the
two curves coincide.
\end{prop}

\begin{rem*}
Part (2) above can be thought of as being in the same spirit as the
identification of the $W_{2}$ gradient flow of the entropy, and the
$L^{2}$ gradient flow of the Cheeger energy: see for instance \cite[Theorem 8.5]{ambrosio2014calculus}
for an abstract version of this identification.
\end{rem*}
\begin{proof}
(1) From the fact that $u(t)$ is caloric, it also holds, in particular,
that for all $\psi\in C^{1}(X)\subset\mathcal{F}$, 
\[
\frac{d}{dt}\int_{X}u(t,x)\varphi(x)d\pi(x)+\frac{1}{2}\iint_{X\times X\backslash\{x=y\}}(u(t,x)-u(t,y))(\psi(x)-\psi(y))\eta(x,y)d\pi(x)d\pi(y)=0.
\]
We first observe that $\int_{X}u(t,x)d\pi(x)=1$ for all $t>0$ as
well, so that $u(t)$ is a probability density w.r.t. $\pi$. Since
$C^{1}(X)\subset\mathcal{F}$, we have in particular that the constant
function $1$ is in $\mathcal{F}$. Hence
\[
\frac{d}{dt}\int_{X}u(t,x)d\pi(x)+\frac{1}{2}\iint_{X\times X\backslash\{x=y\}}(u(t,x)-u(t,y))\underbrace{(1-1)}_{=0}\eta(x,y)d\pi(x)d\pi(y)=0
\]
and so $\frac{d}{dt}\int_{X}u(t,x)d\pi(x)=0$. Hence $u(t)$ is a
probability density w.r.t. $\pi$ for all time. Let us define the
probability measure $\rho_{t}$ by $u(t)=\frac{d\rho_{t}}{d\pi}$.

Next we argue that $\rho_{t}$ is narrowly continuous as a function
from $[0,T]$ to $\mathcal{P}(X)$. From the assumption that $u_{t}$
is caloric and $C^{1}(X)\subset\mathcal{F}$, we have that for every
$\psi\in C^{1}(X)$, the function $t\mapsto\int_{X}u(t,x)\psi(x)d\pi(x)$
is differentiable at every $t\in[0,T]$, so in particular is continuous.
Since $C^{1}(X)$ is dense in $C_{b}(X)$, by duality this means that
$t\mapsto u(t)d\pi$ is narrowly continuous.  

Lastly, we have in particular that for all $\psi\in C_{c}^{\infty}(X)$
and all $t\in[0,T]$, 
\[
\frac{d}{dt}\int_{X}\psi(x)d\rho_{t}(x)+\frac{1}{2}\iint_{X\times X\backslash\{x=y\}}\bar{\nabla}\frac{d\rho_{t}}{d\pi}(x,y)\bar{\nabla}\psi(x,y)\eta(x,y)d\pi(x)d\pi(y)=0.
\]
This implies that for all $\varphi_{t}\in C_{c}^{\infty}(X\times(0,T))$
\[
\int_{0}^{T}\int_{X}\partial_{t}\varphi_{t}(x)d\rho_{t}(x)dt-\frac{1}{2}\int_{0}^{T}\iint_{G}\bar{\nabla}\varphi_{t}(x,y)\bar{\nabla}\frac{d\rho_{t}}{d\pi}(x,y)\eta(x,y)d\pi(x)d\pi(y)dt=0
\]
by the same reasoning as in the proof of Proposition \ref{prop:compactness-nce}
(briefly one uses the fact that linear combinations of test functions
of the form $\varphi(t,x)=\tilde{\varphi}(t)\psi(x)$ are dense in
$C_{c}^{\infty}(X\times(0,T))$). Altogether this shows that $\rho_{t}$
is a weak solution to ($\dagger)$ in the sense of Definition \ref{def:nde-weak-solution}.

(2) now follows immediately by combining part (1) above with the converse
direction of Proposition \ref{prop:weak-soln-vs-gradient-flow}: specifically
we use the fact that, if there exists a curve of maximal slope $\rho_{t}$
for $\mathcal{H}(\cdot\mid\pi)$ w.r.t. $\mathcal{W}_{\eta,\pi}$
beginning at $\rho_{0}$, then $\rho_{t}$ is the unique weak solution
to ($\dagger$) (and is also an entropy weak solution).
\end{proof}
Let us make several comments regarding Proposition \ref{prop:dirichlet-caloric-equivalent}. 

Firstly, note that Proposition \ref{prop:dirichlet-caloric-equivalent}
also allows for an indirect argument that caloric functions for $\mathcal{E}_{\eta,\pi}$
which are nonnegative and have initial condition $\mu_{0}=\frac{d\rho_{0}}{d\pi}$
with $\int\rho_{0}=1$ are \emph{unique}, if we assume that $\mathcal{H}(\rho_{0}\mid\pi)<\infty$
and that $\sqrt{\mathcal{I}_{\eta}(\cdot\mid\pi)}$ is an a.e. upper
gradient for $\mathcal{H}(\cdot\mid\pi)$ w.r.t. $\mathcal{W}_{\eta,\pi}$.
This is because we can avail ourselves of the uniqueness argument
for weak solutions to $(\dagger)$ given in the converse direction
of Proposition \ref{prop:weak-soln-vs-gradient-flow}.

Secondly, it is sometimes possible to drop the assumption that $u(t)$
is nonnegative: namely, one can \emph{deduce} that $u(t)$ is nonnegative
for all $t$ provided $u(0)$ is nonnegative, if we have a \emph{parabolic
maximum principle} for caloric functions for a given Dirichlet form.
See \cite[Prop. 5.2]{grigor2009heat} for an example of this type
of result.

Lastly, one might wonder about the converse direction to part (1)
of Proposition \ref{prop:dirichlet-caloric-equivalent}: given a weak
solution $\rho_{t}$ to $(\dagger)$, is it the case that $u(t)=\frac{d\rho_{t}}{d\pi}$
is automatically caloric for $\mathcal{E}_{\eta,\pi}$? We mention
two obstacles. One is that for a function to be caloric, it is required
that $u(t)$ be weakly differentiable at all times $t$, whereas weak
solutions to ($\dagger$) are by default merely absolutely continuous
in time (but cf. discussion of this issue in \cite[Prop. 1.3]{sturm1995analysis}).
Another more serious issue is that the definition of weak solutions
used in Definition \ref{def:nde-weak-solution} has the class of test
functions as $C_{c}^{\infty}(X\times[0,T]))$, and so in particular
$C_{c}^{\infty}(X)$ is included; however we must then upgrade this
class of admissible test functions to all of $\mathcal{F}$ if $u(t)$
is to be caloric w.r.t. $(\mathcal{E}_{\eta,\pi},\mathcal{F})$. This
can be done if $C_{c}^{\infty}(X)$ happens to be dense in $\mathcal{F}$
w.r.t. the $\mathcal{E}_{\eta,\pi}^{1}$ norm. To our knowledge, the
question of the existence of an $\mathcal{E}^{1}$-closed $\mathcal{F}$
in which $C_{c}^{\infty}(X)$ is $\mathcal{E}^{1}$-dense for Dirichlet
forms of type $\mathcal{E}_{\eta,\pi}$ has not been totally resolved
in the literature; however see \cite{schilling2012structure} for
some results in this direction when $\mathcal{F}=\mathcal{D}[\mathcal{E}_{\eta,\pi}]$.

\section*{Acknowledgements}

AW gratefully acknowledges the support of Labex CARMIN while at IHÉS,
where the first half of this work was completed. At UBC, the latter
half of this work was supported by the Burroughs Wellcome Fund, and
by the Exploration Grant NFRFE-2019-00944 from the New Frontiers in
Research Fund (NFRF). Work-in-progress talks were given as part of
guest visits to the group of Jan Maas at Institute of Science and
Technology Austria, as well as the Stochastic Analysis and Application
Research Center (SAARC) at the Korea Advanced Institute of Science
and Technology (KAIST); we are glad to have enjoyed Jan Maas's and
SAARC's hospitality. AW also thanks Dejan Slep\v{c}ev for suggesting
this line of inquiry, and for countless helpful discussions well before
this project began. We also thank Young-Heon Kim for helpful advice
and Mathav Murugan for helpful discussions regarding Dirichlet forms.

\bibliographystyle{plain}
\bibliography{0_home_andrew_Documents_otrefs}

\appendix

\section{Convolution estimates\protect\label{sec:Convolution-estimates}}

We make use of a number of convolution estimates in the proof of Proposition
\ref{prop:one=000020sided=000020chain=000020rule};   this section
collects the required notions and estimates.

Given a convolution kernel $k$ and a (possibly signed, possibly vector-valued)
measure $\mu$ on $\mathbb{R}^{d}$ or $\mathbb{T}^{d}$, we denote
\[
(k*\mu)(x):=\int k(|x-y|)d\mu(y).
\]
We also abuse notation and write $k*\mu$ for the measure whose Lebesgue
density is $(k*\mu)(x)$. To be precise, $k*\mu$ should be understood
as the convolution of $\mu$ by the measure $k(x)dx$, but we choose
not to write ``$k(x)dx*\mu$'' or similar so as to reduce notational
burden. 

Separately, for measures $\mathbf{v}\in\mathcal{M}_{loc}(G)$, we
follow \cite{erbar2014gradient} define the convolution $k*\mathbf{v}\in\mathcal{M}_{loc}(G)$
of a convolution kernel $k$ (on $\mathbb{R}^{d}$ or $\mathbb{T}^{d}$)
with $\mathbf{v}$ as follows:
\[
k*\mathbf{v}=\int_{\mathbb{R}^{d}}k(|z|)d\mathbf{v}(x-z,y-z)dz
\]
in other words, for all $\varphi\in C_{c}^{\infty}(G)$, 
\[
\iint_{G}\varphi(x,y)d(k*\mathbf{v})(x,y)=\int_{\mathbb{R}^{d}}\iint_{G}k(|z|)\varphi(x+z,y+z)d\mathbf{v}(x,y)dz.
\]
Note that the convolution operation $k*\mathbf{v}$ is \emph{not}
the usual convolution on $G$, but rather is an instance of convolution
with respect to a group action on a space: in this case, $\mathbb{R}^{d}$
or $\mathbb{T}^{d}$ acts on $G$ by the translation $(x,y)\mapsto(x+z,y+z)$;
note that this is a mapping from $G$ to itself under Assumption \ref{assu:eta=000020properties}
(5). To be precise, $k*\mathbf{v}$ should be understood as the convolution
of $\mathbf{v}$ by the measure $k(x)dx$ with respect to the action
$(x,y)\mapsto(x+z,y+z)$; we warn the reader that $k*\mathbf{v}$
does \emph{not }automatically have density (with respect to $\text{Leb}^{d}\otimes\text{Leb}^{d}\llcorner G$)
since the notion of convolution employed is not the typical one.

The choice of specific notions of convolution in our setting is made
to ensure compatibility with the notion of nonlocal transport we study,
as the following result indicates.
\begin{prop}[Stability of action and metric under convolution]
\label{prop:convolution=000020action=000020stability} Let $k$ be
any convolution kernel. Let $(\mu_{t},\mathbf{v}_{t})_{t\in[0,1]}\in\mathcal{CE}$.
Then, $(k*\mu_{t},k*\mathbf{v}_{t})_{t\in[0,1]}\in\mathcal{CE}$ also.
\end{prop}

\begin{proof}
This is established in \cite[proof of Proposition 4.8]{erbar2014gradient}.
\end{proof}
Next we verify that the action $\mathcal{A}$ is stable with respect
to convolution. This is a refinement of existing results (in particular
in \cite{erbar2014gradient}) because our kernel $\eta$ is not necessarily
translation-invariant.
\begin{lem}[{Space regularization of mass-flux pairs; compare \cite[Prop. 2.8]{erbar2014gradient}}]
\label{lem:action-spatial-convolution} Suppose $\pi\ll\text{Leb}^{d}$,
and $\pi$ and $\eta$ satisfy Assumption \ref{assu:eta=000020properties}.
Then, for any convolution kernel $k$ supported on $B(0,1)$, for
all $\delta<\delta_{0}$, and all $(\mu,\mathbf{v})\in\mathcal{P}(X)\times\mathcal{M}_{loc}(G)$
with $\mu\ll\text{Leb}^{d}$ and $\mathbf{v}\ll\text{Leb}^{d}\otimes\text{Leb}^{d}$,
it holds  that
\[
\mathcal{A}(k_{\delta}*\mu,k_{\delta}*\mathbf{v};k_{\delta}*\pi)\leq S\mathcal{A}(\mu,\mathbf{v};\pi).
\]
\end{lem}

\begin{proof}
We assume without loss of generality that $\mathcal{A}(\mu,\mathbf{v};\pi)<\infty$.
Let $z\in X$. Write $\tau_{z}$ to denote the shift map $x\mapsto x+z$.
Now compute that 
\begin{align*}
\mathcal{A}(\mu\circ\tau_{z}^{-1},\mathbf{v}\circ(\tau_{z}\times\tau_{z})^{-1};\pi\circ\tau_{z}^{-1}) & =\iint\frac{\left(\frac{d\mathbf{v}}{dx\otimes dy}(x,y)\right)^{2}}{2\theta\left(\frac{d\mu}{dx}(x)\frac{d\pi}{dy}(y)\eta(x-z,y-z),\frac{d\pi}{dx}(x)\frac{d\mu}{dy}(y)\eta(y-z,x-z)\right)}dxdy.
\end{align*}
Suppose that $|z|<\delta_{0}$. Hence, $\frac{\eta(x,y)}{\eta(x-z,y-z)}\leq S$
under Assumption \ref{assu:eta=000020properties}.

Now, compute, from the monotonicity of $\theta$ in each variable,
then 1-homogeneity of $\theta$, that
\begin{multline*}
\theta\left(\frac{d\mu}{dx}(x)\frac{d\pi}{dy}(y)\eta(x-z,y-z),\frac{d\pi}{dx}(x)\frac{d\mu}{dy}(y)\eta(y-z,x-z)\right)\\
\begin{aligned} & =\theta\left(\frac{d\mu}{dx}(x)\frac{d\pi}{dy}(y)S^{-1}\eta(x,y),\frac{d\pi}{dx}(x)S^{-1}\frac{d\mu}{dy}(y)\eta(y,x)\right)\\
 & =S^{-1}\theta\left(\frac{d\mu}{dx}(x)\frac{d\pi}{dy}(y)\eta(x,y),\frac{d\pi}{dx}(x)\frac{d\mu}{dy}(y)\eta(y,x)\right).
\end{aligned}
\end{multline*}
Consequently, when $|z|<\delta_{0}$, we have that pointwise in $(x,y)$,
\[
\frac{\left(\frac{d\mathbf{v}}{dx\otimes dy}(x,y)\right)^{2}}{2\theta\left(\frac{d(\mu\otimes\pi)}{dx\otimes dy}(x,y)\eta(x-z,y-z),\frac{d(\pi\otimes\mu)}{dx\otimes dy}(x,y)\eta(y-z,x-z)\right)}\leq S\frac{\left(\frac{d\mathbf{v}}{dx\otimes dy}(x,y)\right)^{2}}{2\theta\left(\frac{d(\mu\otimes\pi)}{dx\otimes dy}(x,y)\eta(x,y),\frac{d(\pi\otimes\mu)}{dx\otimes dy}(x,y)\eta(y,x)\right)}
\]
and thus
\begin{multline*}
\int_{|z|<\delta}\frac{\left(\frac{d\mathbf{v}}{dx\otimes dy}(x,y)\right)^{2}}{2\theta\left(\frac{d(\mu\otimes\pi)}{dx\otimes dy}(x,y)\eta(x-z,y-z),\frac{d(\pi\otimes\mu)}{dx\otimes dy}(x,y)\eta(y-z,x-z)\right)}k_{\delta}(z)dz\\
\leq S\frac{\left(\frac{d\mathbf{v}}{dx\otimes dy}(x,y)\right)^{2}}{2\theta\left(\frac{d(\mu\otimes\pi)}{dx\otimes dy}(x,y)\eta(x,y),\frac{d(\pi\otimes\mu)}{dx\otimes dy}(x,y)\eta(y,x)\right)}.
\end{multline*}
Therefore, 
\[
\mathcal{A}(\mu\circ\tau_{z}^{-1},\mathbf{v}\circ(\tau_{z}\times\tau_{z})^{-1};\pi\circ\tau_{z}^{-1})\leq S\mathcal{A}(\mu,\mathbf{v};\pi).
\]

Moreover, since $\mathcal{A}$ is jointly convex, if $\{z_{i},\ldots,z_{N}\}$
has the property that $|z_{i}|<\delta_{0}$ for all $i=1,\ldots,N$,
then
\begin{align*}
\mathcal{A}\left(\frac{1}{N}\sum_{i=1}^{N}\mu\circ\tau_{z_{i}}^{-1},\frac{1}{N}\sum_{i=1}^{N}\mathbf{v}\circ(\tau_{z_{i}}\times\tau_{z_{i}})^{-1};\frac{1}{N}\sum_{i=1}^{N}\pi\circ\tau_{z_{i}}^{-1}\right) & \leq\frac{1}{N}\sum_{i=1}^{N}\mathcal{A}\left(\mu\circ\tau_{z_{i}}^{-1},\mathbf{v}\circ(\tau_{z_{i}}\times\tau_{z_{i}})^{-1};\pi\circ\tau_{z_{i}}^{-1}\right)\\
 & \leq S\mathcal{A}(\mu,\mathbf{v};\pi).
\end{align*}
Now, since $\mathcal{A}$ is jointly sequentially l.s.c. with respect
to narrow convergence in $\mathcal{P}(X)$ for $\mu$, local weak{*}
convergence in $\mathcal{M}_{loc}^{+}(X)$ for $\pi$, and local weak{*}
convergence in $\mathcal{M}_{loc}(G)$, it suffices to show that when
$Z_{i},\ldots,Z_{N}$ are i.i.d. random variables with distribution
$k_{\delta}(z)dz$, then with probability 1, 
\[
\frac{1}{N}\sum_{i=1}^{N}\mu\circ\tau_{Z_{i}}^{-1}\stackrel[N\rightarrow\infty]{*}{\rightharpoonup}k_{\delta}*\mu,\quad\frac{1}{N}\sum_{i=1}^{N}\pi\circ\tau_{Z_{i}}^{-1}\stackrel[N\rightarrow\infty]{*}{\rightharpoonup}k_{\delta}*\pi,\text{ and }\frac{1}{N}\sum_{i=1}^{N}\mathbf{v}\circ(\tau_{Z_{i}}\times\tau_{Z_{i}})^{-1}\stackrel[N\rightarrow\infty]{*}{\rightharpoonup}k_{\delta}*\mathbf{v},
\]
as in this case, with probability 1, 
\[
\mathcal{A}(k_{\delta}*\mu,k_{\delta}*\mathbf{v};k_{\delta}*\pi)\leq\liminf_{n\rightarrow\infty}\mathcal{A}\left(\frac{1}{N}\sum_{i=1}^{N}\mu\circ\tau_{Z_{i}}^{-1},\frac{1}{N}\sum_{i=1}^{N}\mathbf{v}\circ(\tau_{Z_{i}}\times\tau_{Z_{i}})^{-1}\right)\leq S\mathcal{A}(\mu,\mathbf{v};\pi).
\]
We now reason as in the proof of \cite[proof of Lemma A.3]{slepcev2023nonlocal}.\footnote{Here we are basically establishing Jensen's inequality via the strong
law of large numbers for random measures. The usual proofs of Jensen's
inequality do not apply here, since even though $\mathcal{A}$ is
jointly convex and l.s.c., its support has empty interior inside the
space of signed measures, hence versions of Jensen's inequality for
functionals on infinite-dimensional spaces (such as the one given
in \cite{perlman1974jensen}) cannot be directly employed.} Indeed, the fact that $\frac{1}{N}\sum_{i=1}^{N}\mu\circ\tau_{Z_{i}}^{-1}\rightharpoonup^{*}k_{\delta}*\mu$
(and likewise $\frac{1}{N}\sum_{i=1}^{N}\pi\circ\tau_{Z_{i}}^{-1}\rightharpoonup^{*}k_{\delta}*\pi$)
is already established verbatim in \cite[proof of Lemma A.3]{slepcev2023nonlocal}
under more general assumptions, so we omit the details; in any case,
the proof is very similar to (indeed, simpler than) the one we provide
now which shows that $\frac{1}{N}\sum_{i=1}^{N}\mathbf{v}\circ(\tau_{Z_{i}}\times\tau_{Z_{i}})^{-1}\rightharpoonup^{*}k_{\delta}*\mathbf{v}$. 

Let $\mathbf{v}=\mathbf{v}^{+}-\mathbf{v}^{-}$ denote the Jordan
decomposition of $\mathbf{v}$. Note that $\left(\mathbf{v}\circ(\tau_{Z_{i}}\times\tau_{Z_{i}})^{-1}\right)^{+}=\mathbf{v}^{+}\circ(\tau_{Z_{i}}\times\tau_{Z_{i}})^{-1}$
and similarly for $\mathbf{v}^{-}$.

It suffices to show that: for every compact $K\subset G$, and every
$\varphi\in C_{C}(K)$, with probability 1, 
\[
\frac{1}{N}\sum_{i=1}^{N}\iint_{K}\varphi d\left(\mathbf{v}^{+}\circ(\tau_{Z_{i}}\times\tau_{Z_{i}})^{-1}\right)\rightarrow\iint_{K}\varphi d\left(k_{\delta}*\mathbf{v}^{+}\right)
\]
and similarly for $\mathbf{v}^{-}$. More precisely, if this convergence
holds on arbitrary $K$, with arbitrary test function $\varphi\in C_{C}(K)$,
with probability 1, then \cite[Lemma 4.8 (i)]{kallenberg2017random}
indicates that for each $K$, with probability 1, 
\[
\frac{1}{N}\sum_{i=1}^{N}\mathbf{v}^{+}\circ(\tau_{Z_{i}}\times\tau_{Z_{i}})^{-1}\llcorner K\stackrel[N\rightarrow\infty]{*}{\rightharpoonup}k_{\delta}*\mathbf{v}^{+}\llcorner K
\]
and similarly for $\mathbf{v}^{-}$, so by definition of convergence
in the local weak{*} topology, this shows that $\frac{1}{N}\sum_{i=1}^{N}\mathbf{v}^{+}\circ(\tau_{Z_{i}}\times\tau_{Z_{i}})^{-1}\stackrel[N\rightarrow\infty]{*}{\rightharpoonup}k_{\delta}*\mathbf{v}^{+}$
and $\frac{1}{N}\sum_{i=1}^{N}\mathbf{v}^{-}\circ(\tau_{Z_{i}}\times\tau_{Z_{i}})^{-1}\stackrel[N\rightarrow\infty]{*}{\rightharpoonup}k_{\delta}*\mathbf{v}^{-}$.

We claim that for each $\varphi\in C_{C}(K)$, it holds that $\iint_{K}\varphi d\left(\mathbf{v}^{+}\circ(\tau_{Z}\times\tau_{Z})^{-1}\right)$
is an $L^{\infty}$ random variable when $Z$ is distributed according
to $k_{\delta}(z)dz$. (The same argument will work for $\mathbf{v}^{-}$.)
It suffices to show that $\left\Vert \mathbf{v}\circ(\tau_{z}\times\tau_{z})^{-1}\llcorner K\right\Vert _{TV}$
is uniformly bounded in $z$, since 
\begin{align*}
\left|\iint_{K}\varphi d\left(\mathbf{v}^{+}\circ(\tau_{v}\times\tau_{z})^{-1}\right)\right| & \leq\Vert\varphi\Vert_{L^{\infty}(K)}\left\Vert \mathbf{v}^{+}\circ(\tau_{z}\times\tau_{z})^{-1}\llcorner K\right\Vert _{TV}\\
 & \leq\Vert\varphi\Vert_{L^{\infty}(K)}\left\Vert \mathbf{v}\circ(\tau_{z}\times\tau_{z})^{-1}\llcorner K\right\Vert _{TV}.
\end{align*}
To this end, we know from Lemma \ref{lem:C-eta=000020upper=000020bound=000020of=000020flux}
that there is some uniform constant $C$ such that 
\[
|\mathbf{v}(K)|\leq\frac{1}{a_{K}}C\mathcal{A}(\mu,\mathbf{v});\qquad a_{K}:=\min\{|x-y|:(x,y)\in K\}.
\]
Now, note that $\mathbf{v}\circ(\mathcal{\tau}_{z}\times\tau_{z})^{-1}(K)=\mathbf{v}\left(K-\binom{z}{z}\right)$;
at the same time, $a_{K}=a_{K-\binom{z}{z}}$. Thus, uniformly in
$z$,
\[
|\mathbf{v}\circ(\mathcal{\tau}_{z}\times\tau_{z})^{-1}(K)|\leq\frac{1}{a_{K}}C\mathcal{A}(\mu,\mathbf{v}).
\]

Thus, since $\iint_{K}\varphi d\left(\mathbf{v}^{+}\circ(\tau_{Z}\times\tau_{Z})^{-1}\right)$
is an $L^{\infty}$ random variable, we can apply the strong law of
large numbers to deduce that 
\begin{align*}
\iint_{K}\varphi d\left(\frac{1}{N}\sum_{i=1}^{N}\mathbf{v}^{+}\circ(\tau_{Z_{i}}\times\tau_{Z_{i}})^{-1}\right) & =\frac{1}{N}\sum_{i=1}^{N}\iint_{K}\varphi d\left(\mathbf{v}^{+}\circ(\tau_{Z_{i}}\times\tau_{Z_{i}})^{-1}\right)\\
\text{(a.s.)} & \rightarrow\int_{\mathbb{R}^{d}}\iint_{K}\varphi d\left(\mathbf{v}^{+}\circ(\tau_{v}\times\tau_{z})^{-1}\right)k_{\delta}(z)dz\\
 & =\iint_{K}\varphi d\left(\int_{\mathbb{R}^{d}}\mathbf{v}^{+}\circ(\tau_{v}\times\tau_{z})^{-1}k_{\delta}(z)dz\right)\\
 & =\iint_{K}\varphi d\left(k_{\delta}*\mathbf{v}^{+}\right).
\end{align*}
Applying the same reasoning to $\mathbf{v}^{-}$, we deduce from \cite[Lemma 4.8 (i)]{kallenberg2017random}
that with probability 1, both 
\[
\frac{1}{N}\sum_{i=1}^{N}\mathbf{v}^{+}\circ(\tau_{Z_{i}}\times\tau_{Z_{i}})^{-1}\llcorner K\stackrel[N\rightarrow\infty]{*}{\rightharpoonup}k_{\delta}*\mathbf{v}^{+}\llcorner K\text{ and }\frac{1}{N}\sum_{i=1}^{N}\mathbf{v}^{-}\circ(\tau_{Z_{i}}\times\tau_{Z_{i}})^{-1}\llcorner K\stackrel[N\rightarrow\infty]{*}{\rightharpoonup}k_{\delta}*\mathbf{v}^{-}\llcorner K.
\]
Quantifying over a countable compact exhaustion of $G$, this means
that (in the local weak{*} topology) with probability 1,
\[
\frac{1}{N}\sum_{i=1}^{N}\mathbf{v}^{+}\circ(\tau_{Z_{i}}\times\tau_{Z_{i}})^{-1}\stackrel[N\rightarrow\infty]{*}{\rightharpoonup}k_{\delta}*\mathbf{v}^{+}\text{ and }\frac{1}{N}\sum_{i=1}^{N}\mathbf{v}^{-}\circ(\tau_{Z_{i}}\times\tau_{Z_{i}})^{-1}\stackrel[N\rightarrow\infty]{*}{\rightharpoonup}k_{\delta}*\mathbf{v}^{-}.
\]
 Since $k_{\delta}*(\mathbf{v}^{+}-\mathbf{v}^{-})=k_{\delta}*\mathbf{v}^{+}-k_{\delta}*\mathbf{v}^{-}$,
this establishes that with probability 1,
\[
\frac{1}{N}\sum_{i=1}^{N}\mathbf{v}^{+}\circ(\tau_{Z_{i}}\times\tau_{Z_{i}})^{-1}-\frac{1}{N}\sum_{i=1}^{N}\mathbf{v}^{-}\circ(\tau_{Z_{i}}\times\tau_{Z_{i}})^{-1}\stackrel[N\rightarrow\infty]{*}{\rightharpoonup}k_{\delta}*\mathbf{v}
\]
and hence with probability 1, $\frac{1}{N}\sum_{i=1}^{N}\mathbf{v}\circ(\tau_{Z_{i}}\times\tau_{Z_{i}})^{-1}\stackrel[N\rightarrow\infty]{*}{\rightharpoonup}k_{\delta}*\mathbf{v}$,
as desired. 
\end{proof}
\begin{cor}
\label{cor:action-space-convolution-convergence}Suppose that $(\mu_{t},\mathbf{v}_{t})\in\mathcal{CE}$
and $\int_{0}^{1}\mathcal{A}(\mu_{t},\mathbf{v}_{t};\pi)<\infty$.
Assume also that $\mathbf{v}_{t}$ is a tangent flux for $\mu_{t}$
in the sense of Fact \ref{fact:tangent-flux}. Under the same assumptions
as Lemma \ref{lem:action-spatial-convolution}, we have that as $\delta\rightarrow0$,
$\mathcal{A}(k_{\delta}*\mu_{t},k_{\delta}*\mathbf{v}_{t};k_{\delta}*\pi)\rightarrow\mathcal{A}(\mu_{t},\mathbf{v}_{t};\pi)$
for almost all $t$, $d(k_{\delta}*\mathbf{v}_{t})dt\rightharpoonup^{*}d\mathbf{v}_{t}dt$
in $\mathcal{M}_{loc}([0,1]\times G)$, and 
\[
\int_{0}^{1}\mathcal{A}(k_{\delta}*\mu_{t},k_{\delta}*\mathbf{v}_{t};k_{\delta}*\pi)dt\rightarrow\int_{0}^{1}\mathcal{A}(\mu_{t},\mathbf{v}_{t};\pi)dt.
\]
\end{cor}

\begin{proof}
Assume that $\delta<\delta_{0}$. On the one hand, we have that $\mathcal{A}(k_{\delta}*\mu_{t},k_{\delta}*\mathbf{v}_{t};k_{\delta}*\pi)\leq S\mathcal{A}(\mu_{t},\mathbf{v}_{t};\pi)$,
where $S\rightarrow1$ as $\delta_{0}\rightarrow0$. Hence 
\[
\limsup_{\delta\rightarrow0}\mathcal{A}(k_{\delta}*\mu_{t},k_{\delta}*\mathbf{v}_{t};k_{\delta}*\pi)\leq\limsup_{\delta_{0}\rightarrow0}S\mathcal{A}(\mu_{t},\mathbf{v}_{t};\pi)=\mathcal{A}(\mu_{t},\mathbf{v}_{t};\pi).
\]
At the same time, we have uniformly in $\delta<\delta_{0}$ that 
\[
\int_{0}^{1}\mathcal{A}(k_{\delta}*\mu_{t},k_{\delta}*\mathbf{v}_{t};k_{\delta}*\pi)dt\leq S\int_{0}^{1}\mathcal{A}(\mu_{t},\mathbf{v}_{t};\pi)dt<\infty
\]
We know that $k_{\delta}*\mu$ and $k_{\delta}*\pi$ converge narrowly
as $\delta\rightarrow0$ to $\mu$ and $\pi$. Hence, applying the
compactness result from Proposition \ref{prop:compactness-nce}, along
any $\delta_{n}\rightarrow0$ we can pass to some subsequence (which
we do not relabel) along which we can extract a limiting $(\mu_{t},\tilde{\mathbf{v}}_{t})_{t\in[0,1]}$
such that $k_{\delta}*\mathbf{v}_{t}$ converges to $\tilde{\mathbf{v}}_{t}$
in the sense that $d(k_{\delta}*\mathbf{v}_{t})dt\rightharpoonup^{*}d\tilde{\mathbf{v}}_{t}dt$
in $\mathcal{M}_{loc}([0,1]\times G)$, and by Fatou's lemma
\begin{align*}
\int_{0}^{1}\mathcal{A}(\mu_{t},\tilde{\mathbf{v}}_{t};\pi)dt & \leq\liminf_{\delta_{n}\rightarrow0}\int_{0}^{1}\mathcal{A}(k_{\delta_{n}}*\mu_{t},k_{\delta_{n}}*\mathbf{v}_{t};k_{\delta_{n}}*\pi)dt.\\
 & \leq\int_{0}^{1}\limsup_{\delta_{n}\rightarrow0}\mathcal{A}(k_{\delta_{n}}*\mu_{t},k_{\delta_{n}}*\mathbf{v}_{t};k_{\delta_{n}}*\pi)dt\\
 & =\int_{0}^{1}\mathcal{A}(\mu_{t},\mathbf{v}_{t};\pi)dt.
\end{align*}
By the assumption that $\mathbf{v}_{t}$ is tangent, we have that
for almost all $t\in[0,1]$, $\mathcal{A}(\mu_{t},\mathbf{v}_{t};\pi)\leq\mathcal{A}(\mu_{t},\tilde{\mathbf{v}}_{t};\pi)$.
Hence in fact, $\mathcal{A}(\mu_{t},\mathbf{v}_{t};\pi)=\mathcal{A}(\mu_{t},\tilde{\mathbf{v}}_{t};\pi)$
for almost all $t\in[0,1]$, and so by the a.e. uniqueness of tangent
fluxes we actually have $\mathbf{v}_{t}=\tilde{\mathbf{v}}_{t}$ for
almost all $t\in[0,1]$, so $d\tilde{\mathbf{v}}_{t}dt=d\mathbf{v}_{t}dt$.
Moreover, the subsequence of $\delta_{n}$, and $\delta_{n}$ itself
were arbitrary, hence 
\[
\lim_{\delta\rightarrow0}\int_{0}^{1}\mathcal{A}(k_{\delta}*\mu_{t},k_{\delta}*\mathbf{v}_{t};k_{\delta}*\pi)dt=\int_{0}^{1}\mathcal{A}(\mu_{t},\mathbf{v}_{t};\pi)dt.
\]

\end{proof}
We now turn to convolution estimates for the relative entropy and
nonlocal relative Fisher information.
\begin{lem}
\label{lem:convolution-regularity-entropy}Let $k$ be a standard
convolution kernel supported on $B(0,1)$, let $\rho,\pi\in\mathcal{P}(X)$.
Then $\mathcal{H}(k_{\delta}*\rho\mid k_{\delta}*\pi)\leq\mathcal{H}(\rho\mid\pi)$,
and $\lim_{\delta\rightarrow0}\mathcal{H}(k_{\delta}*\rho\mid k_{\delta}*\pi)=\mathcal{H}(\rho\mid\pi)$. 
\end{lem}

\begin{proof}
On the one hand, since $\mathcal{H}(\cdot\mid\cdot)$ is jointly weak{*}
l.s.c. we have 
\[
\mathcal{H}(\rho\mid\pi)\leq\liminf_{\delta\rightarrow0}\mathcal{H}(k_{\delta}*\rho\mid k_{\delta}*\pi).
\]
On the other hand, we can argue similarly to Lemma \ref{lem:action-spatial-convolution}
to get a corresponding upper bound. Indeed, when $Z_{i},\ldots,Z_{N}$
are i.i.d. random variables with distribution $k_{\delta}(z)dz$,
then with probability 1, 
\[
\frac{1}{N}\sum_{i=1}^{N}\rho\circ\tau_{Z_{i}}^{-1}\stackrel[N\rightarrow\infty]{*}{\rightharpoonup}k_{\delta}*\rho;\quad\frac{1}{N}\sum_{i=1}^{N}\pi\circ\tau_{Z_{i}}^{-1}\stackrel[N\rightarrow\infty]{*}{\rightharpoonup}k_{\delta}*\pi.
\]
Thus, from the joint convexity of $\mathcal{H}$ we have almost surely
that
\[
\mathcal{H}(k_{\delta}*\rho\mid k_{\delta}*\pi)\leq\liminf_{N\rightarrow\infty}\frac{1}{N}\sum_{i=1}^{N}\mathcal{H}(\rho\circ\tau_{Z_{i}}^{-1}\mid\pi\circ\tau_{Z_{i}}^{-1}).
\]
But notice that $\frac{d(\rho\circ\tau_{Z_{i}}^{-1})}{d(\pi\circ\tau_{Z_{i}}^{-1})}=\frac{d\rho}{d\pi}$,
and so $\mathcal{H}(\rho\circ\tau_{Z_{i}}^{-1}\mid\pi\circ\tau_{Z_{i}}^{-1})=\mathcal{H}(\rho\mid\pi)$.
It follows that 
\[
\mathcal{H}(k_{\delta}*\rho\mid k_{\delta}*\pi)\leq\mathcal{H}(\rho\mid\pi)
\]
so, as $\delta\rightarrow0$, we have$\mathcal{H}(k_{\delta}*\rho\mid k_{\delta}*\pi)\rightarrow\mathcal{H}(\rho\mid\pi).$
\end{proof}
\begin{lem}[regularity of Fisher information under convolution]
\label{lem:convolution-regularity-of-Fisher} Let $k$ be a standard
convolution kernel supported on $B(0,1)$, and let $\rho,\pi\in\mathcal{P}(X)$.
Assume that $(\eta,\pi)$ satisfy Assumption \ref{assu:eta=000020properties}.
Then, for $\delta<\delta_{0}$, it holds that $\mathcal{I}_{\eta}(k_{\delta}*\rho\mid k_{\delta}*\pi)\leq S\mathcal{I}_{\eta}(\rho\mid\pi)$,
and in particular $\lim_{\delta\rightarrow0}\mathcal{I}_{\eta}(k_{\delta}*\rho\mid k_{\delta}*\pi)=\mathcal{I}_{\eta}(\rho\mid\pi)$. 
\end{lem}

\begin{proof}
Compute that 
\begin{align*}
\mathcal{I}_{\eta}(\rho\circ\tau_{z}\mid\pi\circ\tau_{z}) & =\frac{1}{2}\iint_{G}\bar{\nabla}\log\frac{d\rho}{d\pi}(x,y)\bar{\nabla}\frac{d\rho}{d\pi}(x,y)\eta(x-z,y-z)d\pi(x)d\pi(y).
\end{align*}
\begin{comment}
Under the assumption that $\pi\ll Leb$
\end{comment}
For $|z|\leq\delta_{0}$,
\[
\frac{\eta(x,y)}{\eta(x-z,y-z)}\leq S
\]
under Assumption \ref{assu:eta=000020properties}, and so we find
that pointwise in $(x,y)$, 
\begin{align*}
\bar{\nabla}\log\frac{d\rho}{d\pi}(x,y)\bar{\nabla}\rho(x,y)\eta(x-z,y-z) & \leq S\cdot\bar{\nabla}\log\rho(x,y)\bar{\nabla}\rho(x,y)\eta(x,y)
\end{align*}
which shows that $\mathcal{I}_{\eta}(\rho\circ\tau_{z}\mid\pi\circ\tau_{z})\leq S\mathcal{I}_{\eta}(\rho\mid\pi)$.
Then, by Proposition \ref{prop:fisher-convex-lsc}, $\mathcal{I}$
is jointly convex in both its arguments, so we argue as in the proof
of Lemma \ref{lem:action-spatial-convolution} to see that 
\[
\mathcal{I}_{\eta}(k_{\delta}*\rho\mid k_{\delta}*\pi)\leq\int_{X}\mathcal{I}_{\eta}(\rho\circ\tau_{z}\mid\pi\circ\tau_{z})k_{\delta}(z)dz\leq S\mathcal{I}_{\eta}(\rho\mid\pi).
\]

Sending $\delta_{0}\rightarrow0$ this establishes that $\limsup_{\delta\rightarrow0}\mathcal{I}_{\eta}(k_{\delta}*\rho\mid k_{\delta}*\pi)=\mathcal{I}_{\eta}(\rho\mid\pi)$
since $S\rightarrow1$. At the same time, $\mathcal{I}_{\eta}$ is
jointly l.s.c. with respect to narrow convergence thanks to Proposition
\ref{prop:fisher-convex-lsc}, so $\liminf_{\delta\rightarrow0}\mathcal{I}_{\eta}(k_{\delta}*\rho\mid k_{\delta}*\pi)=\mathcal{I}_{\eta}(\rho\mid\pi)$
as well. Hence altogether $\lim_{\delta\rightarrow0}\mathcal{I}_{\eta}(k_{\delta}*\rho\mid k_{\delta}*\pi)=\mathcal{I}_{\eta}(\rho\mid\pi)$
as desired.
\end{proof}

\begin{lem}
\label{lem:flux-density-pointwise-convergence}Let $X$ be either
$\mathbb{T}^{d}$ or $\mathbb{R}^{d}$. Let $\mathbf{v}\in\mathcal{M}_{loc}(G)$.
Suppose that $\mathbf{v}$ has density $w(x,y)$ with respect to $\text{Leb}^{d}\otimes\text{Leb}^{d}\llcorner G$.
Then $k_{\delta}*\mathbf{v}$ has density $k_{\delta}*w$ with respect
to $\text{Leb}^{d}\otimes\text{Leb}^{d}\llcorner G$, and for almost
all $(x,y)$, 
\[
(k_{\delta}*w)(x,y)\rightarrow w(x,y)\text{ as }\delta\rightarrow0.
\]
\end{lem}

We remind the reader that here, we are not considering the ``usual''
convolution, but rather convolution with respect to the action of
$X$ on $G$ given by $(x,y)\mapsto(x+z,y+z)$. Accordingly, the pointwise
a.e. convergence of the mollified density $k_{\delta}*w$ is not directly
given by standard convergence theorems for mollifiers, but rather
must be established carefully by hand. 
\begin{proof}
First we recall that $k_{\delta}*\mathbf{v}$ is defined as follows:
for any $\varphi\in C_{C}(G)$ we have 
\[
\iint_{G}\varphi(x,y)k_{\delta}*\mathbf{v}(x,y)=\int_{|z|<\delta}\iint_{G}\varphi(x+z,y+z)d\mathbf{v}(x,y)k_{\delta}(|z|)dz.
\]
Since $\mathbf{v}(x,y)$ has Lebesgue density $w(x,y)$ we have 
\begin{align*}
\iint_{G}\varphi(x,y)k_{\delta}*\mathbf{v}(x,y) & =\int_{|z|<\delta}\iint_{G}\varphi(x+z,y+z)w(x,y)dxdyk_{\delta}(|z|)dz\\
 & =\int_{|z|<\delta}\iint_{G}\varphi(x,y)w(x-z,y-z)dxdyk_{\delta}(|z|)dz\\
 & =\iint_{G}\varphi(x,y)\left(\int_{|z|<\delta}w(x-z,y-z)k_{\delta}(|z|)\right)dxdy\\
 & =\iint_{G}\varphi(x,y)k_{\delta}*w(x,y)dxdy.
\end{align*}
Since the test function $\varphi$ is arbitrary this shows that $k_{\delta}*\mathbf{v}=k_{\delta}*w(x,y)dxdy$,
which establishes the first claim.

Next, since $\mathbf{v}\in\mathcal{M}_{loc}$ we have that $w\in L_{loc}^{1}(G,\text{Leb}^{d}\otimes\text{Leb}^{d}\llcorner G)$.
In what follows, we use a more convenient coordinate system on $G$
which corresponds to the orbit decomposition of $G$ with respect
to the action $X\curvearrowright G$. Namely we work in the coordinate
system $(x^{\prime},z^{\prime})$ where 
\[
x^{\prime}=\frac{y-x}{\sqrt{2}};\qquad z^{\prime}=\frac{x+y}{\sqrt{2}}.
\]
This is an orthogonal coordinate system on $X\times X$, and in these
coordinates we can write $G=\{X\times X:x^{\prime}\neq0\}$. Moreover
the action of $X$ on $G$ given by $(x,y)\mapsto(x+z,y+z)$ leaves
$x^{\prime}$ fixed, and acts on the $z^{\prime}$ coordinate by $z^{\prime}\mapsto z^{\prime}+\sqrt{2}z$.
Furthermore this change of coordinates preserves the Lebesgue measure
$\text{Leb}^{d}\otimes\text{Leb}^{d}\llcorner G$.

We rewrite the density $w(x,y)dy$ as $w(x^{\prime},z^{\prime})dx^{\prime}dz^{\prime}$,
and note that 
\[
\int w(x-z,y-z)k_{\delta}(|z|)dz=\int w(x^{\prime},z^{\prime}-\sqrt{2}z)k_{\delta}(|z|)dz.
\]
Now, since $w\in L_{loc}^{1}(G,\text{Leb}^{d}\otimes\text{Leb}^{d}\llcorner G)$,
and our coordinate transformation is measure preserving, we have that
\[
\iint_{K}|w(x^{\prime},z^{\prime})|dx^{\prime}dz^{\prime}<\infty
\]
and so by the Fubini-Tonelli theorem we have that for almost all $x_{0}^{\prime}\in X$,
the function $z^{\prime}\mapsto w(x_{0}^{\prime},z^{\prime})$ is
in $L_{loc}^{1}(X)$. Hence by \cite[Thm. C.7(ii)]{evans2010partial}
we have that for almost all $z_{0}^{\prime}$, 
\begin{align*}
\int w(x_{0}^{\prime},z_{0}^{\prime}-\sqrt{2}z)k_{\delta}(|z|)dz & =\int w(x_{0}^{\prime},z_{0}^{\prime}-z^{\prime})k_{\delta}\left(\frac{|z^{\prime}|}{\sqrt{2}}\right)\frac{1}{\sqrt{2}^{d}}dz^{\prime}\\
 & =\int w(x_{0}^{\prime},z_{0}^{\prime}-z^{\prime})k_{\sqrt{2}\delta}\left(|z^{\prime}|\right)dz^{\prime}\\
 & \rightarrow w(x_{0}^{\prime},z_{0}^{\prime}).
\end{align*}
(Here we have performed the (pedantic) change of coordinates from
$z$ belonging to $X$ (which acts on $G$) to $z^{\prime}$ in our
new coordinate system on $G$, and in particular have used the fact
that since $z^{\prime}=\sqrt{2}z$ we have $dz^{\prime}=\sqrt{2}^{d}dz$
since the differentials encode the $d$-dimensional Lebesgue measure.
Also, recall that $k_{\delta}(|z|)=\frac{1}{\delta^{d}}k\left(\frac{|z|}{\delta}\right)$.)
Since the null sets in $(x,y)$ coordinates are the same as in the
$(x^{\prime},z^{\prime})$ coordinates, we have that $k_{\delta}*w(x,y)\rightarrow w(x,y)$
for $\text{Leb}^{d}\otimes\text{Leb}^{d}\llcorner G$-almost all $(x,y)$
as desired.
\end{proof}
Additionally, we \emph{also} consider mollifications of AC curves
$(\rho_{t},\mathbf{v}_{t})_{t\in[0,T]}$ in \emph{time}, along the
lines of the proof of \cite[Prop. 3.10]{esposito2019nonlocal}. Namely,
let us extend $(\rho_{t},\mathbf{v}_{t})$ by reflection to be defined
on the time interval $[-T,2T]$, and take $h(t)$ to be a standard
mollifier supported on $[-1,1]$. For $\sigma\in(0,T]$, we then define
the convolution of $(k_{\delta}*\rho_{t},k_{\delta}*\mathbf{v}_{t})_{t\in[0,1]}$
as follows:
\[
\forall A\in\mathcal{B}(X)\qquad h_{\sigma}*\rho_{t}(A):=\int_{-\sigma}^{\sigma}h_{\sigma}(t-s)\rho_{s}(A)ds;
\]
\[
\forall U\in\mathcal{B}(G)\qquad h_{\sigma}*\mathbf{v}_{t}(U):=\int_{-\sigma}^{\sigma}h_{\sigma}(t-s)\mathbf{v}_{s}(U)ds.
\]
(Here $h_{\sigma}(t):=\frac{1}{\sigma}h(t/\sigma)$.) Let us verify
that as $\sigma\rightarrow0$, we have the narrow (resp. local weak{*})
convergence of $h_{\sigma}*\rho_{t}$ to $\rho_{t}$ and $h_{\sigma}*\mathbf{v}_{t}$
to $\mathbf{v}_{t}$.
\begin{lem}
\label{lem:time-mollified-measures-convergence}Let $(\rho_{t},\mathbf{v}_{t})_{t\in[0,1]}$
solve the nonlocal continuity equation, and assume $\int_{0}^{1}\sqrt{\mathcal{A}_{\eta}(\rho_{t},\mathbf{v}_{t};\pi)}dt<\infty$.
Then, for all $t\in[0,1]$, $h_{\sigma}*\rho_{t}\rightharpoonup\rho_{t}$;
and, for almost all $t\in[0,1]$, $h_{\sigma}*\mathbf{v}_{t}\rightharpoonup^{*}\mathbf{v}_{t}$. 
\end{lem}

\begin{proof}
\begin{comment}
{[}todo but same as in EPSS --- except that $k_{\delta}*\mathbf{v}_{t}$
is not a tangent flux! Maybe need to interchange order of convolutions.
Hm{]}
\end{comment}
{} %
{} For $\rho_{t}$, we use the fact that $t\mapsto\rho_{t}$ is narrowly
continuous. Let $\varphi\in C_{b}(X)$ be arbitrary. Approximating
$\varphi$ with simple functions, and then appealing to the monotone
convergence theorem, we see that 
\[
\int_{X}\varphi(x)d\left(h_{\sigma}*\rho_{t}\right)(x)=\int_{-\sigma}^{\sigma}h_{\sigma}(t-s)\left(\int_{X}\varphi(x)d\rho_{s}(x)\right)ds.
\]
Now, $s\mapsto\int_{X}\varphi(x)d\rho_{s}(x)$ is itself bounded and
continuous, so by \cite[Thm. C.7(iii)]{evans2010partial} it holds
for all $t\in[0,1]$ that as $\sigma\rightarrow0$, 
\[
\int_{-\sigma}^{\sigma}h_{\sigma}(t-s)\left(\int_{X}\varphi(x)d\rho_{s}(x)\right)ds\rightarrow\left(\int_{X}\varphi(x)d\rho_{t}(x)\right).
\]
Since $\varphi$ was arbitrary, we conclude that for all $t\in[0,1]$
and $\varphi\in C_{b}(X)$, $\int_{X}\varphi(x)d\left(h_{\sigma}*\rho_{t}\right)(x)\rightarrow\left(\int_{X}\varphi(x)d\rho_{t}(x)\right)$
as $\sigma\rightarrow0$. In other words, $h_{\sigma}*\rho_{t}\rightharpoonup^{*}\rho_{t}$.

For $\mathbf{v}_{t}$ the situation is different because $t\mapsto\mathbf{v}_{t}$
is merely Borel. Nonetheless, from Lemma \ref{lem:C-eta=000020upper=000020bound=000020of=000020flux}
we have that for every compact $K$, 
\[
|\mathbf{v}_{t}|(K)\leq C_{K,\eta}\sqrt{\mathcal{A}(\rho_{t},\mathbf{v}_{t};\pi)}.
\]
In particular, $t\mapsto\mathbf{v}_{t}(K)$ is $L^{1}$ since 
\[
\int_{0}^{T}|\mathbf{v}_{t}|(K)dt\leq C_{K,\eta}\int_{0}^{T}\sqrt{\mathcal{A}(\rho_{t},\mathbf{v}_{t};\pi)}dt<\infty.
\]
Therefore, when $U\subset\subset G$ we have, by \cite[Thm. C.7(iv)]{evans2010partial},
that $h_{\sigma}*\mathbf{v}_{t}(U)\rightarrow\mathbf{v}_{t}(U)$ for
Lebesgue-almost every $t\in[0,T]$ (since $|\mathbf{v}_{t}|(U)\leq|\mathbf{v}_{t}|(K)$).
Quantifying over a countable family of $U$'s which generate the open
sets of $K\subset\subset G$, we see that: for Lebesgue-almost all
$t\in[0,1]$, it holds for every open $U\subset K$ that $h_{\sigma}*\mathbf{v}_{t}(U)\rightarrow\mathbf{v}_{t}(U)$.
This implies that, for Lebesgue-almost all $t\in[0,1]$, it holds
that $h_{\sigma}*\mathbf{v}_{t}$ converges locally weak{*}ly to $\mathbf{v}_{t}$
(since this latter type of convergence is equivalent to $\liminf_{\sigma\rightarrow0}|h_{\sigma}*\mathbf{v}_{t}|(U)\geq|\mathbf{v}_{t}|(U)$
for every compactly contained open $U\subset G$). 
\end{proof}
\begin{prop}
\label{prop:time-convolution-action}Let $(\rho_{t},\mathbf{v}_{t})_{t\in[0,1]}$
be a solution to the nonlocal continuity equation with finite total
action $\int_{0}^{1}\mathcal{A}(\rho_{t},\mathbf{v}_{t};\pi)dt$.
Let $\delta>0$. As $\sigma\rightarrow0$, it holds that 
\[
\mathcal{A}(h_{\sigma}*\rho_{t},h_{\sigma}*\mathbf{v}_{t};\pi)\rightarrow\mathcal{A}(\rho_{t},\mathbf{v}_{t};\pi)
\]
for almost all $t\in[0,1]$, and 
\[
\int_{0}^{T}\mathcal{A}(h_{\sigma}*\rho_{t},h_{\sigma}*\mathbf{v}_{t};\pi)dt\rightarrow\int_{0}^{T}\mathcal{A}(\rho_{t},\mathbf{v}_{t};\pi)dt.
\]
\end{prop}

\begin{proof}
We first claim that 
\[
\mathcal{A}(h_{\sigma}*\rho_{t},h_{\sigma}*\mathbf{v}_{t};\pi)dt\leq\int_{-\sigma}^{\sigma}\mathcal{A}(\rho_{t-s},\mathbf{v}_{t-s};\pi)h_{\sigma}(t-s)ds=h_{\sigma}*\mathcal{A}(\rho_{t},\mathbf{v}_{t};\pi).
\]
We argue along similar lines to the proof of Lemma \ref{lem:action-spatial-convolution}.
Let $S_{1},S_{2}\ldots$ be random times which are i.i.d. and distributed
accoding to $h_{\sigma}(t-s)ds$. Then we claim that almost surely,
\[
\frac{1}{N}\sum_{i=1}^{N}\rho_{S_{i}}\rightharpoonup h_{\sigma}*\rho_{t}\text{ and }\frac{1}{N}\sum_{i=1}^{N}\mathbf{v}_{S_{i}}\rightharpoonup^{*}h_{\sigma}*\mathbf{v}_{t}.
\]
We only establish the latter since the former is easier to prove.
Let $\mathbf{v}_{t}=\mathbf{v}_{t}^{+}-\mathbf{v}_{t}^{-}$ denote
the Jordan decomposition of $\mathbf{v}$. It suffices to show that:
for every compact $K\subset G$, and every $\varphi\in C_{C}(K)$,
with probability 1, 
\[
\frac{1}{N}\sum_{i=1}^{N}\iint_{K}\varphi d\left(\mathbf{v}_{S_{i}}^{+}\right)\rightarrow\iint_{K}\varphi d\left(h_{\sigma}*\mathbf{v}_{t}^{+}\right)
\]
and similarly for $\mathbf{v}_{t}^{-}$. We show this convergence
holds by appealing to the strong law of large numbers. Indeed, 
\[
\left|\iint_{K}\varphi d\left(\mathbf{v}_{S_{i}}^{+}\right)\right|\leq\Vert\varphi\Vert_{L^{\infty}(K)}C_{K,\eta}\sqrt{\mathcal{A}(\rho_{S_{i}},\mathbf{v}_{S_{i}};\pi)}.
\]
Under the assumption that $\int_{0}^{1}\mathcal{A}(\rho_{t},\mathbf{v}_{t};\pi)dt<\infty$,
we have that $\int_{-\sigma}^{\sigma}\mathcal{A}(\rho_{s},\mathbf{v}_{s};\pi)h_{\sigma}(t-s)ds<\infty$
also. Hence $S\mapsto\iint_{K}\varphi d\left(\mathbf{v}_{S}^{+}\right)$
is an $L^{2}$ random variable, so in particular the strong law of
large numbers applies, and with probability 1, 
\[
\frac{1}{N}\sum_{i=1}^{N}\iint_{K}\varphi d\left(\mathbf{v}_{S_{i}}^{+}\right)\rightarrow\int_{-\sigma}^{\sigma}\iint_{K}\varphi d\left(\mathbf{v}_{s}^{+}\right)h_{\sigma}(t-s)ds=\iint_{K}\varphi d\left(h_{\sigma}*\mathbf{v}_{t}^{+}\right).
\]
Since $K$ was arbitrary, this shows that $\frac{1}{N}\sum_{i=1}^{N}\mathbf{v}_{S_{i}}^{+}\rightharpoonup^{*}h_{\sigma}*\mathbf{v}_{t}^{+}$
almost surely. The same reasoning applies to $\mathbf{v}_{t}^{-}$.
Finally, $h_{\sigma}*(\mathbf{v}_{t}^{+}-\mathbf{v}_{t}^{-})=h_{\sigma}*\mathbf{v}_{t}^{+}-h_{\sigma}*\mathbf{v}_{t}^{-}$
so altogether we have $\frac{1}{N}\sum_{i=1}^{N}\mathbf{v}_{S_{i}}\rightharpoonup^{*}h_{\sigma}*\mathbf{v}_{t}$.

As a consequence, we have, by the joint convexity and lower semicontinuity
of $\mathcal{A}$, that 
\begin{align*}
\mathcal{A}(h_{\sigma}*\rho_{t},h_{\sigma}*\mathbf{v}_{t};\pi)dt & \leq\liminf_{N\rightarrow\infty}\mathcal{A}\left(\frac{1}{N}\sum_{i=1}^{N}\rho_{S_{i}},\frac{1}{N}\sum_{i=1}^{N}\mathbf{v}_{S_{i}};\pi\right)\\
 & \leq\liminf_{N\rightarrow\infty}\frac{1}{N}\sum_{i=1}^{N}\mathcal{A}(\rho_{S_{i}},\mathbf{v}_{S_{i}};\pi).
\end{align*}
Now, since $t\mapsto\mathcal{A}(\rho_{t},\mathbf{v}_{t};\pi)$ is
measurable, and $\int_{-\sigma}^{\sigma}\mathcal{A}(\rho_{s},\mathbf{v}_{s};\pi)h_{\sigma}(t-s)ds<\infty$,
we have that $S\mapsto\mathcal{A}(\rho_{S},\mathbf{v}_{S};\pi)$ is
an $L^{1}$ random variable, hence $\frac{1}{N}\sum_{i=1}^{N}\mathcal{A}(\rho_{S_{i}},\mathbf{v}_{S_{i}};\pi)\rightarrow\int_{-\sigma}^{\sigma}\mathcal{A}(\rho_{s},\mathbf{v}_{s};\pi)h_{\sigma}(t-s)ds$
almost surely. This shows that $\mathcal{A}(h_{\sigma}*\rho_{t},h_{\sigma}*\mathbf{v}_{t};\pi)dt\leq\int_{-\sigma}^{\sigma}\mathcal{A}(\rho_{s},\mathbf{v}_{s};\pi)h_{\sigma}(t-s)ds$
as desired.

Now, observe that 
\begin{align*}
\int_{0}^{T}\mathcal{A}(h_{\sigma}*\rho_{t},h_{\sigma}*\mathbf{v}_{t};\pi)dt & \leq\int_{0}^{T}\int_{-\sigma}^{\sigma}\mathcal{A}(\rho_{s},\mathbf{v}_{s};\pi)h_{\sigma}(t-s)dsdt\\
 & \leq\int_{-T}^{2T}\mathcal{A}(\rho_{t},\mathbf{v}_{t};\pi)dt\\
 & =3\int_{0}^{T}\mathcal{A}(\rho_{t},\mathbf{v}_{t};\pi)dt.
\end{align*}
  And so, assuming that $\int_{0}^{T}\mathcal{A}(\rho_{t},\mathbf{v}_{t};\pi)dt<\infty$,
it follows that 
\[
t\mapsto\mathcal{A}(h_{\sigma}*\rho_{t},h_{\sigma}*\mathbf{v}_{t};\pi)\in L^{1}([0,T]).
\]
First, we show that for almost all $t$, $\mathcal{A}(h_{\sigma}*\rho_{t},h_{\sigma}*\mathbf{v}_{t};\pi)\rightarrow\mathcal{A}(\rho_{t},\mathbf{v}_{t};\pi)$
as $\sigma\rightarrow0$. On the one hand, from the $t$-a.s. weak{*}
(resp. local weak{*}) convergence of $h_{\sigma}*\rho_{t}$ to $\rho_{t}$
and $h_{\sigma}*\mathbf{v}_{t}$ to $\mathbf{v}_{t}$, we have that
\[
\liminf_{\sigma\rightarrow0}\mathcal{A}(h_{\sigma}*\rho_{t},h_{\sigma}*\mathbf{v}_{t};\pi)\geq\mathcal{A}(\rho_{t},\mathbf{v}_{t};\pi)\text{ }t-a.s..
\]
At the same time, from the fact that $t\mapsto\mathcal{A}(\rho_{t},\mathbf{v}_{t};\pi)$
is $L^{1}$, we have 
\begin{align*}
\mathcal{A}(h_{\sigma}*\rho_{t},h_{\sigma}*\mathbf{v}_{t};\pi) & \leq h_{\sigma}*\mathcal{A}(\rho_{t},\mathbf{v}_{t};\pi)\\
 & \rightarrow\mathcal{A}(\rho_{t},\mathbf{v}_{t};\pi)\text{ }t-a.s.
\end{align*}
so that $\mathcal{A}(h_{\sigma}*\rho_{t},h_{\sigma}*\mathbf{v}_{t};\pi)\rightarrow\mathcal{A}(\rho_{t},\mathbf{v}_{t};\pi)$
for almost all $t\in[0,1]$. Also, from Fatou's lemma, we have that
\begin{align*}
\liminf_{\sigma\rightarrow0}\int_{0}^{T}\mathcal{A}(h_{\sigma}*\rho_{t},h_{\sigma}*\mathbf{v}_{t};\pi)dt & \geq\int_{0}^{T}\liminf_{\sigma\rightarrow0}\mathcal{A}(h_{\sigma}*\rho_{t},h_{\sigma}*\mathbf{v}_{t};\pi)dt\\
 & \geq\int_{0}^{T}\mathcal{A}(\rho_{t},\mathbf{v}_{t};\pi)dt
\end{align*}
but also 
\begin{align*}
\limsup_{\sigma\rightarrow0}\int_{0}^{T}\mathcal{A}(h_{\sigma}*\rho_{t},h_{\sigma}*\mathbf{v}_{t};\pi)dt & \leq\limsup_{\sigma\rightarrow0}\int_{0}^{T}h_{\sigma}*\mathcal{A}(\rho_{t},\mathbf{v}_{t};\pi)dt\\
 & =\int_{0}^{T}\mathcal{A}(\rho_{t},\mathbf{v}_{t};\pi)dt
\end{align*}
where in the last line we have used \cite[Thm. C.7(iv)]{evans2010partial}.
Hence we actually have that 
\[
\int_{0}^{T}\mathcal{A}(h_{\sigma}*\rho_{t},h_{\sigma}*\mathbf{v}_{t};\pi)dt\rightarrow\int_{0}^{T}\mathcal{A}(\rho_{t},\mathbf{v}_{t};\pi)dt
\]
as $\sigma\rightarrow0$. 
\end{proof}
\begin{lem}
\label{lem:smoothed-density-regularity}Let $\mu_{t}\in\mathcal{P}(X)$
and suppose that $t\mapsto\mu_{t}$ is narrowly continuous. Let $k_{\delta}*$
and $h_{\sigma}*$ denote space- and time-mollification with respect
to standard mollifiers. Then, $k_{\delta}*\mu_{t}\ll\text{Leb}^{d}$,
and its density $k_{\delta}*\mu_{t}(x)$ and time-mollified density
$h_{\sigma}*k_{\delta}*\mu_{t}(x)$ have the following regularity
properties:
\begin{enumerate}
\item $k_{\delta}*\mu_{t}(x)\in C^{1}(X)$ for fixed $\delta,t$, and likewise
$h_{\sigma}*k_{\delta}*\mu_{t}(x)\in C^{1}(X)$, with a bound on $C^{1}$
norm that depends only on $\delta$;
\item $h_{\sigma}*k_{\delta}*\mu_{t}(x)\in C^{1}([0,T])$ for fixed $\delta,x$
with a bound on $C^{1}$ norm that depends only on $\delta$ and $\sigma$;
\item For each $t\in[0,1]$, as $\sigma\rightarrow0$, $h_{\sigma}*k_{\delta}*\mu_{t}(x)\rightarrow k_{\delta}*\mu_{t}(x)$
uniformly on compact subsets of $X$. 
\end{enumerate}
\end{lem}

\begin{proof}
(i) The density $k_{\delta}*\mu_{t}(x)$ is bounded uniformly in $t$
since 
\[
\sup_{t}\Vert k_{\delta}*\mu_{t}\Vert_{L^{\infty}(X)}\leq\int_{X}\Vert k_{\delta}(x-y)\Vert_{\infty}d\mu_{t}(y)\leq O(\delta^{-d}).
\]
Likewise, it holds that $k_{\delta}*\mu_{t}(x)$ is smooth by standard
properties of mollifiers, and $h_{\sigma}*k_{\delta}*\mu_{t}(x)$
is smooth also. To wit, the fact that $k_{\delta}(x-y)$ is bounded
means that we can apply the dominated convergence theorem to differentiate
under the integral sign:
\[
\frac{d}{dx}k_{\delta}*\mu_{t}(x)=\int_{X}\frac{d}{dx}k_{\delta}(x-y)d\mu_{t}(y).
\]
And so $\sup_{x}\frac{d}{dx}k_{\delta}*\mu_{t}(x)\leq\sup_{x}\frac{d}{dx}k_{\delta}(x)$.
Similarly 
\[
\left|\frac{d}{dx}k_{\delta}*\mu_{t}(x)-\frac{d}{dx^{\prime}}k_{\delta}*\mu_{t}(x^{\prime})\right|\leq\int_{X}\left|\frac{d}{dx}k_{\delta}(x-y)-\frac{d}{dx^{\prime}}k_{\delta}(x^{\prime}-y)\right|d\mu_{t}(y)
\]
and so the derivative of $k_{\delta}*\mu_{t}(x)$ automatically has
the same modulus of uniform continuity as that of $\frac{d}{dx}k_{\delta}(x)$.
This continuity is inherited by $h_{\sigma}*k_{\delta}*\mu_{t}$ since
again by the Leibniz integration rule,
\[
\frac{d}{dx}h_{\sigma}*k_{\delta}*\mu_{t}(x)=\int_{-\sigma}^{\sigma}\left(\int_{X}\frac{d}{dx}k_{\delta}(x-y)d\mu_{t}(y)\right)h_{\sigma}(|t-s|)ds
\]
and thus 
\[
\left|\frac{d}{dx}h_{\sigma}*k_{\delta}*\mu_{t}(x)-\frac{d}{dx^{\prime}}h_{\sigma}*k_{\delta}*\mu_{t}(x^{\prime})\right|\leq\int_{-\sigma}^{\sigma}\int_{X}\left|\frac{d}{dx}k_{\delta}(x-y)-\frac{d}{dx^{\prime}}k_{\delta}(x^{\prime}-y)\right|d\mu_{t}(y)h_{\sigma}(|t-s|)ds.
\]
Since we assume $k_{\delta}$ to be a standard mollifier which in
particular satisfies $\Vert\frac{d}{dx}k_{\delta}(x)\Vert_{\infty}\leq O(\delta^{-1})$,
this establishes claim (i). 

(ii) For fixed $\delta,x$ we have that $t\mapsto k_{\delta}*\mu_{t}(x)$
is uniformly bounded by $O(\delta^{-d})$, hence we can differentiate
under the integral sign:
\begin{align*}
\frac{d}{dt}h_{\sigma}*k_{\delta}*\mu_{t}(x) & =\frac{d}{dt}\int_{-\sigma}^{\sigma}k_{\delta}*\mu_{s}(x)h_{\sigma}(t-s)ds\\
 & =\int_{-\sigma}^{\sigma}k_{\delta}*\mu_{s}(x)\frac{d}{dt}h_{\sigma}(t-s)ds\\
 & \leq\sup_{t\in[0,T]}k_{\delta}*\mu_{s}(x)\cdot\int_{-\sigma}^{\sigma}\frac{d}{dt}h_{\sigma}(t-s)ds.
\end{align*}
Since we assume $h_{\sigma}$ to be a standard mollifier, we have
$\Vert\frac{d}{dt}h_{\sigma}(t)\Vert_{\infty}\leq O(\sigma^{-1})$,
and therefore $\Vert\frac{d}{dt}h_{\sigma}*k_{\delta}*\mu_{t}(x)\Vert_{L^{\infty}[0,T]}\leq O(\delta^{-d}\sigma^{-1})$
uniformly in $x$.

(iii) First compute that 
\[
|k_{\delta}*\mu_{t}(x)-k_{\delta}*\mu_{t}(x^{\prime})|\leq\int_{X}|k_{\delta}(y-x)-k_{\delta}(y-x^{\prime})|d\mu_{t}(y).
\]
Note that $k_{\delta}$ is compactly supported, hence uniformly continuous,
and also translation invariant. Letting $\Phi_{\delta}$ be a modulus
of uniform continuity for $k_{\delta}(|0-\cdot|)$, we see that $|k_{\delta}*\mu_{t}(x)-k_{\delta}*\mu_{t}(x^{\prime})|\leq\Phi_{\delta}(|x-x^{\prime}|)$.
Note that this modulus of uniform continuity is independent of $\mu_{t}$.
Hence the family $\{k_{\delta}*\mu_{t}(x)\}_{t\in[0,1]}$ is uniformly
equicontinuous with modulus $\Phi_{\delta}$.

Furthermore, we have that 
\begin{align*}
|h_{\sigma}*k_{\delta}*\mu_{t}(x)-h_{\sigma}*k_{\delta}*\mu_{t}(x^{\prime})| & \leq\int_{-\sigma}^{\sigma}|k_{\delta}*\mu_{t}(x)-k_{\delta}*\mu_{t}(x^{\prime})|h_{\sigma}(t)dt\\
 & \leq\Phi_{\delta}(|x-x^{\prime}|)
\end{align*}
as well, which means that the larger family $\{h_{\sigma}*k_{\delta}*\mu_{t}(x)\}_{t\in[0,1];\sigma\geq0}$
is uniformly equicontinuous with modulus $\Phi_{\sigma}$. 

Next, we observe that $t\mapsto k_{\delta}*\mu_{t}(x)$ is continuous,
since $t\mapsto\mu_{t}$ is narrowly continuous and $k_{\delta}*\mu_{t}(x)=\int k_{\delta}(|x-y|)d\mu_{t}(y)$
is a duality pairing between the bounded continuous function $k_{\delta}(|\cdot-y|)$
and the probability measure $\mu_{t}$. Therefore, we have that for
fixed $x$, we have $h_{\sigma}*k_{\delta}*\mu_{t}(x)$ converges
uniformly in $t$ to $k_{\delta}*\mu_{t}(x)$ by standard convergence
properties of mollifiers of continuous functions on compact sets,
see \cite[Thm. C.7(ii)]{evans2010partial}. 

In particular, for fixed $t$, we have that $h_{\sigma}*k_{\delta}*\mu_{t}(x)\rightarrow k_{\delta}*\mu_{t}(x)$
pointwise in $x$. Also we have that $k_{\delta}*\mu_{t}(x)$ is uniformly
bounded by $k(0)/\delta^{d}$ since $k$ is a standard mollifier,
and so the same holds for $h_{\sigma}*k_{\delta}*\mu_{t}(x)$. Let
$\sigma_{n}\rightarrow0$ be any subsequence. Restricting to any compact
$K\subset X$, we find, by the Arzela-Ascoli theorem, that along a
further subsequence of $\sigma_{n}$ (which we do not relabel), $h_{\sigma_{n}}*k_{\delta}*\mu_{t}(x)$
converges uniformly; and since we know the pointwise limit is $k_{\delta}*\mu_{t}(x)$,
we have that $h_{\sigma_{n}}*k_{\delta}*\mu_{t}(x)$ converges uniformly
on $K$. By choosing a countable compact exhaustion of $X$ and selecting
a diagonal subsequence, we produce a further subsequence of $\sigma_{n}$
(which we do not relabel) such that $h_{\sigma_{n}}*k_{\delta}*\mu_{t}(x)$
converges uniformly on compacts to $k_{\delta}*\mu_{t}(x)$. This
establishes that every subsequence $\sigma_{n}$ contains a further
subsequence along which $h_{\sigma_{n}}*k_{\delta}*\mu_{t}(x)$ converges
uniformly on compacts to $k_{\delta}*\mu_{t}(x)$, hence we finally
conclude that $h_{\sigma}*k_{\delta}*\mu_{t}(x)\rightarrow k_{\delta}*\mu_{t}(x)$
uniformly on compact subsets of $X$.
\end{proof}

\section{Deferred proof of the chain rule\protect\label{sec:Deferred-proofs}}

In the proof of Proposition \ref{prop:one=000020sided=000020chain=000020rule},
we shall make repeated use of the following variant of the dominated
convergence theorem, from \cite[Ex. 2.3.20]{folland1999real}.
\begin{thm*}[``Generalized Lebesgue dominated convergence theorem'']
 Let $f_{n},g_{n},f,g\in L^{1}(X,\mu)$. Suppose $f_{n}\rightarrow f$
and $g_{n}\rightarrow g$ pointwise $\mu$-almost everywhere. If $|f_{n}|\leq g_{n}$
pointwise $\mu$-almost everywhere for each $n$, and $\int g_{n}d\mu\rightarrow\int gd\mu$,
then $\int f_{n}d\mu\rightarrow\int fd\mu$ also.
\end{thm*}
\begin{proof}[Proof of Proposition \ref{prop:one=000020sided=000020chain=000020rule}]
 We suppose that $\rho_{t}\ll\pi$ for almost all $t\in[0,T]$, since
otherwise the claimed inequality holds trivially. Also, note that
$\int_{0}^{T}\sqrt{\mathcal{I}(\rho_{t}\mid\pi)}|\rho_{t}^{\prime}|dt<\infty$,
since $\int_{0}^{T}|\rho_{t}^{\prime}|^{2}dt=\int_{0}^{T}\mathcal{A}(\rho_{t},\mathbf{v}_{t})dt<\infty$
, and together with $\int_{0}^{T}\mathcal{I}(\rho_{t}\mid\pi)dt<\infty$
we have $\int_{0}^{T}\sqrt{\mathcal{I}(\rho_{t}\mid\pi)}|\rho_{t}^{\prime}|dt<\infty$
by Young's inequality. %
\begin{comment}
Additionally, we note that the integral $\int_{0}^{T}\sqrt{\mathcal{I}(\rho_{t}\mid\pi)}|\rho_{t}^{\prime}|dt$
is invariant under time reparametrization, so without loss of generality
we can take $(\rho_{t})_{t\in[0,T]}$ to be constant speed (this relies
on the general existence of arc-length reparametrization of a.c. curves
in metric spaces, see \cite[Lem. 1.1.4 (b)]{ambrosio2008gradient}).
Consequently, we can also assume below that $\int_{0}^{T}\sqrt{\mathcal{I}(\rho_{t}\mid\pi)}dt<\infty$.
\end{comment}
{} 

\emph{Step 0: regularization}. First, we replace the functional $\mathcal{H}(\cdot\mid\pi)$
with $\mathcal{H}^{n}(\cdot\mid\pi)$, as follows. Recall that $\mathcal{H}(\cdot\mid\pi)$
can be rewritten as an integral functional with a nonnegative integrand,
as for all $\rho\in\mathcal{P}(X)$,
\[
\int\frac{d\rho}{d\pi}\log\frac{d\rho}{d\pi}d\pi=\int\left(\frac{d\rho}{d\pi}\log\frac{d\rho}{d\pi}+1-\frac{d\rho}{d\pi}\right)d\pi
\]
and the function $x\log x+1-x$ is nonnegative. Note that the derivative
of $x\log x+1-x$ is simply $\log x$. Now, define $\tilde{g}_{n}:[0,\infty)\rightarrow[-n,n]$
by 
\[
\tilde{g}_{n}(x):=\begin{cases}
-n & \log x\leq-n\\
\log x & -n<\log x<n\\
n & n\leq\log x
\end{cases}
\]
and let $g_{n}(u)$ be a smooth approximation of $\tilde{g}_{n}(u)$
--- more precisely, we take $g_{n}(u)$ to be $C^{1}$, non-decreasing,
bounded between $-n$ and $n$, and monotonically converging to $\log u$
as $n\rightarrow\infty$ (that is, monotone downwards for $0<u<1$
and monotone upwards for $u\geq1$). Then, we define $f_{n}:[0,\infty)\rightarrow\mathbb{R}$
by $f_{n}(x):=\int_{1}^{x}g_{n}(s)ds$ (with the usual convention
that $\int_{1}^{x}g_{n}(s)ds=-\int_{x}^{1}g_{n}(s)ds$), and then
define
\[
\mathcal{H}^{n}(\rho\mid\pi):=\begin{cases}
\int_{X}f_{n}\left(\frac{d\rho}{d\pi}\right)d\pi & \rho\ll\pi\\
\infty & \text{else}.
\end{cases}
\]
Note that the minimum for both $x\log x-x+1$ and $f_{n}(x)$ is attained
in both cases at $x=1$, with the value 0. We further note that $f_{n}$
is the integral of a non-decreasing function, hence convex, and moreover
$f_{n}(x)\leq x\log x-x+1$ by construction; thus, $\mathcal{H}^{n}(\rho\mid\pi)$
is a convex integral functional which converges monotonically upward
to $\mathcal{H}(\rho\mid\pi)$ as $n\rightarrow\infty$.

Having regularized the entropy functional, we next regularize the
solution to the nonlocal continuity equation. First we mollify $(\rho_{t},\mathbf{v}_{t})_{t\in[0,1]}$
in space: let $k$ be a convolution kernel on $X$, supported on $B(0,1)$,
and consider $(k_{\delta}*\rho_{t},k_{\delta}*\mathbf{v}_{t})_{t\in[0,1]}$.
Note that this is still a solution to the nonlocal continuity equation,
and that the action is stable with respect to this regularization.
In what follows, we take $\delta<\delta_{0}$. 

Additionally, we \emph{also} mollify $(\rho_{t},\mathbf{v}_{t})_{t\in[0,T]}$
in time, so as to argue as in the proof of \cite[Prop. 3.10]{esposito2019nonlocal}.
Namely, let us extend $(\rho_{t},\mathbf{v}_{t})$ by reflection to
be defined on the time interval $[-T,2T]$, and take $h(t)$ to be
a standard mollifier supported on $[-1,1]$. For $\sigma\in(0,T]$,
we recall that the time-convolution of $(k_{\delta}*\rho_{t},k_{\delta}*\mathbf{v}_{t})_{t\in[0,1]}$
is defined as follows:
\[
\forall A\in\mathcal{B}(\mathbb{R}^{d})\qquad h_{\sigma}*(k_{\delta}*\rho_{t})(A):=\int_{-\sigma}^{\sigma}h_{\sigma}(t-s)(k_{\delta}*\rho_{s})(A)ds;
\]
\[
\forall U\in\mathcal{B}(G)\qquad h_{\sigma}*(k_{\delta}*\mathbf{v}_{t})(U):=\int_{-\sigma}^{\sigma}h_{\sigma}(t-s)(k_{\delta}*\mathbf{v}_{s})(U)ds.
\]
(Here $h_{\sigma}(t):=\frac{1}{\sigma}h(t/\sigma)$.) By Lemma \ref{lem:time-mollified-measures-convergence},
it holds that as $\sigma\rightarrow0$, we have the narrow (resp.
local weak{*}) convergence of $h_{\sigma}*(k_{\delta}*\rho_{t})$
to $k_{\delta}*\rho_{t}$ for all $t\in[0,1]$, and $h_{\sigma}*(k_{\delta}*\mathbf{v}_{t})$
to $k_{\delta}*\mathbf{v}_{t}$ for almost all $t\in[0,1]$.%
\begin{comment}
{[}todo but same as in EPSS --- except that $k_{\delta}*\mathbf{v}_{t}$
is not a tangent flux! Maybe need to interchange order of convolutions.
Hm{]}
\end{comment}
{} %
{}  

On the other hand, $k_{\delta}*\rho_{t}$ converges narrowly to $\rho_{t}$,
and $d(k_{\delta}*\mathbf{v}_{t})dt\rightharpoonup^{*}d\mathbf{v}_{t}dt$
in $\mathcal{M}_{loc}([0,T]\times G)$ by Corollary \ref{cor:action-space-convolution-convergence}.

Let us verify that $\partial_{t}\rho_{t}^{\sigma}$ is well-defined.
Indeed, $k_{\delta}*\rho_{t}$ has density $k_{\delta}*\rho_{t}(x)$,
since $k_{\delta}*\rho_{t}=\int k_{\delta}(x-y)d\rho_{t}(y)dx$. Thus,
we have that 
\[
(h_{\sigma}*k_{\delta}*\rho_{t})(A)=\int_{-\sigma}^{\sigma}k_{\delta}*\rho_{t}(A)h_{\sigma}(|t-s|)ds=\int_{-\sigma}^{\sigma}\left(\int_{A}k_{\delta}*\rho_{t}(x)dx\right)h_{\sigma}(|t-s|)ds
\]
and so by Fubini's theorem 
\[
\int_{-\sigma}^{\sigma}\left(\int_{A}k_{\delta}*\rho_{t}(x)dx\right)h_{\sigma}(|t-s|)ds=\int_{A}\int_{-\sigma}^{\sigma}k_{\delta}*\rho_{t}(x)h_{\sigma}(|t-s|)dsdx.
\]
Therefore the density $(h_{\sigma}*k_{\delta}*\rho_{t})(x)$ exists
as well and is given by $\int_{-\sigma}^{\sigma}k_{\delta}*\rho_{t}(x)h_{\sigma}(|t-s|)ds$.
For fixed $x$, differentiability of $t\mapsto h_{\sigma}*k_{\delta}*\rho_{t}(x)$
is given in Lemma \ref{lem:smoothed-density-regularity}(ii).

Lastly, we introduce one further regularization, for the following
purpose: below we shall need regularizations $\tilde{\rho}_{t}$ and
$\tilde{\pi}$ of $\rho_{t}$ and $\pi$ which satisfy $\Vert\frac{d\tilde{\rho}_{t}}{d\tilde{\pi}}\Vert_{C^{1}}<\infty$,
such that $\tilde{\rho}_{t}$ is still a curve of finite action. In
particular, the preceding calculations show that $\frac{d(h_{\sigma}*k_{\delta}*\rho_{t})}{dx}$
is differentiable in both space and time, and Lemma \ref{lem:smoothed-density-regularity}
establishes that $\Vert\frac{d(h_{\sigma}*k_{\delta}*\rho_{t})}{dx}\Vert_{C^{1}}\leq C(\delta)$
for some uniform constant $C(\delta)$. Since the Radon-Nikodym theorem
tells us that 
\[
\frac{d(h_{\sigma}*k_{\delta}*\rho_{t})}{d(k_{\delta}*\pi)}=\frac{d(h_{\sigma}*k_{\delta}*\rho_{t})}{dx}\frac{dx}{d(k_{\delta}*\pi)}=\frac{d(h_{\sigma}*k_{\delta}*\rho_{t})}{dx}\left(\frac{d(k_{\delta}*\pi)}{dx}\right)^{-1}
\]
we would like for it to be the case that $\left(\frac{d(k_{\delta}*\pi)}{dx}\right)^{-1}$
is $C^{1}$ (in space), however this does not hold when $X=\mathbb{R}^{d}$.
Accordingly, we replace $k_{\delta}*\pi$ with $(1-\epsilon)k_{\delta}*\pi+\epsilon\text{Leb}$.
More precisely, let $\mu_{0}$ be any probability measure with smooth
density satisfying $\mathcal{H}(\mu_{0}\mid\text{Leb})<\infty$ and
$\mathcal{I}(\mu_{0}\mid\text{Leb})<\infty$. Then we have that 
\[
((1-\epsilon)h_{\sigma}*k_{\delta}*\rho_{t}+\epsilon\mu_{0},(1-\epsilon)h_{\sigma}*k_{\delta}*\mathbf{v}_{t}+\epsilon0)_{t\in[0,1]}
\]
 solves the nonlocal continuity equation. From joint convexity of
the action we have 
\begin{multline*}
\mathcal{A}((1-\epsilon)h_{\sigma}*k_{\delta}*\rho_{t}+\epsilon\mu_{0},(1-\epsilon)h_{\sigma}*k_{\delta}*\mathbf{v}_{t});(1-\epsilon)k_{\delta}*\pi+\epsilon\text{Leb})\\
\begin{aligned} & \leq(1-\epsilon)\mathcal{A}(h_{\sigma}*k_{\delta}*\rho_{t},h_{\sigma}*k_{\delta}*\mathbf{v}_{t};k_{\delta}*\pi)+\epsilon\mathcal{A}(\mu_{0},0;\text{Leb})\\
 & =(1-\epsilon)\mathcal{A}(h_{\sigma}*k_{\delta}*\rho_{t},h_{\sigma}*k_{\delta}*\mathbf{v}_{t};k_{\delta}*\pi);
\end{aligned}
\end{multline*}
and so the action is stable with respect to this regularization. 

As shorthand, we typically write $\rho_{t}^{\delta,\epsilon,\sigma}:=(1-\epsilon)h_{\sigma}*k_{\delta}*\rho_{t}+\epsilon\mu_{0}$,
$\mathbf{v}_{t}^{\delta,\epsilon,\sigma}:=(1-\epsilon)h_{\sigma}*k_{\delta}*\mathbf{v}_{t}$,
and $\pi^{\delta,\epsilon}=(1-\epsilon)k_{\delta}*\pi+\epsilon\text{Leb}$
below. We drop the small parameter superscripts when that parameter
is equal to zero, so that for instance $\rho_{t}^{\delta,\epsilon}:=(1-\epsilon)k_{\delta}*\rho_{t}+\epsilon\mu_{0}$.

Observe that
\[
\left(\frac{d((1-\epsilon)k_{\delta}*\pi+\epsilon\text{Leb})}{dx}\right)^{-1}=\left((1-\epsilon)\frac{d(k_{\delta}*\pi)}{dx}+\epsilon\right)^{-1}\leq\epsilon^{-1}
\]
so indeed $\left(\frac{d((1-\epsilon)k_{\delta}*\pi+\epsilon\text{Leb})}{dx}\right)^{-1}$
(and hence $\frac{d(h_{\sigma}*k_{\delta}*\rho_{t})}{dx}\left(\frac{d((1-\epsilon)k_{\delta}*\pi+\epsilon\text{Leb})}{dx}\right)^{-1}$)
is bounded. Likewise we have that the gradient (in $x$) of $\left(\frac{d((1-\epsilon)k_{\delta}*\pi+\epsilon\text{Leb})}{dx}\right)^{-1}$
is equal to 
\[
-\frac{1}{\left(\frac{d((1-\epsilon)k_{\delta}*\pi+\epsilon\text{Leb})}{dx}\right)^{2}}(1-\epsilon)\nabla\frac{d(k_{\delta}*\pi)}{dx}
\]
and since $\frac{1}{\left(\frac{d((1-\epsilon)k_{\delta}*\pi+\epsilon\text{Leb})}{dx}\right)^{2}}\leq\epsilon^{-2}$
we have $\Vert\nabla\left(\frac{d((1-\epsilon)k_{\delta}*\pi+\epsilon\text{Leb})}{dx}\right)^{-1}\Vert\leq\epsilon^{-2}\Vert\nabla\frac{d(k_{\delta}*\pi)}{dx}\Vert.$
Altogether this shows that the $C^{1}$ norm of $\left(\frac{d((1-\epsilon)k_{\delta}*\pi+\epsilon\text{Leb})}{dx}\right)^{-1}$
is bounded by some constant $C(\delta,\epsilon)$ which goes to infinity
as $\delta,\epsilon\rightarrow0$. Consequently, $\Vert\frac{d\rho_{t}^{\delta,\epsilon,\sigma}}{d\pi^{\delta,\epsilon}}\Vert_{C^{1}(X)}<\infty$
also, and likewise $\Vert g_{n}\left(\frac{d\rho_{t}^{\delta,\epsilon,\sigma}}{d\pi^{\delta,\epsilon}}\right)\Vert_{C^{1}(X)}<\infty$. 

Similarly, time-regularity of $f_{n}\left(\frac{d\rho_{t}^{\delta,\epsilon,\sigma}}{d\pi^{\delta,\epsilon}}\right)$
follows from the fact that 
\begin{align*}
\frac{\partial}{\partial t}f_{n}\left(\frac{d\rho_{t}^{\delta,\epsilon,\sigma}}{d\pi^{\delta,\epsilon}}(x)\right) & =g_{n}\left(\frac{d\rho_{t}^{\delta,\epsilon,\sigma}}{d\pi^{\delta,\epsilon}}(x)\right)\partial_{t}\left(\frac{d\rho_{t}^{\delta,\epsilon,\sigma}}{d\pi^{\delta,\epsilon}}\right)(x)\\
 & =g_{n}\left(\frac{d\rho_{t}^{\delta,\epsilon,\sigma}}{d\pi^{\delta,\epsilon}}(x)\right)\left(\frac{d\pi^{\delta,\epsilon}}{dx}\right)^{-1}\partial_{t}\left(\frac{d\rho_{t}^{\delta,\epsilon,\sigma}}{dx}\right)(x)\\
 & \leq\Vert g_{n}\Vert_{\infty}\epsilon^{-1}O(\delta^{-d}\sigma^{-1})
\end{align*}
where in the last line we have used Lemma \ref{lem:smoothed-density-regularity}(ii).

For fixed $n$, the preceding computation indicates that we can apply
the dominated convergence theorem and differentiate under the integral
sign, to compute that 
\[
\frac{d}{dt}\mathcal{H}^{n}(\rho_{t}^{\delta,\epsilon,\sigma}\mid\pi^{\delta,\epsilon})=\int_{X}g_{n}\left(\frac{d\rho_{t}^{\delta,\epsilon,\sigma}}{d\pi^{\delta,\epsilon}}\right)\partial_{t}\left(\frac{d\rho_{t}^{\delta,\epsilon,\sigma}}{d\pi^{\delta,\epsilon}}\right)d\pi^{\delta,\epsilon}.
\]
Now, since $\frac{d\rho_{t}^{\delta,\epsilon,\sigma}}{d\pi^{\delta,\epsilon}}$
is smooth in both space and time, and $g_{n}$ is $C^{1}$, it follows
that 
\[
\int_{X}g_{n}\left(\frac{d\rho_{t}^{\delta,\epsilon,\sigma}}{d\pi^{\delta,\epsilon}}\right)\partial_{t}\left(\frac{d\rho_{t}^{\delta,\epsilon,\sigma}}{d\pi^{\delta,\epsilon}}\right)d\pi^{\delta,\epsilon}=\frac{1}{2}\iint_{G}\bar{\nabla}g_{n}\left(\frac{d\rho_{t}^{\delta,\epsilon,\sigma}}{d\pi^{\delta,\epsilon}}\right)d\mathbf{v}_{t}^{\delta,\epsilon,\sigma}.
\]
Indeed, we simply take $g_{n}\left(\frac{d\rho_{t}^{\delta,\epsilon,\sigma}}{d\pi^{\delta,\epsilon}}\right)$
as the $C^{1}$ test function $\varphi_{t}(x)$ and use the weak
form of the nonlocal continuity equation after integrating by parts:
\begin{align*}
\int_{X}\varphi_{t}(x)\partial_{t}\left(\frac{d\rho_{t}^{\delta,\epsilon,\sigma}}{d\pi^{\delta,\epsilon}}\right)d\pi^{\delta,\epsilon} & =-\int_{X}\partial_{t}\varphi_{t}(x)\frac{d\rho_{t}^{\delta,\epsilon,\sigma}}{d\pi^{\delta,\epsilon}}d\pi^{\delta,\epsilon}\\
 & =\frac{1}{2}\iint_{G}\bar{\nabla}\varphi_{t}(x,y)d\mathbf{v}_{t}^{\delta,\epsilon,\sigma}.
\end{align*}
Hence, 
\begin{align}
\mathcal{H}^{n}(\rho_{0}^{\delta,\epsilon,\sigma} & \mid\pi^{\delta,\epsilon})-\mathcal{H}^{n}(\rho_{\tau}^{\delta,\epsilon,\sigma}\mid\pi^{\delta,\epsilon})=-\frac{1}{2}\int_{0}^{\tau}\iint_{G}\bar{\nabla}g_{n}\left(\frac{d\rho_{t}^{\delta,\epsilon,\sigma}}{d\pi^{\delta,\epsilon}}\right)d\mathbf{v}_{t}^{\delta,\epsilon,\sigma}dt.\label{eq:regularized-chain-rule}
\end{align}
It therefore remains for us to pass to the limit. We will send $\sigma\rightarrow0$,
then $n\rightarrow\infty$, then $\epsilon\rightarrow0$ and $\delta\rightarrow0$,
and seek to show that for all $\tau\in[0,T]$, 
\[
\lim_{\sigma,n,\epsilon,\delta}\int_{0}^{\tau}\iint_{G}\bar{\nabla}g_{n}\left(\frac{d\rho_{t}^{\delta,\epsilon,\sigma}}{d\pi^{\delta,\epsilon}}\right)d\mathbf{v}_{t}^{\delta,\epsilon,\sigma}dt=\int_{0}^{\tau}\iint_{G}\bar{\nabla}\log\rho_{t}d\mathbf{v}_{t}dt
\]
and for almost all $\tau\in[0,T]$, 
\[
\mathcal{H}(\rho_{0}\mid\pi)-\mathcal{H}(\rho_{\tau}\mid\pi)\leq\liminf_{\sigma,n,\epsilon,\delta}\left[\mathcal{H}^{n}(\rho_{0}^{\delta,\epsilon,\sigma}\mid\pi^{\delta,\epsilon})-\mathcal{H}^{n}(\rho_{\tau}^{\delta,\epsilon,\sigma}\mid\pi^{\delta,\epsilon})\right].
\]
 we first consider the right hand side of \ref{eq:regularized-chain-rule}.

\emph{Step 1: convergence of RHS}. 

\emph{Step 1a}. We first send $\sigma\rightarrow0$. Note that as
a function of $x$, $\Vert\frac{d\rho_{t}^{\delta,\epsilon,\sigma}}{d\pi^{\delta,\epsilon}}\Vert_{C^{1}(X)}\leq C(\delta,\epsilon)$
(with $C(\delta,\epsilon)\rightarrow\infty$ as $\delta\rightarrow0$),
uniformly in $t$ (and $\sigma$). Consequently, $\Vert g_{n}\left(\frac{d\rho_{t}^{\delta,\epsilon,\sigma}}{d\pi^{\delta,\epsilon}}\right)\Vert_{C^{1}(X)}\leq\Vert g_{n}\Vert_{C^{1}(X)}C(\delta,\epsilon)$
where $C(\delta,\epsilon)$ is some constant which is uniform in $t$
but may blow up as $\delta,\epsilon\rightarrow0$. Additionally, it
holds in general for $C^{1}$ functions $\varphi$ that $|\bar{\nabla}\varphi(x,y)|\leq\Vert\varphi\Vert_{C^{1}(X)}(1\wedge|x-y|)$.
Therefore, we can estimate as follows: 
\begin{align*}
\iint_{G}\left|\bar{\nabla}g_{n}\left(\frac{d\rho_{t}^{\delta,\epsilon,\sigma}}{d\pi^{\delta,\epsilon}}\right)\right|d|\mathbf{v}_{t}^{\delta,\epsilon,\sigma}|dt & \leq\Vert g_{n}\Vert_{C^{1}(\mathbb{R}^{d})}C(\delta,\epsilon)\iint_{G}(1\wedge|x-y|)d|h_{\sigma}*k_{\delta}*\mathbf{v}_{t}|\\
 & \leq\Vert g_{n}\Vert_{C^{1}(\mathbb{R}^{d})}C(\delta,\epsilon)\cdot C_{\eta}\sqrt{\mathcal{A}(h_{\sigma}*k_{\delta}*\rho_{t},h_{\sigma}*k_{\delta}*\mathbf{v}_{t};k_{\delta}*\pi)}
\end{align*}
where in the second line we have applied Lemma \ref{lem:C-eta=000020upper=000020bound=000020of=000020flux}. 

Now, let $\sigma_{n}$ be a sequence converging to zero.  From the
local weak{*} convergence of $h_{\sigma_{n}}*k_{\delta}*\mathbf{v}_{t}$
to $k_{\delta}*\mathbf{v}_{t}$, and the fact that $\sup_{\sigma_{n}}\mathcal{A}(h_{\sigma}*k_{\delta}*\rho_{t},h_{\sigma}*k_{\delta}*\mathbf{v}_{t};k_{\delta}*\pi)<\infty$,
we can apply Lemma \ref{lem:total-flux-convergence}, and deduce that\footnote{To apply this lemma, we need to argue that $g_{n}\left(\frac{d((1-\epsilon)h_{\sigma}*k_{\delta}*\rho_{t}+\epsilon\mu_{0})}{d((1-\epsilon)k_{\delta}*\pi+\epsilon\text{Leb})}\right)$
and $g_{n}\left(\frac{d((1-\epsilon)k_{\delta}*\rho_{t}+\epsilon\mu_{0})}{d((1-\epsilon)k_{\delta}*\pi+\epsilon\text{Leb})}\right)$
are indeed $C^{1}$ functions and we have uniform convergence on compact
sets as $\sigma\rightarrow0$. We estimate as follows. First note
that  $h_{\sigma_{n}}*k_{\delta}*\rho_{t}(x)$ converges uniformly
on compacts to $k_{\delta}*\rho_{t}(x)$, by Lemma \ref{lem:time-mollified-measures-convergence}.
So let's see that we still have the same uniform convergence on compacts
for $g_{n}\left(\frac{d((1-\epsilon)h_{\sigma}*k_{\delta}*\rho_{t}+\epsilon\mu_{0})}{d((1-\epsilon)k_{\delta}*\pi+\epsilon\text{Leb})}\right)$.
Firstly, $\epsilon\frac{d\mu_{0}}{d((1-\epsilon)k_{\delta}*\pi+\epsilon\text{Leb})}$
is independent of $\sigma$, and
\[
\frac{d(h_{\sigma}*k_{\delta}*\rho_{t})}{d(k_{\delta}*\pi+\epsilon\text{Leb})}=\frac{d(h_{\sigma}*k_{\delta}*\rho_{t})}{dx}\frac{dx}{d(k_{\delta}*\pi+\epsilon\text{Leb})}.
\]
So since $(\frac{d(k_{\delta}*\pi+\epsilon\text{Leb})}{dx})^{-1}$
is $C^{1}$ we are fine. Then, $g_{n}$ is $C^{1}$, hence uniformly
continuous. Therefore the uniform convergence of $\frac{d((1-\epsilon)h_{\sigma}*k_{\delta}*\rho_{t}+\epsilon\mu_{0})}{d((1-\epsilon)k_{\delta}*\pi+\epsilon\text{Leb})}(x)$
on compacts to $\frac{d((1-\epsilon)k_{\delta}*\rho_{t}+\epsilon\mu_{0})}{d((1-\epsilon)k_{\delta}*\pi+\epsilon\text{Leb})}$
means that $g_{n}\left(\frac{d((1-\epsilon)h_{\sigma}*k_{\delta}*\rho_{t}+\epsilon\mu_{0})}{d((1-\epsilon)k_{\delta}*\pi+\epsilon\text{Leb})}\right)$
converges uniformly on compacts as well.} 
\[
\iint_{G}\bar{\nabla}g_{n}\left(\frac{d\rho_{t}^{\delta,\epsilon,\sigma}}{d\pi^{\delta,\epsilon}}\right)d|\mathbf{v}_{t}^{\delta,\epsilon,\sigma}|\stackrel{\sigma\rightarrow0}{\longrightarrow}\iint_{G}\bar{\nabla}g_{n}\left(\frac{d\rho_{t}^{\delta,\epsilon}}{d\pi^{\delta,\epsilon}}\right)d|\mathbf{v}_{t}^{\delta,\epsilon}|\text{\ensuremath{\qquad t\text{-a.s..}}}
\]
\begin{comment}
{[}need to explicitly argue that $g_{n}\left(\frac{d(h_{\sigma}*k_{\delta}*\rho_{t})}{d\pi}\right)\rightarrow g_{n}\left(\frac{d(k_{\delta}*\rho_{t})}{d\pi}\right)$
pointwise with bounded $C^{1}$ norm{]}
\end{comment}
{} The claim now follows from an application of the generalized Lebesgue
dominated convergence theorem.  Indeed,we have that pointwise $t$-a.s.
\begin{align*}
\iint_{G}\left|\bar{\nabla}g_{n}\left(\frac{d\rho_{t}^{\delta,\epsilon,\sigma}}{d\pi^{\delta,\epsilon}}\right)\right|d(|\mathbf{v}_{t}^{\delta,\epsilon,\sigma}|) & \leq C\sqrt{\mathcal{A}(h_{\sigma}*k_{\delta}*\rho_{t},h_{\sigma}*k_{\delta}*\mathbf{v}_{t};k_{\delta}*\pi)}\\
 & \leq C\left(\mathcal{A}(h_{\sigma}*k_{\delta}*\rho_{t},h_{\sigma}*k_{\delta}*\mathbf{v}_{t};k_{\delta}*\pi)+1\right).
\end{align*}
Now, observe that by Proposition \ref{prop:time-convolution-action}
we have that 
\[
\mathcal{A}(h_{\sigma}*k_{\delta}*\rho_{t},h_{\sigma}*k_{\delta}*\mathbf{v}_{t};k_{\delta}*\pi)dt\rightarrow\mathcal{A}(k_{\delta}*\rho_{t},k_{\delta}*\mathbf{v}_{t};k_{\delta}*\pi)\text{ }t-a.s.
\]
and
\[
\int_{0}^{T}\mathcal{A}(h_{\sigma}*k_{\delta}*\rho_{t},h_{\sigma}*k_{\delta}*\mathbf{v}_{t};k_{\delta}*\pi)dt\rightarrow\int_{0}^{T}\mathcal{A}(k_{\delta}*\rho_{t},k_{\delta}*\mathbf{v}_{t};k_{\delta}*\pi)dt
\]
as $\sigma\rightarrow0$, since 
\[
\int_{0}^{T}\mathcal{A}(k_{\delta}*\rho_{t},k_{\delta}*\mathbf{v}_{t};k_{\delta}*\pi)dt\leq S\int_{0}^{T}\mathcal{A}(\rho_{t},\mathbf{v}_{t})dt<\infty
\]
 by Lemma \ref{lem:action-spatial-convolution}. Therefore, from the
generalized Lebesgue dominated convergence theorem, the fact that
\[
1+\int_{0}^{T}\mathcal{A}(h_{\sigma}*k_{\delta}*\rho_{t},h_{\sigma}*k_{\delta}*\mathbf{v}_{t};k_{\delta}*\pi)dt\stackrel{\sigma\rightarrow0}{\longrightarrow}1+\int_{0}^{T}\mathcal{A}(k_{\delta}*\rho_{t},k_{\delta}*\mathbf{v}_{t};k_{\delta}*\pi)dt
\]
then implies that 
\[
\int_{0}^{T}\iint_{G}\left|\bar{\nabla}g_{n}\left(\frac{d\rho_{t}^{\delta,\epsilon,\sigma}}{d\pi^{\delta,\epsilon}}\right)\right|d(|\mathbf{v}_{t}^{\delta,\epsilon,\sigma}|)dt\stackrel{\sigma\rightarrow0}{\longrightarrow}\int_{0}^{T}\iint_{G}\left|\bar{\nabla}g_{n}\left(\frac{d\rho_{t}^{\delta,\epsilon}}{d\pi^{\delta,\epsilon}}\right)\right|d|\mathbf{v}_{t}^{\delta,\epsilon}|dt
\]
as desired.

\begin{comment}
{[}continue{]} {[}should be similar to argument from proof of compactness
theorem in Erbar{]}
\end{comment}
{} %
\begin{comment}
However, (as already observed in the proof of Proposition \ref{prop:compactness-nce})
this convergence follows directly from the fact that as $\sigma\rightarrow0$,
$h_{\sigma_{n}}*k_{\delta}*\mathbf{v}_{t}$ converges locally weak{*}ly
to $k_{\delta}*\mathbf{v}_{t}$ in $\mathcal{M}_{loc}(G)$.
\end{comment}

\emph{Step 1b}. Next, we send $n\rightarrow\infty$. For this estimate,
we argue similarly to the proof of \cite[Claim 5.12]{erbar2014gradient}.
In what follows, for $0\leq\delta<\delta_{0}$, we write $v_{t}^{\delta,\epsilon}$
to denote the density of $\mathbf{v}_{t}^{\delta,\epsilon}$ with
respect to $d\pi^{\delta,\epsilon}(x)d\pi^{\delta,\epsilon}(y)$ restricted
to $G$; such a density is guaranteed to exist $t$-a.s. by \cite[Lem. 2.3]{erbar2014gradient}
(in the case $X=\mathbb{R}^{d}$; if $X=\mathbb{T}^{d}$ the equivalent
lemma is \cite[Lem. 3.5]{ferreira2019minimizing}) since $\mathcal{A}(\rho_{t}^{\delta,\epsilon},\mathbf{v}_{t}^{\delta,\epsilon};\pi^{\delta,\epsilon})<\infty$
for almost all $t$, and $\rho_{t}\ll\pi$ for almost all $t$ since
$\int_{0}^{T}\mathcal{H}(\rho_{t}\mid\pi)<\infty$. Now, since $g_{n}(x)$
converges pointwise to $\log x$, in order to apply the dominated
converge theorem, we want to show that uniformly in $n$, $\left(\bar{\nabla}\log\left(\frac{d\rho_{t}^{\delta,\epsilon}}{d\pi^{\delta,\epsilon}}\right)-\bar{\nabla}g_{n}\left(\frac{d\rho_{t}^{\delta,\epsilon}}{d\pi^{\delta,\epsilon}}\right)\right)v_{t}^{\delta}$
is dominated in $L^{1}$ in spacetime. 

We first estimate 
\begin{multline*}
\int_{0}^{T}\iint_{G}\left|\left(\bar{\nabla}\log\left(\frac{d\rho_{t}^{\delta,\epsilon}}{d\pi^{\delta,\epsilon}}\right)-\bar{\nabla}g_{n}\left(\frac{d\rho_{t}^{\delta,\epsilon}}{d\pi^{\delta,\epsilon}}\right)\right)v_{t}^{\delta,\epsilon}\right|d(\pi^{\delta,\epsilon}\otimes\pi^{\delta,\epsilon})dt\\
\begin{aligned} & \leq\int_{0}^{T}\iint_{G}\left|\left(\bar{\nabla}\log\left(\frac{d\rho_{t}^{\delta,\epsilon}}{d\pi^{\delta,\epsilon}}\right)-\bar{\nabla}g_{n}\left(\frac{d\rho_{t}^{\delta,\epsilon}}{d\pi^{\delta,\epsilon}}\right)\right)\right||v_{t}^{\delta,\epsilon}|d(\pi^{\delta,\epsilon}\otimes\pi^{\delta,\epsilon})dt\end{aligned}
.
\end{multline*}
Note that since $|g_{n}|\leq|\log|$, it holds that pointwise, 
\[
\left|\left(\bar{\nabla}\log\left(\frac{d\rho_{t}^{\delta,\epsilon}}{d\pi^{\delta,\epsilon}}\right)-\bar{\nabla}g_{n}\left(\frac{d\rho_{t}^{\delta,\epsilon}}{d\pi^{\delta,\epsilon}}\right)\right)\right|\leq2\left|\bar{\nabla}\left(\log\right)\left(\frac{d\rho_{t}^{\delta,\epsilon}}{d\pi^{\delta,\epsilon}}\right)\right|.
\]
By Lemma \ref{lem:Young's=000020inequality=000020action=000020fisher=000020integrands},
we see at the level of integrands that 
\begin{multline}
\left|\bar{\nabla}\left(\log\right)\left(\frac{d\rho_{t}^{\delta,\epsilon}}{d\pi^{\delta,\epsilon}}\right)(x,y)\right||v_{t}^{\delta,\epsilon}|(x,y)\leq\frac{\left(\frac{d\mathbf{v}_{t}^{\delta,\epsilon}}{d(\pi^{\delta,\epsilon}\otimes\pi^{\delta,\epsilon})}(x,y)\right)^{2}}{2\theta\left(\frac{d(\rho_{t}^{\delta,\epsilon}\otimes\pi^{\delta,\epsilon})}{d(\pi^{\delta,\epsilon}\otimes\pi^{\delta,\epsilon})}(x,y),\frac{d(\pi^{\delta,\epsilon}\otimes\rho_{t}^{\delta,\epsilon})}{d(\pi^{\delta,\epsilon}\otimes\pi^{\delta,\epsilon})}(x,y)\right)\eta(x,y)}\label{eq:fisher-delta-cauchy-schwarz}\\
+\frac{1}{2}\left(\bar{\nabla}\log\frac{d\rho_{t}^{\delta,\epsilon}}{d\pi^{\delta,\epsilon}}(x,y)\right)^{2}\theta\left(\frac{d(\rho_{t}^{\delta,\epsilon}\otimes\pi^{\delta,\epsilon})}{d(\pi^{\delta,\epsilon}\otimes\pi^{\delta,\epsilon})}(x,y),\frac{d(\pi^{\delta,\epsilon}\otimes\rho_{t}^{\delta,\epsilon})}{d(\pi^{\delta,\epsilon}\otimes\pi^{\delta,\epsilon})}(x,y)\right)\eta(x,y).
\end{multline}
Integrating both sides with respect to $d(\pi^{\delta,\epsilon}\otimes\pi^{\delta,\epsilon})$
(and using the fact that $\theta$ is the logarithmic mean) we get
\[
\int_{0}^{T}\iint_{G}\left|\bar{\nabla}\left(\log\right)\left(\frac{d\rho_{t}^{\delta,\epsilon}}{d\pi^{\delta,\epsilon}}\right)\right||v_{t}^{\delta,\epsilon}|d(\pi^{\delta,\epsilon}\otimes\pi^{\delta,\epsilon})dt\leq\int_{0}^{T}\left[\mathcal{A}(\rho_{t}^{\delta,\epsilon},\mathbf{v}_{t}^{\delta,\epsilon};\pi^{\delta,\epsilon})+\mathcal{I}(\rho_{t}^{\delta,\epsilon}\mid\pi^{\delta,\epsilon})\right]dt
\]
In turn, by Lemma \ref{lem:Cauchy-Schwarz} and the convexity of
$\mathcal{I}$ in its first argument, and the fact that $(1-\epsilon)\mathcal{A}(k_{\delta}*\rho_{t},k_{\delta}*\mathbf{v}_{t};k_{\delta}*\pi)\leq S\mathcal{A}(\rho_{t},\mathbf{v}_{t};\pi)$
and $\mathcal{I}(k_{\delta}*\rho_{t}\mid k_{\delta}*\pi)\leq S\mathcal{I}(\rho_{t}\mid\pi)$,
we see that
\begin{multline*}
\int_{0}^{T}\iint_{G}\left|\bar{\nabla}\left(\log\right)\left(\frac{d\rho_{t}^{\delta,\epsilon}}{d\pi^{\delta,\epsilon}}\right)\right||v_{t}^{\delta,\epsilon}|d(\pi^{\delta,\epsilon}\otimes\pi^{\delta,\epsilon})dt\\
\begin{aligned} & \leq S\int_{0}^{T}\mathcal{A}(\rho_{t},\mathbf{v}_{t};\pi)dt+S\int_{0}^{T}\left(\mathcal{I}(\rho_{t}\mid\pi)+\epsilon\mathcal{I}(\mu_{0}\mid\text{Leb})\right)dt<\infty.\end{aligned}
\end{multline*}
Altogether this shows that $\left(\bar{\nabla}\log\left(\frac{d\rho_{t}^{\delta,\epsilon}}{d\pi^{\delta,\epsilon}}\right)-\bar{\nabla}g_{n}\left(\frac{d\rho_{t}^{\delta,\epsilon}}{d\pi^{\delta,\epsilon}}\right)\right)v_{t}^{\delta,\epsilon}$
is dominated by the $L^{1}(G\times[0,T],(\pi^{\delta,\epsilon}\otimes\pi^{\delta,\epsilon})\times\text{Leb})$
function $2\left|\bar{\nabla}\left(\log\right)\left(\frac{d\rho_{t}^{\delta,\epsilon}}{d\pi^{\delta,\epsilon}}\right)\right||v_{t}^{\delta,\epsilon}|$
as required, and so as $n\rightarrow\infty$, 
\[
\int_{0}^{T}\iint_{G}\left|\bar{\nabla}g_{n}\left(\frac{d\rho_{t}^{\delta,\epsilon}}{d\pi^{\delta,\epsilon}}\right)v_{t}^{\delta,\epsilon}\right|d(\pi^{\delta,\epsilon}\otimes\pi^{\delta,\epsilon})dt\rightarrow\int_{0}^{T}\iint_{G}\left|\bar{\nabla}\log\left(\frac{d\rho_{t}^{\delta,\epsilon}}{d\pi^{\delta,\epsilon}}\right)v_{t}^{\delta,\epsilon}\right|d(\pi^{\delta,\epsilon}\otimes\pi^{\delta,\epsilon})dt.
\]

\emph{Step 1c}. Lastly, we send $\epsilon\rightarrow0$ and $\delta\rightarrow0$.
We want to show that for almost all $t\in[0,T]$, we have the convergence
(denoting $v_{t}:=\frac{d\mathbf{v}_{t}}{d(\pi\otimes\pi)}$)
\[
\int_{0}^{T}\iint_{G}\left|\bar{\nabla}\log\left(\frac{d\rho_{t}^{\delta,\epsilon}}{d\pi^{\delta,\epsilon}}\right)v_{t}^{\delta,\epsilon}\right|d(\pi^{\delta,\epsilon}\otimes\pi^{\delta,\epsilon})dt\stackrel{\epsilon,\delta\rightarrow0}{\longrightarrow}\int_{0}^{T}\iint_{G}\left|\bar{\nabla}\log\left(\frac{d\rho_{t}}{d\pi}\right)v_{t}\right|d(\pi\otimes\pi)dt.
\]
Equation \ref{eq:fisher-delta-cauchy-schwarz} above already shows
that $\left|\bar{\nabla}\log\left(\frac{d\rho_{t}^{\delta,\epsilon}}{d\pi^{\delta,\epsilon}}\right)v_{t}^{\delta,\epsilon}\right|$
is dominated in, $L^{1}(G\times[0,T];(\pi^{\delta,\epsilon}\otimes\pi^{\delta,\epsilon})\times\text{Leb})$.
Multiplying both sides of Equation \ref{eq:fisher-delta-cauchy-schwarz}
by $\frac{d(\pi^{\delta,\epsilon}\otimes\pi^{\delta,\epsilon})}{dxdy}$
and integrating both sides w.r.t. $dxdydt$ instead achieves the same
result. Therefore, to apply the generalized Lebesgue dominated convergence
theorem, it suffices to show that: 
\begin{enumerate}
\item $\left|\bar{\nabla}\log\left(\frac{d\rho_{t}^{\delta,\epsilon}}{d\pi^{\delta,\epsilon}}\right)v_{t}^{\delta,\epsilon}\right|\frac{d(\pi^{\delta,\epsilon}\otimes\pi^{\delta,\epsilon})}{dxdy}$
converges pointwise almost everywhere (w.r.t. $dxdy\llcorner G$ in
space and $dt$ in time) to $\left|\bar{\nabla}\log\left(\frac{d\rho_{t}}{d\pi}\right)v_{t}\right|\frac{d(\pi\otimes\pi)}{dxdy}$
;
\item For the same notion of pointwise convergence, 
\begin{multline*}
\frac{1}{2}\left(\bar{\nabla}\log\frac{d\rho_{t}^{\delta,\epsilon}}{d\pi^{\delta,\epsilon}}(x,y)\right)^{2}\theta\left(\frac{d\rho_{t}^{\delta,\epsilon}}{d\pi^{\delta,\epsilon}}(x),\frac{d\rho_{t}^{\delta,\epsilon}}{d\pi^{\delta,\epsilon}}(y)\right)\eta(x,y)\frac{d(\pi^{\delta,\epsilon}\otimes\pi^{\delta,\epsilon})}{dxdy}\rightarrow\\
\frac{1}{2}\left(\bar{\nabla}\log\frac{d\rho_{t}}{d\pi}(x,y)\right)^{2}\theta\left(\frac{d\rho_{t}}{d\pi}(x),\frac{d\rho_{t}}{d\pi}(y)\right)\eta(x,y)\frac{d(\pi\otimes\pi)}{dxdy};
\end{multline*}
\item And likewise,
\[
\frac{\left(\frac{d\mathbf{v}_{t}^{\delta,\epsilon}}{d(\pi^{\delta,\epsilon}\otimes\pi^{\delta,\epsilon})}(x,y)\right)^{2}}{2\theta\left(\frac{d\rho_{t}^{\delta,\epsilon}}{d\pi^{\delta,\epsilon}}(x),\frac{d\rho_{t}^{\delta,\epsilon}}{d\pi^{\delta,\epsilon}}(y)\right)\eta(x,y)}\frac{d(\pi^{\delta,\epsilon}\otimes\pi^{\delta,\epsilon})}{dxdy}\rightarrow\frac{\left(\frac{d\mathbf{v}_{t}}{d(\pi\otimes\pi)}(x,y)\right)^{2}}{2\theta\left(\frac{d\rho_{t}}{d\pi}(x),\frac{d\rho_{t}}{d\pi}(y)\right)\eta(x,y)}\frac{d(\pi\otimes\pi)}{dxdy};
\]
\item And finally,
\[
\int_{0}^{T}\left[\mathcal{I}(\rho_{t}^{\delta,\epsilon}\mid\pi^{\delta,\epsilon})+\mathcal{A}(\rho_{t}^{\delta,\epsilon},\mathbf{v}_{t}^{\delta,\epsilon};\pi^{\delta,\epsilon})\right]dt\stackrel{\epsilon,\delta\rightarrow0}{\longrightarrow}\int_{0}^{T}\left[\mathcal{I}(\rho_{t}\mid\pi)+\mathcal{A}(\rho_{t},\mathbf{v}_{t};\pi)\right]dt.
\]
\end{enumerate}
To establish the pointwise convergences, for (2), it suffices to
observe that $\frac{d\pi^{\delta,\epsilon}}{dx}=(1-\epsilon)\frac{d(k_{\delta}*\pi)}{dx}+\epsilon$
and therefore pointwise $dx$-a.e. convergence as $\delta,\epsilon\rightarrow0$
follows from \cite[Thm. C.7(ii)]{evans2010partial} and the fact that
$\pi\ll dx$. Similarly, $\rho_{t}\ll\pi$ for almost all $t$ since
we assumed $\int_{0}^{T}\mathcal{H}(\rho_{t}\mid\pi)<\infty$, and
so by the same token, for almost all $t\in[0,T]$ we have that $\frac{d\rho_{t}^{\delta,\epsilon}}{dx}=(1-\epsilon)\frac{d(k_{\delta}*\rho_{t})}{dx}+\epsilon\frac{d\mu_{0}}{dx}$
converges pointwise $dx$-a.e to $\frac{d\rho_{t}}{dx}$. It now follows
from the Radon-Nikodym theorem that $\frac{d\rho_{t}^{\delta,\epsilon}}{d\pi^{\delta,\epsilon}}$
converges to $\frac{d\rho_{t}}{d\pi}$ for almost all $x$ and $t$.
Finally the integrand in (2) is pointwise continuous with respect
to each of its Radon-Nikodym derivative arguments once we know that
$\mathcal{I}(\rho_{t}^{\delta,\epsilon}\mid\pi^{\delta,\epsilon})$
and $\mathcal{I}(\rho_{t}\mid\pi)$ are finite. 

To establish the pointwise almost everywhere convergence for (1) and
(3), we reason identically, except we also need that $v_{t}^{\delta,\epsilon}$
converges pointwise a.e. to $v_{t}$. To see this, observe that 
\[
v_{t}^{\delta,\epsilon}\frac{d(\pi^{\delta,\epsilon}\otimes\pi^{\delta,\epsilon})}{dxdy}=\frac{d\mathbf{v}_{t}^{\delta,\epsilon}}{d(\pi^{\delta,\epsilon}\otimes\pi^{\delta,\epsilon})}\frac{d(\pi^{\delta,\epsilon}\otimes\pi^{\delta,\epsilon})}{dxdy}=\frac{d\mathbf{v}_{t}^{\delta,\epsilon}}{dxdy}=(1-\epsilon)\frac{d(k_{\delta}*\mathbf{v}_{t})}{dxdy}
\]
so in light of the pointwise a.e. convergence of $\frac{d(\pi^{\delta,\epsilon}\otimes\pi^{\delta,\epsilon})}{dxdy}$
to $\frac{d(\pi\otimes\pi)}{dxdy}$, it suffices that $\frac{d(k_{\delta}*\mathbf{v}_{t})}{dxdy}$
converges pointwise a.e. to $\frac{d\mathbf{v}_{t}}{dxdy}$ for almost
all $t\in[0,T]$. However this is established by Lemma \ref{lem:flux-density-pointwise-convergence}
as soon as we know that $\mathbf{v}_{t}\ll dxdy\llcorner G$; and
this in turn holds for almost all $t$ since $\rho_{t}\ll\pi\ll dx$,
and $\int_{0}^{T}\mathcal{A}(\rho_{t},\mathbf{v}_{t};\pi)dt<\infty$,
so \cite[Lem. 2.3]{erbar2014gradient} (resp. \cite[Lem. 3.5]{ferreira2019minimizing})
applies.

Finally, to establish (4), we first observe that $\rho_{t}^{\delta,\epsilon}\rightharpoonup\rho_{t}$,
$\pi^{\delta,\epsilon}\rightharpoonup\pi$, and $d\mathbf{v}_{t}^{\delta,\epsilon}dt\rightharpoonup^{*}d\mathbf{v}_{t}dt$
in $\mathcal{M}_{loc}([0,T]\times G)$, and so by Propositions \ref{prop:action=000020convex=000020lsc}
and \ref{prop:fisher-convex-lsc} and Fatou's lemma we have 
\[
\int_{0}^{T}\left[\mathcal{I}(\rho_{t}\mid\pi)+\mathcal{A}(\rho_{t},\mathbf{v}_{t};\pi)\right]dt\leq\liminf_{\delta,\epsilon\rightarrow0}\int_{0}^{T}\left[\mathcal{I}(\rho_{t}^{\delta,\epsilon}\mid\pi^{\delta,\epsilon})+\mathcal{A}(\rho_{t}^{\delta,\epsilon},\mathbf{v}_{t}^{\delta,\epsilon};\pi^{\delta,\epsilon})\right]dt.
\]
On the other hand, we have 
\[
\int_{0}^{T}\left[\mathcal{I}(\rho_{t}^{\delta,\epsilon}\mid\pi^{\delta,\epsilon})+\mathcal{A}(\rho_{t}^{\delta,\epsilon},\mathbf{v}_{t}^{\delta,\epsilon};\pi^{\delta,\epsilon})\right]dt\leq\int_{0}^{T}\left[S\mathcal{I}(\rho_{t}\mid\pi)+\epsilon\mathcal{I}(\mu_{0}\mid\text{Leb})+S\mathcal{A}(\rho_{t},\mathbf{v}_{t};\pi)\right]dt
\]
and so, since $S\rightarrow1$ as $\delta\rightarrow0$, we have 
\[
\limsup_{\delta,\epsilon\rightarrow0}\int_{0}^{T}\left[\mathcal{I}(\rho_{t}^{\delta,\epsilon}\mid\pi^{\delta,\epsilon})+\mathcal{A}(\rho_{t}^{\delta,\epsilon},\mathbf{v}_{t}^{\delta,\epsilon};\pi^{\delta,\epsilon})\right]dt\leq\int_{0}^{T}\left[\mathcal{I}(\rho_{t}\mid\pi)+\mathcal{A}(\rho_{t},\mathbf{v}_{t};\pi)\right]dt
\]
so (4) holds.

\emph{Step 2: convergence of LHS}. We first send $\sigma\rightarrow0$,
then $n\rightarrow\infty$, then $\epsilon,\delta\rightarrow0$, just
as for the RHS.   Before proceeding, we note that $\mathcal{H}^{n}(\cdot\mid\cdot)$
is jointly convex on $\mathcal{M}_{loc}^{+}$, since $f^{n}$ is convex;
hence 
\[
\mathcal{H}^{n}(\rho_{t}^{\delta,\epsilon,\sigma}\mid\pi^{\delta,\epsilon})\leq(1-\epsilon)\mathcal{H}^{n}(h_{\sigma}*k_{\delta}*\rho_{t}\mid k_{\delta}*\pi)+\epsilon\mathcal{H}^{n}(\mu_{0}\mid\text{Leb}).
\]
At the same time, we have that $f^{n}(x)\leq x\log x-x+1$ , hence
$\mathcal{H}^{n}(\mu_{0}\mid\text{Leb})\leq\mathcal{H}(\mu_{0}\mid\text{Leb})<\infty$
by assumption.

\emph{Step 2a}. We first send $\sigma\rightarrow0$.  By the construction
in step zero, we have that $f_{n}(x)$ converges monotonically upwards
to $x\log x-x+1$. We use this and the convexity of $x\log x-x+1$
to compute that
\begin{align*}
\mathcal{H}^{n}(\rho_{t}^{\delta,\epsilon,\sigma}\mid\pi^{\delta,\epsilon}) & =\int_{X}f_{n}\left(\frac{d\rho_{t}^{\delta,\epsilon,\sigma}}{d\pi^{\delta,\epsilon}}\right)d\pi^{\delta,\epsilon}\\
 & \leq\int_{X}\left(\frac{d\rho_{t}^{\delta,\epsilon,\sigma}}{d\pi^{\delta,\epsilon}}\log\left(\frac{d\rho_{t}^{\delta,\epsilon,\sigma}}{d\pi^{\delta,\epsilon}}\right)-\frac{d\rho_{t}^{\delta,\epsilon,\sigma}}{d\pi^{\delta,\epsilon}}+1\right)d\pi^{\delta,\epsilon}\\
\text{(Jensen)} & \leq\int_{X}\int_{-\sigma}^{\sigma}\left(\frac{d\rho_{t-s}^{\delta,\epsilon}}{d\pi^{\delta,\epsilon}}\log\left(\frac{d\rho_{t-s}^{\delta,\epsilon}}{d\pi^{\delta,\epsilon}}\right)-\frac{d\rho_{t-s}^{\delta,\epsilon}}{d\pi^{\delta,\epsilon}}+1\right)h_{\sigma}(s)dsd\pi^{\delta,\epsilon}\\
 & =h_{\sigma}*\mathcal{H}(\rho_{t}^{\delta,\epsilon}\mid\pi^{\delta,\epsilon}).
\end{align*}
At the same time,  since pointwise in $t$, 
\begin{align*}
\mathcal{H}(\rho_{t}^{\delta,\epsilon}\mid\pi^{\delta,\epsilon}) & \leq(1-\epsilon)\mathcal{H}(k_{\delta}*\rho_{t}\mid k_{\delta}*\pi)+\epsilon\mathcal{H}(\mu_{0}\mid\text{Leb})\\
 & \leq(1-\epsilon)\mathcal{H}(\rho_{t}\mid\pi)+\epsilon\mathcal{H}(\mu_{0}\mid\text{Leb})
\end{align*}
we have that $t\mapsto\mathcal{H}(\rho_{t}^{\delta,\epsilon}\mid\pi^{\delta,\epsilon})$
is in $L^{1}([0,T])$, and therefore $h_{\sigma}*\mathcal{H}(\rho_{t}^{\delta,\epsilon}\mid\pi^{\delta,\epsilon})$
converges pointwise a.s. to $\mathcal{H}(\rho_{t}^{\delta,\epsilon}\mid\pi^{\delta,\epsilon})$
as $\sigma\rightarrow0$. In other words, for almost all $t$, 
\[
\int_{X}\int_{-\sigma}^{\sigma}\frac{d\rho_{t-s}^{\delta,\epsilon}}{d\pi^{\delta,\epsilon}}\log\left(\frac{d\rho_{t-s}^{\delta,\epsilon}}{d\pi^{\delta,\epsilon}}\right)h_{\sigma}(s)dsd\pi^{\delta,\epsilon}\rightarrow\int_{X}\frac{d\rho_{t}^{\delta,\epsilon}}{d\pi^{\delta,\epsilon}}\log\left(\frac{d\rho_{t}^{\delta,\epsilon}}{d\pi^{\delta,\epsilon}}\right)d\pi^{\delta,\epsilon}.
\]
At the same time, by convexity of $x\log x-x+1$ we have the pointwise
bound
\[
f_{n}\left(\frac{d\rho_{t}^{\delta,\epsilon,\sigma}}{d\pi^{\delta,\epsilon}}\right)\leq\int_{-\sigma}^{\sigma}\left(\frac{d\rho_{t-s}^{\delta,\epsilon}}{d\pi^{\delta,\epsilon}}\log\left(\frac{d\rho_{t-s}^{\delta,\epsilon}}{d\pi^{\delta,\epsilon}}\right)-\frac{d\rho_{t-s}^{\delta,\epsilon}}{d\pi^{\delta,\epsilon}}+1\right)h_{\sigma}(s)ds.
\]
Note also that $\frac{d\rho_{t}^{\delta,\epsilon,\sigma}}{d\pi^{\delta,\epsilon}}$
converges pointwise $\text{Leb}^{d}$-almost everywhere to $\frac{d\rho_{t}^{\delta,\epsilon}}{d\pi^{\delta,\epsilon}}$.
So by the generalized dominated convergence theorem we have that
$t$-a.s.,  
\[
\int_{X}f_{n}\left(\frac{d\rho_{t}^{\delta,\epsilon,\sigma}}{d\pi^{\delta,\epsilon}}\right)d\pi^{\delta,\epsilon}\rightarrow\int_{X}f_{n}\left(\frac{d\rho_{t}^{\delta,\epsilon}}{d\pi^{\delta,\epsilon}}\right)d\pi^{\delta,\epsilon}.
\]
Thus, for almost all $\tau_{0}\leq\tau_{1}\in[0,T]$, it holds that
\[
\mathcal{H}^{n}(\rho_{\tau_{0}}^{\delta,\epsilon}\mid\pi^{\delta,\epsilon})-\mathcal{H}^{n}(\rho_{\tau_{1}}^{\delta,\epsilon}\mid\pi^{\delta,\epsilon})=\lim_{\sigma\rightarrow0}\left(\mathcal{H}^{n}(\rho_{\tau_{0}}^{\delta,\epsilon,\sigma}\mid k_{\delta}*\pi)-\mathcal{H}^{n}(\rho_{\tau_{1}}^{\delta,\epsilon,\sigma}\mid\pi^{\delta,\epsilon})\right).
\]

\emph{Step 2b}. Next, observe that by construction, as $n\rightarrow0$,
$\mathcal{H}^{n}(\rho_{t}^{\delta,\epsilon}\mid\pi^{\delta,\epsilon})\nearrow\mathcal{H}(\rho_{t}^{\delta,\epsilon}\mid\pi^{\delta,\epsilon})$
for all $t\in[0,T]$.

\emph{Step 2c}. Finally, we send $\epsilon,\delta\rightarrow0$. On
the one hand,  from Lemma \ref{lem:convolution-regularity-entropy}
we have that as $\delta\rightarrow0$, we have$\mathcal{H}(k_{\delta}*\rho_{t}\mid k_{\delta}*\pi)\rightarrow\mathcal{H}(\rho_{t}\mid\pi).$
 At the same time, from the joint convexity of $\mathcal{H}$ we
have 
\[
\mathcal{H}((1-\epsilon)k_{\delta}*\rho_{t}+\epsilon\mu_{0}\mid(1-\epsilon)k_{\delta}*\pi+\epsilon\text{Leb})\leq(1-\epsilon)\mathcal{H}(k_{\delta}*\rho_{t}\mid k_{\delta}*\pi)+\epsilon\mathcal{H}(\mu_{0}\mid\text{Leb})
\]
so that 
\[
\limsup_{\epsilon,\delta\rightarrow0}\mathcal{H}((1-\epsilon)k_{\delta}*\rho_{t}+\epsilon\mu_{0}\mid(1-\epsilon)k_{\delta}*\pi+\epsilon\text{Leb})\leq\mathcal{H}(\rho_{t}\mid\pi)
\]
and clearly $(1-\epsilon)k_{\delta}*\rho+\epsilon\mu_{0}\rightharpoonup^{*}\rho$
and $(1-\epsilon)k_{\delta}*\pi+\epsilon\text{Leb}\rightharpoonup^{*}\pi$
as $\epsilon,\delta\rightarrow0$, so that 
\[
\liminf_{\epsilon,\delta\rightarrow0}\mathcal{H}((1-\epsilon)k_{\delta}*\rho_{t}+\epsilon\mu_{0}\mid(1-\epsilon)k_{\delta}*\pi+\epsilon\text{Leb})\geq\mathcal{H}(\rho_{t}\mid\pi).
\]
Hence 
\[
\lim_{\epsilon,\delta\rightarrow0}\mathcal{H}((1-\epsilon)k_{\delta}*\rho_{t}+\epsilon\mu_{0}\mid(1-\epsilon)k_{\delta}*\pi+\epsilon\text{Leb})=\mathcal{H}(\rho_{t}\mid\pi)
\]
as desired.

Altogether, this shows that in the limit, for almost all $\tau_{0},\tau_{1}\in[0,T]$,
the LHS (with endpoints $\tau_{0}$ and $\tau_{1}$) converges to
$\mathcal{H}(\rho_{\tau_{0}}\mid\pi)-\mathcal{H}(\rho_{\tau_{1}}\mid\pi)$.
 Lastly, let $(\tau_{n})_{n\in\mathbb{N}}$ be a sequence of points
converging to $0$, each of which belongs to the set (of full measure)
for which we have deduced the convergence of $\mathcal{H}^{n}(\rho_{t}^{\delta,\epsilon,\sigma}\mid\pi^{\delta,\epsilon})$
to $\mathcal{H}(\rho_{\tau}\mid\pi)$. Since $(\rho_{t})_{t\in[0,T]}$
is weak{*} continuous in $t$, it follows that 
\[
\mathcal{H}(\rho_{0}\mid\pi)\leq\liminf_{\tau_{n}\rightarrow0}\mathcal{H}(\rho_{\tau_{n}}\mid\pi).
\]
This implies that for almost all $\tau\in[0,T]$,
\[
\mathcal{H}(\rho_{0}\mid\pi)-\mathcal{H}(\rho_{\tau}\mid\pi)\leq-\int_{0}^{\tau}\iint_{G}\bar{\nabla}\log\rho_{t}d\mathbf{v}_{t}dt.
\]
Indeed, we have that as $\tau_{n}\rightarrow0$, 
\[
\int_{\tau_{n}}^{\tau}\iint_{G}\bar{\nabla}\log\rho_{t}d\mathbf{v}_{t}dt\rightarrow\int_{0}^{\tau}\iint_{G}\bar{\nabla}\log\rho_{t}d\mathbf{v}_{t}dt,
\]
since 
\[
\left|\int_{0}^{\tau_{n}}\iint_{G}\bar{\nabla}\log\rho_{t}d\mathbf{v}_{t}dt\right|\leq\int_{0}^{\tau_{n}}\mathcal{I}(\rho_{t}\mid\pi)|\rho_{t}^{\prime}|dt\rightarrow0.
\]
\end{proof}

\end{document}